 \documentclass [11pt]{article}

 \usepackage [english]{babel}
 \usepackage [utf8]{inputenc}
 \usepackage [pdftex]{graphicx}
 \usepackage{multicol}
 \usepackage{amsfonts}
 \usepackage{array}
 \usepackage{amsmath}
 \usepackage{amsthm}
 \usepackage{amssymb}
 \usepackage{multirow}
 \usepackage{multicol}
 \usepackage{color}
\usepackage[dvipsnames]{xcolor}
\usepackage[normalem]{ulem}
\usepackage[final]{pdfpages}
\usepackage{pdfpages}
\usepackage{caption,graphicx,newfloat}
\DeclareCaptionType{Graphics}
\usepackage{booktabs}
\usepackage{float}
\usepackage{lscape}
\usepackage{soul}
\usepackage{xspace}
\usepackage{hyperref}
\usepackage{lscape}
\newcommand{\footremember}[2]{%
    \footnote{#2}
    \newcounter{#1}
    \setcounter{#1}{\value{footnote}}%
}
\newcommand{\footrecall}[1]{%
    \footnotemark[\value{#1}]%
}

\allowdisplaybreaks

\newcommand{\define}{\mathrel{{\mathop:}{=}}}

\newtheorem{theorem}{Theorem}

\newtheorem{proposition}{Proposition}

\newcommand\blfootnote[1]{%
  \begingroup
  \renewcommand\thefootnote{}\footnote{#1}%
  \addtocounter{footnote}{0}%
  \endgroup
}

\definecolor{Darkgreen}{RGB}{0,160,0}
\definecolor{byzantine}{rgb}{0.74, 0.2, 0.64}

\title{Portfolio problems with two levels decision-makers: Optimal portfolio selection with pricing decisions on transaction costs. Extended version and complete risk profiles analysis }

\author{%
  Marina Leal  \blfootnote{This research has been funded by the Spanish Ministry of Science and Technology project MTM2016-74983-C2-1-R.}\footremember{imus}{Institute of Mathematics of the University of Seville (IMUS), Spain} \footremember{dto}{Department of Statistics and Operations Research, Faculty of Mathematics, University of Seville, Spain}%
  \and Diego Ponce \footrecall{imus} \footnote{CIRRELT and Department of Mechanical, Industrial and Aerospace Engineering, Faculty of Engineering and Computer Science, Concordia University, Montreal, Canada.}%
  \and Justo Puerto \footrecall{imus} \footrecall{dto}
  }

\date{}

\begin{document}
\maketitle

%
%
%
%
%
%
%
%

\textbf{Abstract}. {This paper presents novel bilevel leader-follower portfolio selection problems in which the financial intermediary becomes a decision-maker. This financial intermediary decides on the unit transaction costs for investing in some securities, maximizing its benefits, and the investor chooses his optimal portfolio, minimizing risk and ensuring a given expected return. Hence,} transaction costs become decision variables in the portfolio problem, {and two levels of decision-makers are incorporated}: the financial intermediary and the investor. These situations give rise to general Nonlinear Programming formulations in both levels of the decision process. We present different bilevel versions of the problem: financial intermediary-leader, investor-leader, and social welfare; besides, their properties are analyzed.  Moreover, we develop Mixed Integer Linear Programming formulations for some of the proposed {problems} and effective algorithms for some others. Finally, {we report on some computational experiments performed on data taken from the Dow Jones Industrial Average, and analyze and compare the results obtained by the different models.}

 Conditional Value at Risk measure (CVaR).
%
%

\section{Introduction}

The classical model in portfolio optimization was originally proposed by Markowitz in 1952 \cite{mar52}. This model has served as the initial point for the development of modern portfolio theory. Over time, portfolio optimization problems have become more realistic, incorporating real-life aspects that make the resulting portfolios more cost-effective than the alternatives that do not consider them \cite{cas11, kol14, Lynch11,man14,man15}. Transaction costs can be seen as one of these important actual features to be included in portfolio optimization. These costs are those incurred by the investors when buying and selling assets on financial markets, charged by {the brokers, the financial institutions or the market makers} playing the role of intermediary. Transaction costs usually include banks and brokers commissions, fees, etc. These commissions or fees have a direct impact on the portfolio, especially for individual or small investors, since they will determine the net returns, reducing them and decreasing also the budget available for future investments \cite{bau,bau10, Liu02}.\par

To the best of our knowledge, in the existing literature,  transaction costs are assumed to be given {\cite{dav90, kor98, lob07, mag76, man14,man15, mor95}}. They can be a fixed cost applied to each selected security in the portfolio; or a variable {cost} to be paid which depends on the amount invested in each security included in the portfolio (see e.g. \cite{bau, bau10,kel00,man14, man15, val14,woo13} and the references therein). This dependence can be proportional {to the investment or} given by a fixed cost that is only charged if the amount invested exceeds a given threshold, or by some other functional form (see e.g. \cite{bau10, kon05, le09, man14,man15}  and the references therein). But in any case, unit transaction costs are known and predetermined in the optimization process. Nevertheless, it is meaningful to analyze the situations where transaction costs {can be decision variables set by financial institutions so that} they are trying to maximize its profit as part of the decision process that leads to optimal portfolios for the investors.\par

The portfolio optimization problem considered in this paper is based on a single-period model of investment and incorporates a {transaction costs setting phase.}  We assume that there are two decision-makers involved in the situation: {on the one hand, the investor and on the other hand, the broker specialist, market maker or financial institution ({that we will call from now on, for simplicity broker-dealer})}. At the beginning of a period, an investor allocates his capital among various assets and during the investment period, each asset generates a random rate of return. Moreover, we consider that the broker-dealer can charge some {unit} transaction costs on the securities selected by the investor trying to maximize its benefits {but anticipating the rational response of the investor}. This is a pricing phase in which the {broker-dealer} decides on how much is going to charge to {the investor for} the traded securities. Considering {unit} transaction costs as a decision variable of the model is a novel element in portfolio optimization and this is one of the main contributions of this paper. Then, at the end of the period, the result for the investor is a variation of his capital (increased or decreased) which is measured by the weighted average of the individual rates of return minus commissions or fees. {In addition}, the result for the {broker-dealer} is the amount paid by the investor, which depends on the {costs} set on the traded securities included in the portfolio chosen by the investor. \par

Based on the structure of financial markets, we assume a hierarchical relationship between the parties involved in the portfolio problem, that is, we {define a natural problem} in which the {broker-dealer} sets the {unit transaction costs} first, trying to anticipate the rational response of the investor. This hierarchical analysis of the portfolio problem has not been addressed before and it is another contribution of our paper. Once the {costs} are fixed, the investor chooses his optimal portfolio. For the sake of completeness, we also analyze the case in which the investor chooses his portfolio first, and after that, the {broker-dealer} sets the transaction costs. In order to model these hierarchical structures, we use a bilevel optimization approach (see e.g. \cite{bar13,col05, lab16, sin17}). Furthermore, we consider a social welfare {problem} {where} both, {broker-dealer} and investor, cooperate to maximize their returns. We assume in the different {problems} that all economic or financial information is common knowledge and that all the decision-makers in the problem have access to it.

The contributions of this paper can be summarized as follows: 1) it incorporates for the first time the above hierarchical approaches with two-levels of decision-makers on portfolio optimization problems (the {broker-dealer} sets {unit} transaction costs trying to maximize its benefits, whereas the investor minimizes risk while ensuring a given expected return \cite{ben03,ben14}); 2) it introduces transaction costs as decision variables controlled by the broker-dealer; and 3) it develops different bilevel programming formulations to obtain optimal solutions for the considered {problems}. This paper introduces new models for the bilevel portfolio optimization  problem.  As far as we know, bilevel models for the portfolio selection that {set unit transaction costs as decision variables of the problem} have not been considered in the literature {before}.\par

The rest of the paper is organized as follows. Section \ref{s:prelim} states the preliminaries and the notation used throughout the paper. In Section \ref{Sect:Bank-leader}, we present the {problem} in which the {broker-dealer} is the leader and we develop two different Mixed Integer Linear Programming (MILP) formulations to solve such problem. Section \ref{Investor-leader} introduces the investor-leader {problem} and develops a Linear Programming (LP) formulation for it. In the more general case where additional constraints are required on the portfolio selection, it is presented a convergent iterative algorithm based on an ``ad hoc'' decomposition of the model. Next, in Section \ref{Cooperative}, it is addressed a social welfare {problem} . There, we propose a MILP formulation and an algorithm based on Benders decomposition for solving {it}. Section \ref{Computational} is devoted to reporting on the computational study of the different {problems and solution methods} discussed in the previous sections. Our results are based on data taken from Dow Jones Industrial Average. Finally, Section \ref{sec:Conclusions} concludes the paper.

\section{Preliminaries\label{s:prelim}}

Let $N=\{1,...,n\}$ be the set of securities considered for an investment, $B\subseteq N$ a subset of securities in which the broker-dealer can charge {unit} transaction costs to the investor {and $R \define N \setminus \{B\}$}. In most cases, $B=N$, but there is no loss of generality to consider that $B$ is a proper subset of $N$. \par
{First,} we assume that the {broker-dealer} can  {price} security $j\in B$ from a discrete set, with cardinality $s_j$, of admissible {costs}, $\mathbb{P}_j=\{c_{j1},...,c_{js_j}\}$, and the {broker-dealer}'s goal is to maximize its benefit. Further, we consider {proportional transaction costs:} the {cost} charged by the {broker-dealer} per security is proportional to the amount invested in such security. {Hence, the {broker-dealer}'s decision variables are unit transaction costs (commissions, fees, ...) to be charged (proportionally) to the securities. }

Let $x=(x_j)_{j=1,...,n}$ denote a vector of decision variables: $x_j$ being the weight of security $j$ in the portfolio. We only assume that the invested capital can not exceed the available budget {and non-negativity}, i.e.,
\[x: \sum_{j=1}^n x_j\leq 1, \ x_j\geq 0, \quad \text{ for } j=1,...,n.\]
This budget constraint is the minimum requirement on the structure of the portfolios. Nevertheless and without loss of generality, we could have assumed that some other linear constraints are imposed on the structure of the requested portfolio $x$.  All the results in this paper can be easily extended to more general situations that consider  polyhedral sets of constraints defining the admissible set of portfolios. \par

Let us denote by $p_j$ the {unit transaction cost} chosen by the {broker-dealer} to {charge} security $j$, $j\in B$. Then, for a given portfolio $x$ (fixed),  the problem faced by the {broker-dealer} can be modeled using the following set of binary decision variables: $a_{jk}=1$ if {cost} $c_{jk}$ is assigned to $p_j$, this is, if $p_j=c_{jk}$; $a_{jk}=0$ otherwise. Thus, to maximize his profit the {broker-dealer} solves the  following problem:\par
\begin{align}
\label{for:bank_of} \tag{\textbf{PricP}} \max &\displaystyle \sum_{j\in B} p_jx_j\\
\label{for:bank_c1} \ \ \ \ \ \hbox{s.t.} \  & \displaystyle p_j= \sum_{k=1}^{s_j}c_{jk}a_{jk}, \quad  j\in B,\\
\label{for:bank_c2}&\displaystyle \sum_{k=1}^{s_j}a_{jk}=1, \hspace*{1cm}  j \in B,\\
\label{for:bank_fr}&\displaystyle a_{jk}\in \{0,1\}, \hspace*{1cm} j\in B, k=1,...,s_j.
\end{align}

If no further constraints are imposed on {costs} the above is a valid formulation. However, in general, we will assume without loss of generality that the set of {costs} for the {broker-dealer} can be restricted to belong to some polyhedron $\mathbb{P}$, allowing $\mathbb{P}=\mathbb{R}^{|B|}_+$. This can be easily included in the above formulation with the following constraint:
\begin{equation}
\label{for:bank_frP} p\in \mathbb{P}.
\end{equation}

We observe that, if $x$ is known, and constraint (\ref{for:bank_frP}) is not included, the above problem is easy to solve (see Proposition \ref{pro:bank-follow}): the {broker-dealer} will set {transaction costs} to the maximum ones among those available for each security. Nevertheless, if the portfolio is unknown (to be decided by the investor) or additional constraints, such as regulation constraints, are imposed into the model, the problem becomes more difficult to be solved, since there exists no explicit expression for an optimal solution.

{Moreover}, we suppose that the investor wants to reduce the risk of his investment while ensuring a given expected return. At this point, several risk measures could be considered, among them variance of returns, Mean Absolute Deviation (MAD), Conditional Value at Risk (CVaR), Gini's Mean Difference, etcetera. (Here, we refer the reader to \cite{man03} for further details on the topic.) In this paper, we have focused on a  portfolio optimization problem based on the CVaR risk measure. This risk measure aims to avoid large losses: for a specific probability level $\alpha$, the CVaR measures the conditional expectation of the  smallest returns with a cumulative probability $\alpha$, that is, the average return of the given size (quantile) of worst realizations \cite{man03, prt17, roc00}. Therefore, we assume that the investor's goals are to maximize its CVaR and, at the same time, to ensure that a minimum expected reward $\mu_0$ is obtained with his portfolio.

{There exist in the literature different ways of accounting for the transaction costs into the portfolio model \cite{man15, man15_2}. For instance,  including them in the objective function \cite{ang12, oli18, woo13}, subtracting them from the expected return \cite{kre11,man05}, reducing the capital available for the investment \cite{woo13}, etcetera. We assume in our approach that transaction costs are directly removed from the expected return.\par}

In order to model the above situation, we consider that the rate of return of each security $j\in N$ is represented by a random variable $R_j$ with a given mean $\mu_j=E(R_j)$. Each portfolio $x$ defines a random variable $R_x=\sum_{j=1}^nR_jx_j$ that represents the portfolio rate of return (its expected value can be computed as $\mu(x)=\sum_{j=1}^n\mu_jx_j$). We consider $T$ scenarios, each of them with probability $\pi_t$, $t = 1,..., T$, and assume that for each random variable $R_j$ its realization, $r_{jt}$, under the scenario $t$ is known. Thus, once the {broker-dealer} has set {the transaction costs}, $p$, the realization of the portfolio rate of return $R_x$ under scenario $t$ is given as $y_t=\sum_{j=1}^nr_{jt}x_j-\sum_{i\in B}p_ix_i$.  \par

With this information, we assume that the investor wants to maximize the CVaR$_\alpha$, namely the conditional expectation of the smallest returns with cumulative probability $\alpha$,  while ensuring a minimum expected return $\mu_0$. Thus, the portfolio optimization {problem}  that the investor wants to solve can be formulated as:
{
\begin{align}
\label{for:CVaR_of} \tag{\textbf{CVaRP}}\max & \displaystyle \  \eta - \dfrac{1}{\alpha}\sum_{t=1}^T\pi_td_t\\
\label{for:CVaR_cy} \ \ \ \ \ \hbox{s.t.} \   &y_t=\sum_{j=1}^nr_{jt}x_j-\sum_{i\in B}p_ix_i, \quad t=1,...,T, \\
\label{for:CVaR_creturn}&\sum_{t=1}^T\pi_ty_t\ge \mu_0,\\
\label{for:CVaR_cd} &\displaystyle d_t\ge \eta - y_t, \hspace*{0.9cm}  t=1,...,T, \\
\label{for:CVaR_frd}& d_{t}\ge 0, \hspace*{1.8cm} t=1,...,T,\\
\label{for:CVaR_cx}&\displaystyle \sum_{j=1}^n x_{j}\leq 1,  \\
\label{for:CVaR_frx}& x_{j}\ge 0, \hspace*{1.7cm} j=1,...,n,
\end{align}

{Observe that $\eta$ is a continuous variable that models the $\alpha$ \emph{Value at Risk}, $VaR_\alpha$, namely the value of the minimum threshold for which the probability of the scenarios with a return less than or equal to $\eta$ is at least $\alpha$. }

Next, (\ref{for:CVaR_cy}) and (\ref{for:CVaR_creturn}) are the scenario constraints. Constraint (\ref{for:CVaR_cy}) gives the expected return in each scenario. Note that, the expected return in each scenario is for the net rate of returns, $\sum_{j=1}^nr_{jt}x_j$,  minus  the transaction rates $\sum_{i\in B}p_ix_i$. Whereas constraint (\ref{for:CVaR_creturn}) ensures an expected return of, at least, $\mu_0$. The objective function and the set of constraints (\ref{for:CVaR_cd}) and (\ref{for:CVaR_frd}) model the CVaR (see Mansini et. al \cite{man03} for details). And finally, the sets of constraints (\ref{for:CVaR_cx}) and (\ref{for:CVaR_frx}) force $x$ to define a portfolio.}

We note also that by choosing different values for the parameters $\alpha$ and $\mu_0$, in the formulation above, different types of investors (i.e., different level of attitude towards risk) can be {incorporated in the model}.

\section{Bilevel {broker-dealer}-leader Investor-follower Portfolio Problem (BLIFP)} \label{Sect:Bank-leader}
We start analyzing a hierarchical structure in the financial markets in which the {broker-dealer} sets the transaction costs first, and after that, the investor chooses his portfolio. Observe that in this situation, the problem faced from the investor point of view reduces to a portfolio selection, under the considered criterion, which in this case is to hedge against risk maximizing the average $\alpha$-quantile of his smallest returns (CVaR$_\alpha$). Therefore, we study this situation from the point of view of both the financial intermediary and the investor, simultaneously, which is a novel perspective.\par

We model the situation as a bilevel leader-follower problem in which the {broker-dealer} has to fix the {transaction costs}, from the polyhedral set $\mathbb{P}\in \mathbb{R}^{|B|}$, maximizing his benefits by assuming that, after his decision is made, the investor will choose his optimal portfolio. \par
Using the bilevel optimization framework, the \textbf{BLIFP} can be modeled as follows:

{
\begin{align*}
\label{for:BLIFP} \tag{\textbf{BLIFP0}} \max &\displaystyle \sum_{j\in B} p_jx_j\\
\nonumber \ \ \ \ \hbox{s.t. }
&(\ref{for:bank_c1}), (\ref{for:bank_c2}), (\ref{for:bank_fr}),(\ref{for:bank_frP}), \tag{\text{Bank Constraints}}\\
&x\in arg \max \quad \displaystyle \eta - \dfrac{1}{\alpha}\sum_{t=1}^T\pi_td_t,\\
&\ \ \ \ \ \ \ \ \nonumber \ \ \ \ \hbox{s.t.} \quad (\ref{for:CVaR_cy}), (\ref{for:CVaR_creturn}), (\ref{for:CVaR_cd}), (\ref{for:CVaR_frd}), (\ref{for:CVaR_cx}),   (\ref{for:CVaR_frx}) \tag{\text{Investor Constraints}}.
\end{align*}
}

\bigskip
Our goal is to solve the above {bilevel programming model} to provide answers to the new portfolio optimization {problem}. We propose two different MILP formulations with the aim of making a  computational comparison to check which one is more effective.
\par

\subsection{Formulation BLIFP1}
The main difficulty in handling \ref{for:BLIFP} is that some of its decision variables are constrained to be optimal solutions of a nested optimization problem. In order to overcome that issue  we observe that the follower problem in \ref{for:BLIFP} is linear on $x$ when $p$ is given. This allows us to {easily} compute its exact dual as:

\begin{align}
\label{for:dual1}\tag{\textbf{Dual1}}\min \; & \displaystyle \beta +\mu_0\mu\\
\label{for:dual1_c1} \ \ \ \ \ \hbox{s.t.} \  &\displaystyle\beta-\sum_{t=1}^T(r_{jt}-p_j)\delta_t\ge 0, \quad  j\in B,\\
\label{for:dual1_c2}&\displaystyle\beta-\sum_{t=1}^Tr_{jt}\delta_t\ge 0, \hspace*{1.5cm}  j\in R,\\
\label{for:dual1_c3}&\displaystyle -\sum_{t=1}^T \gamma_{t}=1,  \\
\label{for:dual1_c4}&\displaystyle \gamma_t\ge -\dfrac{\pi_t}{\alpha}, \hspace*{2.7cm} t=1,...,T, \\
\label{for:dual1_c5}&\displaystyle\gamma_t+\delta_t+\pi_t\mu =0, \hspace*{1.35cm} t=1,...,T, \\
\label{for:dual1_fr1}& \gamma_{t}\le 0, \hspace*{3.1cm} t=1,...,T,\\
\label{for:dual1_fr2}&\mu \le 0, \beta \geq 0.
\end{align}

{We note in passing that variables $\delta_t$, $\mu$, $\gamma_t$ and $\beta$, are the dual variables associated to constraints (\ref{for:CVaR_cy}), (\ref{for:CVaR_creturn}), (\ref{for:CVaR_cd}) and (\ref{for:CVaR_cx}), respectively. Therefore, they can be interpreted as multipliers explaining the marginal variation of the objective function values as a function of the corresponding constraints' right-hand-sides. Nevertheless,  we do not go into details in the economic insights of this dual model, since our use is instrumental to obtain a single level reformulation of the hierarchical model.}

Then, \ref{for:BLIFP} can be reformulated, applying the Strong Duality Theorem, including the constraints of the primal and dual problem together with the equation that matches the objective values of the follower primal and dual problems. Thus, \ref{for:BLIFP} is equivalent to solving this new mathematical programming model:

\begin{align}
\nonumber \max &\displaystyle \sum_{j\in B} p_jx_j\\
\nonumber \ \ \ \ \ \hbox{s.t.} \   &(\ref{for:bank_c1}), (\ref{for:bank_c2}), (\ref{for:bank_fr}),(\ref{for:bank_frP}),  \tag{\text{Bank Constraints}}\\
\label{for:strongduality_1}& \displaystyle \eta - \dfrac{1}{\alpha}\sum_{t=1}^T\pi_td_t=\beta +\mu_0\mu,\\
\nonumber & {(\ref{for:CVaR_cy}), (\ref{for:CVaR_creturn}), (\ref{for:CVaR_cd}), (\ref{for:CVaR_frd}), (\ref{for:CVaR_cx}),   (\ref{for:CVaR_frx}), \tag{\text{Investor Constraints}}}\\
\nonumber & (\ref{for:dual1_c1}), (\ref{for:dual1_c2}), (\ref{for:dual1_c3}), (\ref{for:dual1_c4}), (\ref{for:dual1_c5}), (\ref{for:dual1_fr1}), (\ref{for:dual1_fr2} \tag{\text{Dual Constraints}}).
\end{align}

We can observe that in the above formulation we have some bilinear terms, $p_jx_j$ and $p_j\delta_t$ that appear in the leader objective function and constraints (\ref{for:CVaR_cy}) and (\ref{for:dual1_c1}). In order to solve the problem using off-the-shelf solvers, they can be linearized `a la' McKormick, {\cite{mcc76}}, giving rise to another exact MILP formulation for the bilevel problem.\par

Indeed, since $p_j= \sum_{k=1}^{s_j}c_{jk}a_{jk},\;  \forall j\in B$, we could substitute the terms $p_jx_j=\sum_{k=1}^{s_j}c_{jk}\hat{a}_{jk}$ adding variables $\hat{a}_{jk},\; \forall j\in B, k=1,...,s_j,$ and the following set of constraints:

\begin{equation}\label{for:linear_a_NO}
\begin{array}{lll}
&\displaystyle \hat{a}_{jk}\leq x_j, &  j\in B, k=1,...,s_j,\\
&\displaystyle \hat{a}_{jk}\leq a_{jk}, &  j\in B, k=1,...,s_j,\\
&\displaystyle \hat{a}_{jk}\geq x_j-(1-a_{jk}), &  j\in B, k=1,...,s_j,\\
&\displaystyle \hat{a}_{jk}\geq 0, &  j\in B, k=1,...,s_j.
\end{array}
\end{equation}

Furthermore, this linearization can be simplified. Observe that it is sufficient to include in (\ref{for:BLIFP}) variables $\hat a_{jk}$ and constraints

\begin{equation}\label{for:linear_a}
\begin{array}{lll}
&\displaystyle \hat{a}_{jk}\leq a_{jk},&  j\in B, k=1,...,s_j,\\
&\displaystyle \hat{a}_{jk}\geq 0, &  j\in B, k=1,...,s_j,\\
\end{array}
\end{equation}

\noindent from (\ref{for:linear_a_NO}) and to substitute the variables $x_j=\sum_{k=1}^{s_j}\hat{a}_{jk}, \forall j\in B$. We obtain in this manner an equivalent, {smaller} formulation with the bilinear terms $a_{jk}x_j$ linearized for all $j\in B, k=1,...,s_j$, but with less constraints and decision variables.

Following a similar argument we can linearize the products $p_j\delta_t=\sum_{k=1}^{s_j}c_{jk}a_{jk}\delta_t$. To do that, take $M$ a sufficiently large positive number and define the new variables $\hat{\delta}_{jkt}=a_{jk}\delta_t,$ $\forall j \in B, k=1,...,s_j, t=1,...,T$. This set of variables together with the following family of constraints linearize all the bilinear terms:

\begin{equation}\label{for:linear_delta+}
\begin{array}{lll}

&\displaystyle \hat{\delta}_{jkt}\leq \delta_t, &  j\in B, k=1,...,s_j, t=1,...,T,\\
&\displaystyle \hat{\delta}_{jkt}\leq Ma_{jk}, &  j\in B, k=1,...,s_j, t=1,...,T,\\
&\displaystyle \hat{\delta}_{jkt}\geq \delta_t-(1-a_{jk})M, &  j\in B, k=1,...,s_j, t=1,...,T,\\
&\displaystyle \hat{\delta}_{jkt}\geq 0, &  j\in B, k=1,...,s_j, t=1,...,T.\\
\end{array}
\end{equation}

Combining the above elements, all together, we obtain a valid MILP formulation for \textbf{BLIFP}:

{
\begin{align}
 \label{for:BLIFP1} \tag{\textbf{BLIFP1}} \max &\displaystyle \sum_{j\in B} \sum_{k=1}^{s_j}c_{jk}\hat{a}_{jk}\\
\nonumber \tag{\ref{for:bank_c2}}\ \ \ \ \ \hbox{s.t.} \   & \displaystyle \sum_{k=1}^{s_j}a_{jk}=1, \hspace*{2.05cm}  j \in B,\\
\nonumber \tag{\ref{for:bank_fr}}&\displaystyle a_{jk}\in \{0,1\}, \hspace*{2cm}  j\in B, k=1,...,s_j,\\
\nonumber \tag{\ref{for:strongduality_1}}& \displaystyle \eta - \dfrac{1}{\alpha}\sum_{t=1}^T\pi_td_t=\beta +\mu_0\mu\\
\label{for:CVaR_cy_linear}&y_t=\sum_{j\in B}r_{jt}\left( \sum_{k=1}^{s_j} \hat a_{jk}\right) +\sum_{j \in R}r_{jt}x_j-\sum_{j \in B}\sum_{k=1}^{s_j}c_{jk}\hat a_{jk}, \hspace*{0.3cm} t=1,...,T, \\
\nonumber \tag{\ref{for:CVaR_creturn}}&\sum_{t=1}^T\pi_ty_t\ge \mu_0,\\
\nonumber \tag{\ref{for:CVaR_cd}}&\displaystyle d_t\ge \eta - y_t, \hspace*{2.2cm}  t=1,...,T, \\
\nonumber \tag{\ref{for:CVaR_frd}}& d_{t}\ge 0, \hspace*{3.5cm}  t=1,...,T,\\
\label{for:CVaR_cx_a}&\displaystyle \sum_{j \in B}\sum_{k=1}^{s_j}\hat a_{jk}+\sum_{j \in R}x_j \leq 1,  \\
\nonumber& \tag{\ref{for:CVaR_frx}}x_{j}\ge 0, \hspace*{3.45cm} j \in R,\\
\nonumber & \begin{array}{l}
\displaystyle \hat{a}_{jk}\leq a_{jk}, \hspace*{2.75cm}   j\in B, k=1,...,s_j,\\
\displaystyle \hat{a}_{jk}\geq 0, \hspace*{3.05cm}   j\in B, k=1,...,s_j,
\end{array}  \tag{\ref{for:linear_a}}\\
\label{for:dual1_c1_linear}  &\displaystyle\beta-\sum_{t=1}^T\left(r_{jt}\delta_t-\sum_{k=1}^{s_j}c_{jk}\hat{\delta}_{jkt} \right)\ge 0, \hspace*{0.3cm}  j\in B,\\
\nonumber \tag{\ref{for:dual1_c2}}&\displaystyle\beta-\sum_{t=1}^Tr_{jt}\delta_t\ge 0, \hspace*{1.7cm}  j\in R,\\
\nonumber \tag{\ref{for:dual1_c3}}&\displaystyle -\sum_{t=1}^T \gamma_{t}=1,  \\
\nonumber \tag{\ref{for:dual1_c4}}&\displaystyle \gamma_t\ge -\dfrac{\pi_t}{\alpha}, \hspace*{2.9cm} t=1,...,T, \\
\nonumber \tag{\ref{for:dual1_c5}}&\displaystyle\gamma_t+\delta_t+\pi_t\mu =0, \hspace*{1.6cm} t=1,...,T, \\
\nonumber \tag{\ref{for:dual1_fr1}}& \gamma_{t}\le 0, \hspace*{3.4cm} t=1,...,T,\\
\nonumber \tag{\ref{for:dual1_fr2}}&\mu \le 0, \beta \geq 0,\\
\nonumber &\begin{array}{l}
\displaystyle \hat{\delta}_{jkt}\leq \delta_t, \hspace*{2.95cm}  j\in B, k=1,...,s_j, t=1,...,T,\\
\displaystyle \hat{\delta}_{jkt}\leq Ma_{jk}, \hspace*{2.25cm}  j\in B, k=1,...,s_j, t=1,...,T,\\
\displaystyle \hat{\delta}_{jkt}\geq \delta_t-(1-a_{jk})M, \hspace*{0.5cm}  j\in B, k=1,...,s_j, t=1,...,T,\\
\displaystyle \hat{\delta}_{jkt}\geq 0, \hspace*{2.9cm}  j\in B, k=1,...,s_j, t=1,...,T.
\end{array} \tag{\ref{for:linear_delta+}}
\end{align}
}

The above long formulation can be easily understood once the different sets of constraints are grouped by meaningful blocks. We observe that (\ref{for:bank_c2}), (\ref{for:bank_fr}) and (\ref{for:bank_frP}) are the constraints that define the feasible domain of the {broker-dealer} problem. Constraint (\ref{for:strongduality_1}) imposes the strong duality condition among the primal and dual formulation of the follower problem. Next, { (\ref{for:CVaR_cy_linear}), (\ref{for:CVaR_creturn}), (\ref{for:CVaR_cd}), (\ref{for:CVaR_frd}), (\ref{for:CVaR_cx_a}),   (\ref{for:CVaR_frx}) and (\ref{for:linear_a}) are the constraints that correctly define the linearized version of the investor subproblem}.  Finally, the constraints that come from the linearized version of the dual of the follower problem are (\ref{for:dual1_c1_linear}),(\ref{for:dual1_c2}), (\ref{for:dual1_c3}), (\ref{for:dual1_c4}), (\ref{for:dual1_c5}), (\ref{for:dual1_fr1}), (\ref{for:dual1_fr2}) and(\ref{for:linear_delta+}).

Using these blocks of constraints \ref{for:BLIFP1} can be written in the following compact form.
\begin{align}
 \label{for:BLIFP1} \tag{\textbf{BLIFP1}} \max &\displaystyle \sum_{j\in B} \sum_{k=1}^{s_j}c_{jk}\hat{a}_{jk}\\
\nonumber \ \ \ \ \ \hbox{s.t.} \   & (\ref{for:bank_c2}), (\ref{for:bank_fr}),(\ref{for:bank_frP}),  \tag{\text{Linear {broker-dealer} Constraints}}\\
\nonumber & (\ref{for:strongduality_1}), \tag{\text{Strong Duality Constraint}}\\
\nonumber & { (\ref{for:CVaR_cy_linear}), (\ref{for:CVaR_creturn}), (\ref{for:CVaR_cd}), (\ref{for:CVaR_frd}), (\ref{for:CVaR_cx_a}),   (\ref{for:CVaR_frx}), (\ref{for:linear_a}), \tag{\text{Linear investor Constraints 1}}}\\
 &\begin{array}{l}
(\ref{for:dual1_c1_linear}), (\ref{for:dual1_c2}), (\ref{for:dual1_c3}),
(\ref{for:dual1_c4}), (\ref{for:dual1_c5}), (\ref{for:dual1_fr1}),\\
(\ref{for:dual1_fr2}),(\ref{for:linear_delta+}).
\end{array} \tag{\text{Linear Dual Constraints}}
\end{align}

This valid formulation of \ref{for:BLIFP1} requires to set a valid value for the big-$M$ constraint. Setting an appropriate value is important to improve the performance of the resulting MIP. In the following, we prove the existence of a valid upper bound for such a value.

\begin{proposition} \label{pro:boundM1}
Let ${\cal B}(p) $ be the set of all full rank submatrices of the matrix representing the constraints of problem \ref{for:dual1} in standard form, where $p$ is a fixed set of {cost values}, and let $\cal{B}^S$(p) be the set of all matrices that result from ${\cal B}(p)$ replacing, one each time,  their columns by the RHS of that problem.  Moreover, let $\Delta(p):=\min \{ |det(B)|: B\in {\cal B}(p)\}$ and $\Delta^S(p)=\max \{|det(B)|: B\in {\cal B}^S(p)\}$.

Then $\displaystyle  UB_{\delta}:=\max_{p} \Delta^S(p)/ \Delta(p)$ is a valid upper bound for the big-$M$ constant in \ref{for:BLIFP1}.

\end{proposition}

\begin{proof}

It is easy to observe that for each fixed set of {costs} $p$, $M\leq \max_{t=1,...,T}\delta_t$. Therefore the proof reduces to bound the terms $\delta_t$.\par


From constraint (\ref{for:dual1_c5}) in formulation \ref{for:dual1} we know that $\delta_t=-\gamma_t-\pi_t\mu , \quad\forall t=1,...,T,$ which implies that $\delta_t\ge 0$ for all $t=1,...,T$, since $\mu\leq 0$, and $\delta_t\leq 0$,  and $\pi_t \geq 0$ for all $t=1,...,T$.

We observe that $\beta+\mu_0 \mu$ is bounded for any $\mu_0$ and for any set of {costs} $p$ (recall that this o.f. gives a CVaR). If we denote by $r_{max}=\displaystyle \max_{j=1,...,n, t=1,...,T}r_{jt}$, $r_{min}=\displaystyle \min_{j=1,...,n, t=1,...,T}r_{jt}$ {and $c_{max}=\displaystyle \max_{j=1,...,n,\ k=1,...,s_j}c_{jk}$,  then $r_{min}-c_{max}\le \beta+\mu_0 \mu\le r_{max}$}. This implies that the solution of \ref{for:dual1} is attained at an extreme point and therefore no rays have to be considered.
Next, the extreme points of the feasible regions are solutions of systems of full dimensional equations taken from the constraint matrix of \ref{for:dual1} in standard form. Therefore, applying Cramer's rule we obtain that, at the extreme points, the values of any variable $\delta_t$ for all $t=1,\ldots,T$ satisfy: $\delta_t\le \Delta^S(p)/\Delta(p)$. Next, letting $p$ vary on the finite set of possible {costs} we obtain that $\displaystyle \delta_t\le \max_{p} \Delta^S(p)/\Delta(p)$.
\end{proof}

\bigskip

This bound is only of theoretical interest and in our computational experiments, we have set it empirically to be more accurate.

\subsection{Formulation BLIFP2\label{ss:BLIFP2}}
In this section, we derive an alternative formulation for \textbf{BLIFP} based on the representation of the {costs} as $p_jx_j=\sum_{k=1}^{s_j}c_{jk}\hat{a}_{jk}$ in the follower problem before its dual problem is obtained. This artifact produces an alternative single level model that we will analyze in the following.

Let us consider the CVaR problem in \ref{for:BLIFP}, and let us linearize the products of variables $p_ix_i$, as in the previous formulation. This way we obtain:

{
\begin{align*}
\max \  &\displaystyle \eta - \dfrac{1}{\alpha}\sum_{t=1}^T\pi_td_t\\
\nonumber \ \ \ \ \ \hbox{s.t.} \  \tag{\ref{for:CVaR_cy_linear}}&y_t=\sum_{j\in B}r_{jt}\left( \sum_{k=1}^{s_j} \hat a_{jk}\right) +\sum_{j \in R}r_{jt}x_j-\sum_{j \in B}\sum_{k=1}^{s_j}c_{jk}\hat a_{jk}, \hspace*{0.3cm} t=1,...,T, \\
\nonumber&\sum_{t=1}^T\pi_ty_t\ge \mu_0, \tag{\ref{for:CVaR_creturn}} \\
\nonumber & \displaystyle d_t\ge \eta - y_t, \hspace*{0.75cm}  t=1,...,T,  \tag{\ref{for:CVaR_cd}} \\
\nonumber& d_{t}\ge 0, \hspace*{1.7cm}  t=1,...,T, \tag{\ref{for:CVaR_frd}} \\
\tag{\ref{for:CVaR_cx_a}}&\displaystyle \sum_{j \in B}\sum_{k=1}^{s_j}\hat a_{jk}+\sum_{j \in R}x_j \leq 1,  \\
\nonumber& x_{j}\ge 0, \hspace*{1.6cm}  j=1,...,n,\tag{\ref{for:CVaR_frx}} \\
\nonumber& \begin{array}{l}
\displaystyle \hat{a}_{jk}\leq a_{jk}, \hspace*{0.95cm}  j\in B, k=1,...,s_j,\\
\displaystyle \hat{a}_{jk}\geq 0, \hspace*{1.25cm}    j\in B, k=1,...,s_j.
\end{array}  \tag{\ref{for:linear_a}}
\end{align*}
}

Once again, to ease presentation, we write the above formulation in the following compact format.

\begin{align*}
\max \  &\displaystyle \eta - \dfrac{1}{\alpha}\sum_{t=1}^T\pi_td_t\\
\nonumber \ \ \ \ \ \hbox{s.t.} \   &  { (\ref{for:CVaR_cy_linear}), (\ref{for:CVaR_creturn}), (\ref{for:CVaR_cd}), (\ref{for:CVaR_frd}), (\ref{for:CVaR_cx_a}),   (\ref{for:CVaR_frx}), (\ref{for:linear_a}). \tag{\text{Linear investor Constraints 1}}}\\
\end{align*}

Its dual problem is:
\begin{align}
\label{for:dual2_of} \min \ & \displaystyle \beta +\mu_0 \mu +\sum_{j\in B}\sum_{k=1}^{s_j}a_{jk}\sigma_{jk} \tag{\textbf{Dual2}} \\
\nonumber \ \ \ \ \ \hbox{s.t. } &  (\ref{for:dual1_c2}), (\ref{for:dual1_c3}), (\ref{for:dual1_c4}), (\ref{for:dual1_c5}), (\ref{for:dual1_fr1}), (\ref{for:dual1_fr2}), \\
\label{for:dual2_c1} &  \displaystyle \beta-\sum_{t=1}^Tr_{jt}\delta_t+\sum_{t=1}^Tc_{jk}\delta_t+\sigma_{jk}\ge 0, \quad  j\in B, k=1,...,s_j,\\
\label{for:dual2_fr1} & \sigma_{jk}\ge 0, \quad  j\in B, k=1,...,s_j.
\end{align}

Therefore, we can replace in \ref{for:BLIFP} the nested optimization problem on the CVaR including the group of constraints in (Linear investor Constraints 1) and (\ref{for:dual1_c2})-(\ref{for:dual1_fr2}), (\ref{for:dual2_c1}), (\ref{for:dual2_fr1}), that will be referred from now on as (Dual2 Constraints), together with the strong duality condition given by  $$\displaystyle \eta - \dfrac{1}{\alpha}\sum_{t=1}^T\pi_td_t=\beta +\mu_0 \mu +\sum_{j\in B}\sum_{k=1}^{s_j}a_{jk}\sigma_{jk}.$$
The combination of all these elements results in the following  alternative valid formulation for \ref{for:BLIFP}.

\begin{align}
\nonumber\max &\displaystyle \sum_{j\in B} \sum_{k=1}^{s_j}c_{jk}\hat{a}_{jk}\\
\nonumber \ \ \ \ \ \hbox{s.t.} \   &(\ref{for:bank_c2}), (\ref{for:bank_fr}),(\ref{for:bank_frP})  \tag{\text{{Broker-dealer} Constraints}}\\
\label{for:strongduality_2}& \eta - \dfrac{1}{\alpha}\sum_{t=1}^T\pi_td_t=\beta +\mu_0 \mu +\sum_{j\in B}\sum_{k=1}^{s_j}a_{jk}\sigma_{jk}\\
\nonumber & { (\ref{for:CVaR_cy_linear}), (\ref{for:CVaR_creturn}), (\ref{for:CVaR_cd}), (\ref{for:CVaR_frd}), (\ref{for:CVaR_cx_a}),   (\ref{for:CVaR_frx}), (\ref{for:linear_a}), \tag{\text{Linear investor Constraints 1}}}\\
\nonumber & \begin{tabular}{l}
(\ref{for:dual1_c2}),  (\ref{for:dual1_c3}), (\ref{for:dual1_c4}), (\ref{for:dual1_c5}), (\ref{for:dual1_fr1}), (\ref{for:dual1_fr2}),\\
(\ref{for:dual2_c1}), (\ref{for:dual2_fr1}).
\end{tabular} \tag{\text{Dual2 Constraints}}
\end{align}

The formulation above still contains bilinear terms, namely $a_{jk}\sigma_{jk}$,  in constraint (\ref{for:strongduality_2}). Therefore, we linearize them as in \ref{for:BLIFP1} and we obtain another valid MILP formulation for \textbf{BLIFP}.\par

\begin{align}
\label{for:BLIFP2} \tag{\textbf{BLIFP2}}  \max &\displaystyle \sum_{j\in B} \sum_{k=1}^{s_j}c_{jk}\hat{a}_{jk}\\
\nonumber \ \ \ \ \ \hbox{s.t.} \   &(\ref{for:bank_c2}), (\ref{for:bank_fr}),(\ref{for:bank_frP})  \tag{\text{Linear {broker-dealer} Constraints}}\\
\label{for:strongduality_3}& \eta - \dfrac{1}{\alpha}\sum_{t=1}^T\pi_td_t=\beta +\mu_0 \mu +\sum_{j\in B}\sum_{k=1}^{s_j}\hat \sigma_{jk},\\
\nonumber & { (\ref{for:CVaR_cy_linear}), (\ref{for:CVaR_creturn}), (\ref{for:CVaR_cd}), (\ref{for:CVaR_frd}), (\ref{for:CVaR_cx_a}),   (\ref{for:CVaR_frx}), (\ref{for:linear_a}), \tag{\text{Linear investor Constraints 1}}}\\
\nonumber & \begin{tabular}{l}
(\ref{for:dual1_c2}),  (\ref{for:dual1_c3}), (\ref{for:dual1_c4}), (\ref{for:dual1_c5}), (\ref{for:dual1_fr1}), (\ref{for:dual1_fr2}),\\
(\ref{for:dual2_c1}), (\ref{for:dual2_fr1}).
\end{tabular} \tag{\text{Dual2 Constraints}}\\
\label{for:bigM}&\begin{array}{l}
\displaystyle \hat{\sigma}_{jk}\leq \sigma_{jk}, \hspace*{2.8cm}  j\in B, k=1,...,s_j,\\
\displaystyle \hat{\sigma}_{jk} \leq Ma_{jk},\hspace*{2.4cm}  j\in B, k=1,...,s_j,\\
\displaystyle \hat{\sigma}_{jk}\geq \sigma_{jk}-M(1-a_{jk}), \quad  j\in B, k=1,...,s_j\\
\displaystyle \hat{\sigma}_{jk} \geq 0, \hspace*{3.1cm}  j\in B, k=1,...,s_j,
\end{array}
\end{align}

Again, this valid formulation for \ref{for:BLIFP2} requires to prove the existence of a valid upper bound for the big-$M$ constant in (\ref{for:bigM}. In the following, we prove that a valid upper bound for such a value does exist.

\begin{proposition} \label{pro:boundM2}
Let $UB_{\delta}$ be the bound obtained in Proposition \ref{pro:boundM1} and $\displaystyle LB_{\beta}=\min_{p} \Delta^S(p)/ \Delta(p)$. Then $\displaystyle \max\{T(r_{max}-c_{min})UB_{\delta} - LB_{\beta}, 0\}$ is a valid upper bound for $M$ in \ref{for:BLIFP2}.

\end{proposition}

\begin{proof}
It is easy to observe that $M=\displaystyle \max_{j\in B, k=1,\ldots,s_j}\{\sigma_{jk}\}$ is a valid upper bound.

Since $\sigma_{jk}$ is being minimized (it is minimized in \ref{for:dual2_of}) and it must satisfy constraints  (\ref{for:dual2_c1}) and (\ref{for:dual2_fr1}), there always exists, $\forall j\in B, k=1,...,s_j$,  an optimal solution where these variables get the values:

\begin{equation*}
    \sigma_{jk} = \left\{
	       \begin{array}{ll}
		 0,      & \mathrm{if\ } \beta +\sum_{t=1}^T (c_{jk}-r_{jt})\delta_t \ge 0 \\
		-\beta +\sum_{t=1}^T (r_{jt}-c_{jk})\delta_t, & \mathrm{otherwise.}
	       \end{array}
	     \right.
   \end{equation*}

Because $\beta\geq 0$ by definition, if $\beta +\sum_{t=1}^T (c_{jk}-r_{jt})\delta_t$ is negative, then $\sum_{t=1}^T (c_{jk}-r_{jt})\leq 0$ and therefore $\sum_{t=1}^T (r_{jt}-c_{jk})\geq 0$.

Consequently the maximum value of this variable would be $\max \{0, T(r_{max}-c_{min})UB_{\delta}-LB_{\beta}\}$, where $UB_{\delta}$ and $LB_{\beta}$ are found by doing a similar discussion as in the Proposition \ref{pro:boundM1}. 

%
%
%
\end{proof}

{A first comparison of the above two models, namely \ref{for:BLIFP1} and  \ref{for:BLIFP2}, sheds some light on their problem solving difficulty. For the sake of simplicity, we denote   by $d=\sum_{j\in B} |s_j|$ the number of different admissible costs in the models. Table \ref{t:comparamodelos} shows the number of binary and continuous variables and constraints in both models.
\begin{table}[H]
\begin{center}
\begin{tabular}{|c|c|c||c|}
\hline
                 & Binary & Continuous & Constraints \\ \hline
\ref{for:BLIFP1} &  $d$ & $R+5T+d+dT+3$ & $2|B|+6T+2|R|+2d+4dT+6$ \\ \hline
\ref{for:BLIFP2} & $d$ & $R+4T+3d+3$  & $|B|+5T+|R|+7d+5$ \\ \hline
\end{tabular}
\caption{number of variables and constraints in models \ref{for:BLIFP1} and  \ref{for:BLIFP2} \label{t:comparamodelos}}
\end{center}
\end{table}

The smaller dimension of \ref{for:BLIFP2} explains what we observe later in the computational experience: \ref{for:BLIFP2} is  solved more efficiently  than \ref{for:BLIFP1} (see Section \ref{Computational}).}

\section{Bilevel Investor-leader {broker-dealer}-follower Portfolio Problem (\textbf{ILBFP})} \label{Investor-leader}

For the sake of completeness, in this section, we consider the reverse situation to the one that has been analyzed in Section \ref{Sect:Bank-leader}, i.e., a hierarchical structure in the financial market where the investor acts first and once his portfolio $x$ is chosen the {broker-dealer} sets {transaction costs}. Although one could claim that this situation may be atypical in actual financial markets, we want to analyze this case from a theoretical point of view. Moreover, we wish to analyze its implications depending on different {broker-dealer}s and investors profiles. See Section \ref{Computational} for a comparative analysis. This situation leads to a bilevel leader-follower model in which the investor (leader) has to optimize his utility (maximize the CVaR ensuring a given expected reward, $\mu_0$) by assuming that once he has chosen the portfolio, the {broker-dealer} (follower) will maximize his benefits setting the applicable transaction costs.

We can formulate the problem as:

\begin{align}
\label{for:ILBFP} \max  \ &\displaystyle \eta - \dfrac{1}{\alpha}\sum_{t=1}^T\pi_td_t \tag{\textbf{ILBFP0}} \\
\ \ \ \ \ \hbox{s.t.} \  &\nonumber {(\ref{for:CVaR_cy}), (\ref{for:CVaR_creturn}), (\ref{for:CVaR_cd}), (\ref{for:CVaR_frd}), (\ref{for:CVaR_cx}),   (\ref{for:CVaR_frx}), \tag{\text{Investor Constraints}}}\\
\label{for:bank_follower_of}&p\in arg \max \displaystyle \sum_{j\in B}p_jx_j,\\
\nonumber &\ \ \ \ \ \ \ \ \nonumber \ \ \ \ \hbox{s.t.} \ (\ref{for:bank_c1}), (\ref{for:bank_c2}), (\ref{for:bank_fr}),(\ref{for:bank_frP})  \tag{\text{{Broker-dealer} Constraints}}.
\end{align}

We {state} in the following proposition that if no further polyhedral constraints are imposed on {possible costs}, i.e., $\mathbb{P}=\mathbb{R}^{|B|}_+$, fixing the {transaction costs} to their maximum possible values is always an optimal solution of the follower ({broker-dealer}) problem.

\begin{proposition} \label{pro:bank-follow}
Let \textbf{\ref{for:bank_of} be the follower {broker-dealer} problem, not including constraint (\ref{for:bank_frP}),} in the problem \textbf{ILBFP0}. Let $x$ be a given portfolio and let $\displaystyle p_j^+=\max_{k=1,...,s_j}c_{jk} \quad \forall j \in B$. Then $p_j^+, \ \forall j \in B$, is an optimal solution of \textbf{\ref{for:bank_of}}.
\end{proposition}

Using the previous result, the \ref{for:ILBFP} can be simplified, in the cases in which constraint (\ref{for:bank_frP}) is not included since the nested optimization problem is replaced by the explicit form of an optimal solution. This results in a valid linear programming formulation to solve the problem.

\begin{align*}
\label{for:ILBFP_LP} \tag{\textbf{ILBFP-LP}}\max  &\displaystyle  \ \eta - \dfrac{1}{\alpha}\sum_{t=1}^T\pi_td_t\\
\nonumber \ \ \ \ \hbox{s.t.} \ &{(\ref{for:CVaR_cy}), (\ref{for:CVaR_creturn}), (\ref{for:CVaR_cd}), (\ref{for:CVaR_frd}), (\ref{for:CVaR_cx}),   (\ref{for:CVaR_frx}), \tag{\text{Investor Constraints}}}\\
&y_t=\sum_{j=1}^nr_{jt}x_j-\left(\sum_{j\in B}p_j^+x_j\right), \quad  t=1,...,T.
\end{align*}

Nevertheless, the above result can not be extended to the case in which a more general polyhedron $\mathbb{P}$ defines the admissible set of transaction costs, and a single level MILP formulation can neither be obtained. To solve \textbf{ILBFP}, in this more general case, we propose an \textit{`ad hoc'} algorithm. To justify its validity we need the following theorem.

\begin{theorem}\label{Theo:ILBFP-Compact}
Let {us define} {$\lambda=\sum_{j\in B}p_jx_j$}, and {let us} denote by $\Omega$ the set containing the feasible {commissions and fees} rates of the {broker-dealer} problem in $\mathbb{P}$, {denoted by $p_{int}$}. The problem \ref{for:ILBFP} is equivalent to:
{
\begin{align*}
\label{for:ILBFP_Compact} \tag{\textbf{ILBFP-Compact}} \max  &\ \displaystyle \eta - \dfrac{1}{\alpha}\sum_{t=1}^Tp_td_t\\
\ \ \ \ \hbox{st.} \ \ &y_t=\sum_{j=1}^nr_{jt}x_j-\left(\lambda\right), \hspace*{0.85cm} t=1,...,T, \\
&\sum_{t=1}^T\pi_ty_t\ge \mu_0,\\
&\displaystyle d_t\ge \eta - y_t, \hspace*{2.5cm} t=1,...,T, \\
& d_{t}\ge 0,\hspace*{3.2cm} t=1,...,T,\\
&\displaystyle \sum_{j=1}^n x_{j}\leq 1,  \\
& x_{j}\ge 0, \hspace*{3.15cm} j=1,...,n,\\
&\lambda\geq \sum_{j \in B} p_{int,j}x_j, \hspace*{1.5cm} p_{int}\in \Omega.\\
\end{align*}
}
\end{theorem}

\begin{proof}
We prove first that, maximizing the objective function  $\displaystyle \eta - \dfrac{1}{\alpha} \sum_{t=1}^T \pi_td_t$ in \ref{for:ILBFP} is equivalent to maximizing $\ {\displaystyle \eta (1-c_x)+\frac{1}{\alpha}\sum_{t\in \mathbb{T'}}\sum_{j=1}^n\pi_tr_{jt}x_j-c_x\lambda }$, where $c_x=\sum_{t\in \mathbb{T'}}  \frac{\pi_t}{\alpha}>0 $  and $ \mathbb{T'}:=\{t=1,...,n : \eta -y_t\geq 0\}$.  Observe that the constraints in \ref{for:ILBFP_Compact} imply that $d_t=\max\{0,\eta -y_t\}$ and $y_t=\sum_{j \in B}r_{jt}x_j-\sum_{j \in B}p_jx_j$ for all $t=1,...,T$. Therefore the objective value in the problem satisfies {the following rewriting}:

{\small
\begin{align}
\nonumber \max \eta - \dfrac{1}{\alpha} \sum_{t=1}^T \pi_td_t=&\max \eta - \dfrac{1}{\alpha} \sum_{t=1}^T \pi_t\max\{0,\eta-y_t\}\\
\nonumber &=\max \eta - \dfrac{1}{\alpha} \sum_{t\in \mathbb{T'}} \pi_t (\eta-y_t)\\
\nonumber &=\max \eta(1-c_x) + \frac{1}{\alpha}\sum_{t\in \mathbb{T'}} \pi_t \left(\sum_{j \in B}r_{jt}x_j-\sum_{j \in B}p_jx_j\right)\\
\label{for:theorem_constraint}&=\max \eta(1-c_x) + \frac{1}{\alpha}\sum_{t\in \mathbb{T'}} \pi_t \left(\sum_{j \in B}r_{jt}x_j\right)-c_x \lambda .
\end{align}
}

{Let $\lambda=\sum_{j\in B}p_j x_j$. The expression (\ref{for:theorem_constraint}) proves that the objective function of \ref{for:ILBFP_Compact} depends on $\lambda$ with a negative coefficient.}

Secondly, we have that, for a given portfolio $x$, the optimal value $\bar \lambda$ of the follower problem is

\begin{align*}
\bar \lambda=\max & \displaystyle \sum_{j\in B}p_jx_j\\
 \hbox{s.t.} & \ (\ref{for:bank_c1}), (\ref{for:bank_c2}), (\ref{for:bank_fr}),(\ref{for:bank_frP})  \tag{\text{{Broker-dealer} Constraints}},
\end{align*}

\noindent and it is equivalent to evaluate the objective function in all the feasible points and to choose the largest one: \[\bar \lambda=\max \sum_{j \in B} p_{int,j}x_j,\; \  p_{int}\in \Omega.\]
Since {$c_x\geq 1$}, $\lambda$ is positive, and $\lambda$ is being minimized in (\ref{for:theorem_constraint}), the follower problem in \ref{for:ILBFP}, can be replaced by \[\lambda\geq \sum_{j \in B} p_{int,j}x_j,\;  p_{int}\in \Omega,\]
and the result follows.
\end{proof}

Observe that, if the set of points in $\Omega$ were explicitly known, \ref{for:ILBFP_Compact} would be a MILP compact formulation with very likely an exponential number of constraints for the general case of \ref{for:ILBFP}. However, the points in the set $\Omega$ are usually difficult to enumerate a priori.

The idea of our algorithm is to start with an incomplete formulation of \ref{for:ILBFP_Compact} and reinforce it with a new inequality, coming from a new point in $\Omega$,  after each new iteration of the algorithm.
\bigskip

\noindent {\sc Algorithm 1:}
\begin{description}
\item {\tt Initialization} Choose a feasible portfolio $x^0$. Set $CVaR^{0}=+\infty$
\item {\tt Iteration $\tau=1,2,\ldots$}
\begin{itemize}
\item Solve the {broker-dealer} (follower) problem for $x^{\tau-1}$. Let $p^{\tau}$ be an optimal solution.
\item Solve the incomplete formulation:
\begin{align*}
\label{for:ILBFP_Incomplete} \tag{\textbf{ILBFP-Incomplete$^{\tau}$}} \max  & \displaystyle \ \eta - \dfrac{1}{\alpha}\sum_{t=1}^T\pi_t d_t\\
\ \ \ \ \hbox{st.} \ \ &y_t=\sum_{j=1}^nr_{jt}x_j-\left(\lambda\right), \hspace*{0.85cm} t=1,...,T, \\
&\sum_{t=1}^T\pi_ty_t\ge \mu_0\\
&\displaystyle d_t\ge \eta - y_t, \hspace*{2.5cm} t=1,...,T, \\
& d_{t}\ge 0,\hspace*{3.2cm} t=1,...,T,\\
&\displaystyle \sum_{j=1}^n x_{j}\leq 1,  \\
& x_{j}\ge 0, \hspace*{3.15cm} j=1,...,n,\\
&\lambda\geq \sum_{j \in B} p_j^{\nu}x_j, &  \nu=1,...,\tau.\\
\end{align*}

Let $\chi^{\tau}=(x^{\tau}, y^{\tau}, \eta^{\tau}, d^{\tau})$, and let $(\chi^{\tau}, \lambda^{\tau})$ be an optimal solution and $CVaR^{\tau}$ the optimal value.
\begin{itemize}
\item If $(\chi^{\tau}, \lambda^{\tau})$ is feasible in \ref{for:ILBFP_Incomplete}, $(\chi^{\tau -1 },p^{\tau})$ are optimal solutions of \ref{for:ILBFP}, and $CVaR^{\tau}$ the optimal value. END.
\item If $(\chi^{\tau}, \lambda^{\tau})$ is not feasible in \ref{for:ILBFP_Incomplete},  go to iteration $\tau:=\tau+1$.
\end{itemize}

\end{itemize}
\end{description}


We prove in the following result the optimality of the solution obtained in Algorithm 1 and also its finiteness.

\begin{theorem}
Algorithm 1 finishes in a finite number of iterations with an optimal solution of \ref{for:ILBFP}.
\end{theorem}

\begin{proof}
We start guaranteeing the finiteness of the algorithm. On the one hand, the number of feasible solutions of the {broker-dealer} problem is finite, then the number of different cuts $\lambda\geq \sum_{j \in B} p_j^{\tau}x_j$ that can be added to the incomplete formulation is also finite. On the other hand, if a repeated cut is added then, $x^{\tau -1}$ is feasible in \textbf{ILBFP-Incomplete$^{\tau}$}, since \textbf{ILBFP-Incomplete$^{\tau}$} is equal to \textbf{ILBFP-Incomplete$^{\tau-1}$}, and then the algorithm stops. Therefore the algorithm finishes in a finite number of iterations.

We continue now proving the optimality of the solution obtained. Let us denote by $CVaR^*$ the optimal value of \ref{for:ILBFP}, that by Theorem \ref{Theo:ILBFP-Compact} is also the optimal value of \ref{for:ILBFP_Compact}.

First, assume that $(\chi^{\tau -1}, \lambda^{\tau-1})$ satisfies the stopping criterion. Then, it is clear that $(\chi^{\tau -1}, \lambda^{\tau-1})$ is also feasible in \textbf{ILBFP-Incomplete$^{\tau}$} and $CVaR^{\nu} \leq CVaR^{\nu-1}$ for all $\nu=1,...,\tau$, by construction. Hence, $(\chi^{\tau}, \lambda^{\tau})$ is also optimal in \textbf{ILBFP-Incomplete$^{\tau}$} and $CVaR^{\tau-1} = CVaR^{\tau}$.

Second, we have that $CVaR^*\leq CVaR^{\tau}$ always holds, since the polyhedron describing the feasible region of \ref{for:ILBFP_Compact} is included in the one defining the feasible region in \textbf{ILBFP-Incomplete$^{\tau}$}.

Finally, we have that if $(\chi^{\tau -1 },p^{\tau})$ is feasible in \ref{for:ILBFP}, then $CVaR^*= CVaR^{\tau}$ and it is an optimal solution of \ref{for:ILBFP}. Therefore, it remains to prove that $(\chi^{\tau -1 },p^{\tau})$ is feasible in \ref{for:ILBFP}.

Clearly $\chi^{\tau -1 }$ verifies constraints (\ref{for:CVaR_cd}),  (\ref{for:CVaR_cx}), (\ref{for:CVaR_creturn}), (\ref{for:CVaR_frd}), (\ref{for:CVaR_frx}), since they are all included in the incomplete formulation, and also,  $x^{\tau -1 }$, $p^{\tau}$ verify constraints $p\in arg \max \displaystyle \sum_{j\in B}p_jx_j$, (\ref{for:bank_c1}), (\ref{for:bank_c2}), (\ref{for:bank_fr}) and (\ref{for:bank_frP}), since

\begin{align*}
p^{\tau} \in arg \max \displaystyle \sum_{j\in B}p_jx_j^{\tau-1}\\
\nonumber \ \ \ \ \ \ \ \ \nonumber \ \ \ \ \hbox{s.t.} \ (\ref{for:bank_c1}), (\ref{for:bank_c2}), (\ref{for:bank_fr}),(\ref{for:bank_frP})  \tag{\text{{Broker-dealer} Constraints}}.
\end{align*}

To complete the proof we need to check that constraint (\ref{for:CVaR_cy}) is also satisfied.

Since $p^{\tau} \in arg \max \displaystyle \sum_{j\in B}p_jx_j^{\tau-1}$, then $\sum_{j \in B} p_j^{\tau}x_j^{\tau -1 }\geq \sum_{j \in B} p_jx_j^{\tau -1 }$ for any {cost} $p$ verifying (\ref{for:bank_c1}), (\ref{for:bank_c2}), (\ref{for:bank_fr}) and (\ref{for:bank_frP}). Using the same arguments that in Theorem \ref{Theo:ILBFP-Compact} it follows that variable $\lambda$ is being minimized in \textbf{ILBFP-Incomplete$^{\tau}$}, thus $\lambda^{\tau}= \sum_{j \in B} p_j^{\tau}x_j^{\tau -1 }$ and then constraint (\ref{for:CVaR_cy}) holds.

\end{proof}

\section{The Maximum Social Welfare Problem (\textbf{MSWP})} \label{Cooperative}
In some actual situations, the investor and the broker-dealer may have an incentive to work together to improve the social welfare of society. They can agree to cooperate and share risk and benefits to improve, in this way, their solutions by designing a joint strategy.\par

We also analyze this problem for the sake of completeness and to compare the performance of this situation where none of the parties has a hierarchical position over the other one. We think that even if the actual implementation of the cooperative model may be difficult, in a competitive actual market, one may gain some insights into the problem through analysis.

In {the} social welfare model, we assume that both, broker-dealer and investor, cooperate. {Let $0<\xi<1$ be the marginal rate of substitution between the two objectives. That is, the rate at which one of the parties can give up some units of one of the objective functions in exchange for another unit of the other one while maintaining the same overall value. Then,} the cooperative version of the  problem can be written as a weighted sum of the two objective functions of each party in the feasible region delimited by the constraints of both problems:

\begin{align*}
 \max \ &\displaystyle \xi \sum_{j\in B}p_jx_j+(1-\xi)\left(\eta - \dfrac{1}{\alpha}\sum_{t=1}^T\pi_td_t \right)\\
\ \ \ \ \ \hbox{s.t.} \  & (\ref{for:bank_c1}), (\ref{for:bank_c2}), (\ref{for:bank_fr}), (\ref{for:bank_frP}) \tag{\text{{Broker-dealer} Constraints}},\\
& \\
&{(\ref{for:CVaR_cy}), (\ref{for:CVaR_creturn}), (\ref{for:CVaR_cd}), (\ref{for:CVaR_frd}), (\ref{for:CVaR_cx}),   (\ref{for:CVaR_frx}) \tag{\text{Investor Constraints}}}.
\end{align*}

The above problem can be modeled as a MILP by linearizing the products of variables $a_{jk}x_j, \forall {j\in B}$ following the same scheme as in Section \ref{Sect:Bank-leader}:

\begin{align*}
\label{for:CoopP} \tag{\textbf{MSWP0}} \max \ &\displaystyle \xi \sum_{j\in B}\sum_{k=1}^{s_j} c_{jk}\hat{a}_{jk}+(1-\xi)\left(\eta - \dfrac{1}{\alpha}\sum_{t=1}^T\pi_td_t \right)\\
\ \ \ \ \ \hbox{s.t.} \  & \quad (\ref{for:bank_c2}), (\ref{for:bank_fr}), (\ref{for:bank_frP}) \tag{\text{Linear {broker-dealer} Constraints}},\\
& \\
\nonumber & { (\ref{for:CVaR_cy_linear}), (\ref{for:CVaR_creturn}), (\ref{for:CVaR_cd}), (\ref{for:CVaR_frd}), (\ref{for:CVaR_cx_a}),   (\ref{for:CVaR_frx}), (\ref{for:linear_a}). \tag{\text{Linear investor Constraints 1}}}\\
\end{align*}

{For simplicity, in the remaining, we} consider an unweighted maximum social welfare model where the two objective functions $\sum_{j\in B}\sum_{k=1}^{s_j} c_{jk}{a}_{jk}$ ({broker-dealer}) and $\eta - \dfrac{1}{\alpha}\sum_{t=1}^T\pi_td_t$ (investor) are simply added. The following result proves that cooperation is always profitable for both parties in that the joint return exceeds the sum of individual returns of each of them.

\begin{proposition} \label{prop:social_welfare}
An optimal solution of the unweighted maximum social welfare {problem} induces an objective value that is greater than or equal to the sum of the optimal returns of the two parties in the same bilevel problem in any of the hierarchical {problems}.
\end{proposition}

\begin{proof}
Any feasible solution of \ref{for:BLIFP} and \ref{for:ILBFP} is feasible in \ref{for:CoopP} since all the constraints in this last problem appear in the two former formulations. Therefore, the feasible region of \ref{for:CoopP} includes the feasible regions of both, \ref{for:BLIFP} and \ref{for:ILBFP} and the result follows.
\end{proof}

\subsection{Benders decomposition}

We can also obtain a Benders decomposition, {\cite{ben62},} in order to state a Benders like algorithm to solve \ref{for:CoopP}, and compare the performance of both proposed methods to solve the problem.\par

Recall that the unweighted maximum welfare {problem} can be written as:
\begin{align*}
 \max \ &\displaystyle \sum_{j\in B}p_jx_j+\left(\eta - \dfrac{1}{\alpha}\sum_{t=1}^T\pi_td_t \right)\\
\ \ \ \ \ \hbox{s.t.} \  & (\ref{for:bank_c1}), (\ref{for:bank_c2}), (\ref{for:bank_fr}), (\ref{for:bank_frP}) \tag{\text{{Broker-dealer} Constraints}},\\
& \\
&{(\ref{for:CVaR_cy}), (\ref{for:CVaR_creturn}), (\ref{for:CVaR_cd}), (\ref{for:CVaR_frd}), (\ref{for:CVaR_cx}),   (\ref{for:CVaR_frx}) \tag{\text{Investor Constraints}}}.
\end{align*}
%

In order to apply Benders decomposition we reformulate \ref{for:CoopP} as follows:

\begin{align*}
 \label{for:CoopP'} \tag{\textbf{MSWP1}} \max &\displaystyle \sum_{j\in B}\sum_{k=1}^{s_j} c_{jk}\hat{a}_{jk}+q(y)\\
 \hbox{s.t.} &\quad (\ref{for:bank_c2}), (\ref{for:bank_fr}), (\ref{for:bank_frP}) \tag{\text{Linear {broker-dealer} Constraints}}\\
&\begin{array}{lll}
\displaystyle \hat{a}_{jk}\leq a_{jk},&  j\in B, k=1,...,s_j,\\
\displaystyle \hat{a}_{jk}\geq 0, &  j\in B, k=1,...,s_j,\\
\end{array}\tag{\ref{for:linear_a}}\\
\tag{\ref{for:CVaR_cy_linear}}&y_t=\sum_{j\in B}r_{jt}\left( \sum_{k=1}^{s_j} \hat a_{jk}\right) +\sum_{j \in R}r_{jt}x_j-\sum_{j \in B}\sum_{k=1}^{s_j}c_{jk}\hat a_{jk}, \hspace*{0.3cm} t=1,...,T, \\
\nonumber&\sum_{t=1}^T\pi_ty_t\ge \mu_0, \tag{\ref{for:CVaR_creturn}} \\
\tag{\ref{for:CVaR_cx_a}}&\displaystyle \sum_{j \in B}\sum_{k=1}^{s_j}\hat a_{jk}+\sum_{j \in R}x_j \leq 1,  \\
\nonumber& \tag{\ref{for:CVaR_frx}} x_{j}\ge 0, \hspace*{1.4cm} j \in R,\\
\end{align*}

where
\begin{align*}
 q(y)= \max & \; \displaystyle\;  \eta - \dfrac{1}{\alpha}\sum_{t=1}^T\pi_td_t\\
\mbox{s.t.: } & d_t-\eta \geq -y_t, \hspace*{0.6cm}  t=1,...,T,\\
&d_t\geq 0, \hspace*{1.7cm} t=1,...,T.
\end{align*}

Note that in $q(y)$ we are essentially computing the CVaR for the given solution $\{y_t:t=1,\ldots,T\}$.

Computing again its dual problem, the evaluation of $q(y)$ can also be obtained as:

\begin{align*}
\hspace*{-1cm} \label{for:PrimalP}\tag{\textbf{PrimalP}} q(y)= \min & \sum_{t=1}^T -\gamma_t y_t \\
\mbox{s.t.: } & \gamma_t\geq \frac{-\pi_t}{\alpha}, \quad  t=1,...,T,\\
&-\sum_{t=1}^T \gamma_t=1,\\
&\gamma_t\leq 0.
\end{align*}

Observe that the above problem, which we define as the Primal Problem, is a continuous knapsack problem with lower bounds, therefore it is well known that it can be solved by inspection. It suffices to sort non-increasingly the $y_t$ values and assigning, in that order, to each variable $\gamma_t$  the minimum feasible amount. \par

Note that in the above formulation the feasible region does not depend on the variables in \ref{for:CoopP'}, so if we denote by $\Omega$ the set of extreme point solutions of the feasible region of \ref{for:PrimalP}, $q(y)$ is equivalent to:

\begin{align}
\nonumber q(y)=\max &\quad q \\
  \hbox{s.t. } & \label{for:Benders}\displaystyle q\leq  \sum_{t=1}^T -\gamma_t^{\tau} y_t, \quad  \gamma^{\tau} \in \Omega.
\end{align}

Therefore, the problem \ref{for:CoopP} with discrete {costs} can be written as:

\begin{align*}
 \label{for:MasterP} \tag{\textbf{MasterP}} \max &\displaystyle \sum_{j\in B}\sum_{k=1}^{s_j} c_{jk}\hat{a}_{jk}+q\\
 \hbox{s.t.} &\quad (\ref{for:bank_c2}), (\ref{for:bank_fr}), (\ref{for:bank_frP}) \tag{\text{Linear {broker-dealer} Constraints}}\\
&\begin{array}{lll}
\displaystyle \hat{a}_{jk}\leq a_{jk},&  j\in B, k=1,...,s_j,\\
\displaystyle \hat{a}_{jk}\geq 0, &  j\in B, k=1,...,s_j,\\
\end{array}\tag{\ref{for:linear_a}}\\
\tag{\ref{for:CVaR_cy_linear}}&y_t=\sum_{j\in B}r_{jt}\left( \sum_{k=1}^{s_j} \hat a_{jk}\right) +\sum_{j \in R}r_{jt}x_j-\sum_{j \in B}\sum_{k=1}^{s_j}c_{jk}\hat a_{jk}, \hspace*{0.3cm} t=1,...,T, \\
\nonumber&\sum_{t=1}^T\pi_ty_t\ge \mu_0, \tag{\ref{for:CVaR_creturn}} \\
\tag{\ref{for:CVaR_cx_a}}&\displaystyle \sum_{j \in B}\sum_{k=1}^{s_j}\hat a_{jk}+\sum_{j \in R}x_j \leq 1,  \\
\nonumber& \tag{\ref{for:CVaR_frx}} x_{j}\ge 0, \hspace*{1.4cm} j \in R,\\
&q\leq  \sum_{t=1}^T \gamma^{\tau} y_t,\hspace*{0.5cm}  \gamma^{\tau} \in \Omega. \tag{\ref{for:Benders}}
\end{align*}

This analysis allows us to state a Benders algorithm as follows:
\bigskip

\noindent {\sc Benders Algorithm:}

\begin{description}
\item {\tt Initialization} Choose a solution $y^0$ of the master problem, solve the primal problem \ref{for:PrimalP} for the chosen $y^0$. Let $\gamma^0$ be an optimal solution for \ref{for:PrimalP} under $y^0$ and $q(y^0)$ the corresponding optimal value. Take $\mathbf{\Upsilon}=\{\gamma^0\}$ and go to iteration $\tau=1$.
\item {\tt Iteration $\tau=1,2,\ldots$} Solve the master problem $\ref{for:MasterP}$ replacing $\Omega$ with $\Upsilon$. Let $y^*$ and $q^*$ be optimal solutions of such problem.
    \begin{itemize}
    \item If $\tau=1$ and $q(y^0)=q^*$. END.
    \item If $\tau>1$ and $q(y^*)=q^*$. END.
    \item Otherwise, solve the primal problem \ref{for:PrimalP} for $y=y^*$. 	 Let $\gamma^*$ be an optimal solution of such problem. Take $\gamma^{\tau}=\gamma^*$,  			 $\mathbf{\Upsilon}=\mathbf{\Upsilon}\cup\{\gamma^{\tau}\}$,  and go to iteration $\tau:=\tau+1$.
    \end{itemize}
\end{description}

\section{Computational study and empirical application} \label{Computational}

This section is devoted to reporting some numerical experiments conducted to:
1) compare the effectiveness of the methods proposed to solve the different {problems}; 2) analyze the form of the solutions within each {model}; and 3) compare the profiles of the solutions, in terms of net values for the broker-dealer and expected return for the investor,  across the three {defined problems}.

The computational experiments were carried out on a personal computer with Intel(R) Core(TM) i7-2600 {\scriptsize CPU}, 3.40GHz with 16.0 GB RAM. The algorithms and formulations were implemented and solved by using Xpress IVE 8.0.

{In order to conduct the computational study, we {take} historical data from Dow Jones Industrial Average. We {considered} daily returns of the 30 assets during one year ($T=251$ scenarios), and these $T$ historical periods are considered as equiprobable scenarios ($\pi_t=1/T$). Furthermore, to {perform} a richer comparison, we consider different types of instances for the broker-dealer sets of possible {transaction costs} and different risk profiles for the investor.

We assume that the broker-dealer charges transaction costs in a subset $B$ of the securities. In the instances we generated we {compare} the following cardinals for the set $B$: $|B|= 30, 20, 10$.  In addition, each {cost} $p_j, j\in B$ was chosen {from a discrete set $\mathbb{P}_j=\{c_{j1},...,c_{js_j}\}$  of admissible values} . These parameters $s_j$ were randomly generated in the interval $[0,K]$ with $K=5,15,50$.}

The next table gathers the nine different types of instances (A to I) considered in our computational study:

\bigskip

\begin{table}[H]
\begin{center}
\begin{tabular}{c|c|c|c}
           &     $K=5$ &    $K=15$ &   $K=50$ \\
\hline
 $|B|=30$ & {\bf A} & {\bf B} & {\bf C} \\
\hline
 $|B|=20$ & {\bf D} & {\bf E} & {\bf F} \\
\hline
 $|B|=10$ & {\bf G} & {\bf H} & {\bf I} \\
\end{tabular}
\caption{\label{table:instances} {{\footnotesize Types of instances for the sets of possible {costs} depending on the values of $|B|$ and $K$}}}
\end{center}
\end{table}

{Once the set $B$ and the parameter $s_j$ were set for each type of instance (A-I), we generate the possible {transaction costs} $c_{ij}$ as follows:

\begin{itemize}
\item randomly generated in the interval $[0.001,0.003]$ (\textit{cheaper} {costs}) in approximately $15\%$ of the securities,
\item randomly generated in the interval $[0.002,0.008]$ (\textit{normal} {costs}) in approximately $70\%$ of the securities,
\item randomly generated in the interval $[0.006,0.010]$ (\textit{more expensive} {costs}) in approximately $15\%$ of the securities.
\end{itemize}

}

For each type of instance defined in Table \ref{table:instances}, five different instances are generated and the average values are reported in all the tables and figures.

{Different investor profiles are also considered varying the values of parameters $\mu_0$ and $\alpha$. We assume three thresholds for the expected return $\mu_0=0.0,0.05, 0.1$. This way, we are modeling investors willing not to lose anything, or to win at least, $5\%$ or $10\%$ of their invested amount. In addition, we consider five different CVaR risk levels, $\alpha=0.01,0.05,0.5,0.9$. Note that usually, the smaller the $\alpha$, the higher the risk-aversion.}

\subsection{Comparing solution methods}
This section compares the computational performance of the different methods proposed to solve each one of the {problems}.

For the first {problem}, \textbf{BLIFP}, we proposed two different formulations: \ref{for:BLIFP1} and \ref{for:BLIFP2}. We show in all our tables, the average CPU time expressed in seconds (CPU) and the number of {instances} (\#) solved to optimality (out of 5) for each formulation, with a time limit of 3600 seconds.

Table \ref{table: model1} is organized in three blocks of rows. Each block reports results for $\mu_0=0.0,0.05, 0.1$, respectively. Each row in the table refers to a type of instance ($A,\ldots,I$). The columns are also organized in four blocks. Each block reports the results for a different risk level ($\alpha$).

{It can be observed that \ref{for:BLIFP2} is always faster and  it solves a higher number of problems than \ref{for:BLIFP1} to optimality. {As anticipated in Section \ref{ss:BLIFP2} this behavior is explained by the smaller dimension of \ref{for:BLIFP2} in terms of variables and constraints.} For example, when $\alpha=0.5$ and $\mu=0.0$, \ref{for:BLIFP2} is able to solve all the instances of types D and H in few seconds, while \ref{for:BLIFP1} is not able to solve any of these instances. Therefore, we conclude that formulation \ref{for:BLIFP2} is more effective than \ref{for:BLIFP1} for solving  {\textbf{BLIFP}}.}

{The second {problem} in our analysis is the one presented in Section \ref{Investor-leader}, namely \textbf{ILBFP}. For this situation, we have proposed a single level LP formulation \ref{for:ILBFP_LP} and Algorithm 1 to solve the problem. We report the results concerning this model (when no additional constrains on {transaction costs} are imposed in the set of {costs}) in Table \ref{table: model2}. It can be observed that the compact formulation is faster than the algorithm: all the instances can be solved by using the LP formulation in less than 7 seconds, meanwhile, the algorithm needs more than 100 seconds to solve some of them. However, the Algorithm 1 is also able to solve all the instances, and, as discussed in Section \ref{Investor-leader}, it can also be used when more general sets of {costs} are considered. }

{Finally, for the social welfare {problem,} \textbf{MSWP}, we have also proposed another single level formulation \ref{for:CoopP} and a Benders' like algorithm. The primal problems in the Benders Algorithm were solved by using the inspection method described in the previous section. We report the results concerning this model in Table \ref{table: model3} with the same layout as Table \ref{table: model2}. It can be observed that again the compact formulations is much faster than the algorithm. In spite of that, the algorithm is also able to solve the considered instances. }

{
\begin{table}[H]
\resizebox{\textwidth}{!}{%
\begin{tabular}{rr|rrrr|rrrr|rrrr|rrrr|}

           &            &         \multicolumn{ 4}{c}{{\bf $\alpha=0.05$}} &         \multicolumn{ 4}{c}{{\bf $\alpha=0.1$}} &         \multicolumn{ 4}{c}{{\bf $\alpha=0.5$}} &         \multicolumn{ 4}{c}{{\bf $\alpha=0.9$}}  \\

      $\mu_0$      &           & \multicolumn{ 2}{c}{\ref{for:BLIFP1}} & \multicolumn{ 2}{c}{\ref{for:BLIFP2}} & \multicolumn{ 2}{c}{\ref{for:BLIFP1}} & \multicolumn{ 2}{c}{\ref{for:BLIFP2}} & \multicolumn{ 2}{c}{\ref{for:BLIFP1}} & \multicolumn{ 2}{c}{\ref{for:BLIFP2}} & \multicolumn{ 2}{c}{\ref{for:BLIFP1}} & \multicolumn{ 2}{c}{\ref{for:BLIFP2}} \\

           &            &        {\scriptsize CPU} &       {\scriptsize \#} &        {\scriptsize CPU} &       {\scriptsize \#} &        {\scriptsize CPU} &        {\scriptsize \#} &        {\scriptsize CPU} &       {\scriptsize \#} &        {\scriptsize CPU} &       {\scriptsize \#} &        {\scriptsize CPU} &       {\scriptsize \#} &        {\scriptsize CPU} &       {\scriptsize \#} &        {\scriptsize CPU} &       {\scriptsize \#} \\
\hline
    0     & \textbf{A} & 3600  & 0     & 181   & 5     & 3600  & 0     & 916   & 4     & 3600  & 0     & 3291  & 1     & 3600  & 0     & 5     & 5 \\
          & \textbf{B} & 3600  & 0     & 3600  & 0     & 3600  & 0     & 3600  & 0     & 3600  & 0     & 3600  & 0     & 3600  & 0     & 3079  & 1 \\
          & \textbf{C} & 3600  & 0     & 3600  & 0     & 3600  & 0     & 3600  & 0     & 3600  & 0     & 3600  & 0     & 3600  & 0     & 3600  & 0 \\
          & \textbf{D} & 3600  & 0     & 2     & 5     & 3204  & 1     & 17    & 5     & 3600  & 0     & 59    & 5     & 1603  & 3     & 2     & 5 \\
          & \textbf{E} & 3600  & 0     & 890   & 4     & 3600  & 0     & 2024  & 3     & 3600  & 0     & 3377  & 1     & 1882  & 3     & 3     & 5 \\
          & \textbf{F} & 3600  & 0     & 2895  & 1     & 3600  & 0     & 3600  & 0     & 3600  & 0     & 3600  & 0     & 2984  & 1     & 76    & 5 \\
          & \textbf{G} & 841   & 5     & 1     & 5     & 375   & 5     & 1     & 5     & 810   & 5     & 1     & 5     & 571   & 5     & 1     & 5 \\
          & \textbf{H} & 2282  & 2     & 2     & 5     & 3117  & 1     & 2     & 5     & 3600  & 0     & 7     & 5     & 2178  & 2     & 1     & 5 \\
          & \textbf{I} & 2959  & 1     & 1444  & 3     & 3600  & 0     & 1562  & 3     & 3600  & 0     & 939   & 4     & 1804  & 3     & 5     & 5 \\
 \hline
    0.05  & \textbf{A} & 3600  & 0     & 28    & 5     & 3600  & 0     & 343   & 5     & 3600  & 0     & 3291  & 1     & 3156  & 1     & 5     & 5 \\
          & \textbf{B} & 3600  & 0     & 3600  & 0     & 3600  & 0     & 3600  & 0     & 3600  & 0     & 3600  & 0     & 3600  & 0     & 3079  & 1 \\
          & \textbf{C} & 3600  & 0     & 3600  & 0     & 3600  & 0     & 3600  & 0     & 3600  & 0     & 3600  & 0     & 3600  & 0     & 3600  & 0 \\
          & \textbf{D} & 3600  & 0     & 2     & 5     & 3204  & 1     & 3     & 5     & 3600  & 0     & 59    & 5     & 2217  & 2     & 2     & 5 \\
          & \textbf{E} & 3600  & 0     & 110   & 5     & 3600  & 0     & 1923  & 3     & 3600  & 0     & 3377  & 1     & 1793  & 3     & 3     & 5 \\
          & \textbf{F} & 3600  & 0     & 2905  & 1     & 3600  & 0     & 3600  & 0     & 3600  & 0     & 3600  & 0     & 2930  & 1     & 76    & 5 \\
          & \textbf{G} & 841   & 5     & 1     & 5     & 375   & 5     & 1     & 5     & 810   & 5     & 1     & 5     & 62    & 5     & 1     & 5 \\
          & \textbf{H} & 2282  & 2     & 1     & 5     & 3117  & 1     & 2     & 5     & 3600  & 0     & 7     & 5     & 153   & 5     & 1     & 5 \\
          & \textbf{I} & 2959  & 1     & 930   & 4     & 3600  & 0     & 1575  & 3     & 3600  & 0     & 939   & 4     & 2439  & 2     & 5     & 5 \\
\hline
    0.1   & \textbf{A} & 3600  & 0     & 6     & 5     & 3600  & 0     & 23    & 5     & 3600  & 0     & 3291  & 1     & 3156  & 1     & 5     & 5 \\
          & \textbf{B} & 3600  & 0     & 616   & 5     & 3600  & 0     & 1326  & 4     & 3600  & 0     & 3600  & 0     & 3600  & 0     & 3079  & 1 \\
          & \textbf{C} & 3600  & 0     & 3600  & 0     & 3600  & 0     & 3600  & 0     & 3600  & 0     & 3600  & 0     & 3600  & 0     & 3600  & 0 \\
          & \textbf{D} & 3600  & 0     & 1     & 5     & 3204  & 1     & 2     & 5     & 3600  & 0     & 59    & 5     & 2217  & 2     & 2     & 5 \\
          & \textbf{E} & 3600  & 0     & 24    & 5     & 3600  & 0     & 55    & 5     & 3600  & 0     & 3377  & 1     & 1793  & 3     & 3     & 5 \\
          & \textbf{F} & 3600  & 0     & 1277  & 4     & 3600  & 0     & 2227  & 2     & 3600  & 0     & 3600  & 0     & 2930  & 1     & 76    & 5 \\
          & \textbf{G} & 841   & 5     & 1     & 5     & 375   & 5     & 1     & 5     & 810   & 5     & 1     & 5     & 62    & 5     & 1     & 5 \\
          & \textbf{H} & 2282  & 2     & 1     & 5     & 3117  & 1     & 1     & 5     & 3600  & 0     & 7     & 5     & 153   & 5     & 1     & 5 \\
          & \textbf{I} & 2959  & 1     & 1477  & 3     & 3600  & 0     & 736   & 4     & 3600  & 0     & 939   & 4     & 2439  & 2     & 5     & 5 \\
    \hline
\end{tabular}

}
\caption{\label{table: model1} {\footnotesize Comparison of the average CPU and number of {instances} (out of 5) solved to optimality, for \ref{for:BLIFP1} and \ref{for:BLIFP2}}}

\end{table}
}



\begin{landscape}
\begin{multicols}{2}
\begin{center}
\begin{table}[H]
\centering
\begin{footnotesize}
\begin{tabular}{rr|rr|rr|rr|rr|}
           &            & \multicolumn{ 2}{c}{$\alpha=0.05$} & \multicolumn{ 2}{c}{$\alpha=0.1$} & \multicolumn{ 2}{c}{$\alpha=0.5$} & \multicolumn{ 2}{c}{$\alpha=0.9$} \\

  $\mu_0$ &            & {\scriptsize \textbf{LP}} & {\scriptsize Alg. 1} & {\scriptsize \textbf{LP}} & {\scriptsize Alg. 1} & {\scriptsize \textbf{LP}} &{\scriptsize Alg. 1} & {\scriptsize \textbf{LP}} & {\scriptsize Alg. 1}  \\
\hline
0.0 & \textbf{A} & 0.55  & 4.76  & 0.57  & 17.85 & 0.56  & 54.62 & 0.54  & 19.29 \\
          & \textbf{B} & 1.51  & 12.64 & 1.61  & 50.68 & 1.60  & 144.21 & 1.49  & 47.97 \\
          & \textbf{C} & 6.42  & 44.94 & 6.61  & 178.84 & 6.19  & 557.45 & 5.98  & 187.49 \\
          & \textbf{D} & 0.39  & 6.80  & 0.41  & 12.09 & 0.42  & 40.26 & 0.42  & 13.18 \\
          & \textbf{E} & 1.03  & 36.98 & 1.02  & 29.99 & 1.03  & 95.80 & 0.99  & 32.80 \\
          & \textbf{F} & 3.26  & 24.52 & 3.31  & 86.80 & 3.31  & 298.50 & 3.23  & 89.21 \\
          & \textbf{G} & 0.25  & 2.83  & 0.25  & 8.73  & 0.25  & 26.00 & 0.25  & 8.45 \\
          & \textbf{H} & 0.47  & 3.99  & 0.47  & 14.37 & 0.48  & 46.73 & 0.46  & 15.32 \\
          & \textbf{I} & 1.63  & 13.59 & 1.63  & 45.47 & 1.64  & 149.62 & 1.59  & 49.73 \\
\hline
    0.05& \textbf{A} & 0.55  & 4.76  & 0.57  & 17.85 & 0.56  & 54.62 & 0.54  & 19.29 \\
          & \textbf{B} & 1.51  & 12.64 & 1.61  & 50.68 & 1.60  & 144.21 & 1.49  & 47.97 \\
          & \textbf{C} & 6.42  & 44.94 & 6.61  & 178.84 & 6.19  & 557.45 & 5.98  & 187.49 \\
          & \textbf{D} & 0.39  & 6.80  & 0.41  & 12.09 & 0.42  & 40.26 & 0.42  & 13.18 \\
          & \textbf{E} & 1.03  & 36.98 & 1.02  & 29.99 & 1.03  & 95.80 & 0.99  & 32.80 \\
          & \textbf{F} & 3.26  & 24.52 & 3.31  & 86.80 & 3.31  & 298.50 & 3.23  & 89.21 \\
          & \textbf{G} & 0.25  & 2.83  & 0.25  & 8.73  & 0.25  & 26.00 & 0.25  & 8.45 \\
          & \textbf{H} & 0.47  & 3.99  & 0.47  & 14.37 & 0.48  & 46.73 & 0.46  & 15.32 \\
          & \textbf{I} & 1.63  & 13.59 & 1.63  & 45.47 & 1.64  & 149.62 & 1.59  & 49.73 \\
\hline
    0.1 & \textbf{A} & 0.55  & 4.76  & 0.57  & 17.85 & 0.56  & 54.62 & 0.54  & 19.29 \\
          & \textbf{B} & 1.51  & 12.64 & 1.61  & 50.68 & 1.60  & 144.21 & 1.49  & 47.97 \\
          & \textbf{C} & 6.42  & 44.94 & 6.61  & 178.84 & 6.19  & 557.45 & 5.98  & 187.49 \\
          & \textbf{D} & 0.39  & 6.80  & 0.41  & 12.09 & 0.42  & 40.26 & 0.42  & 13.18 \\
          & \textbf{E} & 1.03  & 36.98 & 1.02  & 29.99 & 1.03  & 95.80 & 0.99  & 32.80 \\
          & \textbf{F} & 3.26  & 24.52 & 3.31  & 86.80 & 3.31  & 298.50 & 3.23  & 89.21 \\
          & \textbf{G} & 0.25  & 2.83  & 0.25  & 8.73  & 0.25  & 26.00 & 0.25  & 8.45 \\
          & \textbf{H} & 0.47  & 3.99  & 0.47  & 14.37 & 0.48  & 46.73 & 0.46  & 15.32 \\
          & \textbf{I} & 1.63  & 13.59 & 1.63  & 45.47 & 1.64  & 149.62 & 1.59  & 49.73 \\

\hline

\end{tabular}

\end{footnotesize}
\caption{\label{table: model2} {\footnotesize Comparison of the average CPU for \ref{for:ILBFP_LP} and Algorithm 1}}
\end{table}

\columnbreak

\begin{table}[H]
\centering
\begin{footnotesize}

\begin{tabular}{rr|rr|rr|rr|rr|}

&            & \multicolumn{ 2}{c}{$\alpha=0.05$} & \multicolumn{ 2}{c}{$\alpha=0.1$} & \multicolumn{ 2}{c}{$\alpha=0.5$} & \multicolumn{ 2}{c}{$\alpha=0.9$} \\
  $\mu_0$ &            & {\tiny \ref{for:CoopP}} &       Ben. & {\tiny \ref{for:CoopP}} &       Ben. & {\tiny \ref{for:CoopP}} &       Ben. & {\tiny \ref{for:CoopP}} &       Ben.  \\
\hline
  0     & \textbf{A} & 0.55  & 4.76  & 0.57  & 17.85 & 0.56  & 54.62 & 0.54  & 19.29 \\
          & \textbf{B} & 1.51  & 12.64 & 1.61  & 50.68 & 1.60  & 144.21 & 1.49  & 47.97 \\
          & \textbf{C} & 6.42  & 44.94 & 6.61  & 178.84 & 6.19  & 557.45 & 5.98  & 187.49 \\
          & \textbf{D} & 0.39  & 6.80  & 0.41  & 12.09 & 0.42  & 40.26 & 0.42  & 13.18 \\
          & \textbf{E} & 1.03  & 36.98 & 1.02  & 29.99 & 1.03  & 95.80 & 0.99  & 32.80 \\
          & \textbf{F} & 3.26  & 24.52 & 3.31  & 86.80 & 3.31  & 298.50 & 3.23  & 89.21 \\
          & \textbf{G} & 0.25  & 2.83  & 0.25  & 8.73  & 0.25  & 26.00 & 0.25  & 8.45 \\
          & \textbf{H} & 0.47  & 3.99  & 0.47  & 14.37 & 0.48  & 46.73 & 0.46  & 15.32 \\
          & \textbf{I} & 1.63  & 13.59 & 1.63  & 45.47 & 1.64  & 149.62 & 1.59  & 49.73 \\
\hline
    0.05  & \textbf{A} & 0.55  & 4.76  & 0.57  & 17.85 & 0.56  & 54.62 & 0.54  & 19.29 \\
          & \textbf{B} & 1.51  & 12.64 & 1.61  & 50.68 & 1.60  & 144.21 & 1.49  & 47.97 \\
          & \textbf{C} & 6.42  & 44.94 & 6.61  & 178.84 & 6.19  & 557.45 & 5.98  & 187.49 \\
          & \textbf{D} & 0.39  & 6.80  & 0.41  & 12.09 & 0.42  & 40.26 & 0.42  & 13.18 \\
          & \textbf{E} & 1.03  & 36.98 & 1.02  & 29.99 & 1.03  & 95.80 & 0.99  & 32.80 \\
          & \textbf{F} & 3.26  & 24.52 & 3.31  & 86.80 & 3.31  & 298.50 & 3.23  & 89.21 \\
          & \textbf{G} & 0.25  & 2.83  & 0.25  & 8.73  & 0.25  & 26.00 & 0.25  & 8.45 \\
          & \textbf{H} & 0.47  & 3.99  & 0.47  & 14.37 & 0.48  & 46.73 & 0.46  & 15.32 \\
          & \textbf{I} & 1.63  & 13.59 & 1.63  & 45.47 & 1.64  & 149.62 & 1.59  & 49.73 \\
\hline
    0.1   & \textbf{A} & 0.55  & 4.76  & 0.57  & 17.85 & 0.56  & 54.62 & 0.54  & 19.29 \\
          & \textbf{B} & 1.51  & 12.64 & 1.61  & 50.68 & 1.60  & 144.21 & 1.49  & 47.97 \\
          & \textbf{C} & 6.42  & 44.94 & 6.61  & 178.84 & 6.19  & 557.45 & 5.98  & 187.49 \\
          & \textbf{D} & 0.39  & 6.80  & 0.41  & 12.09 & 0.42  & 40.26 & 0.42  & 13.18 \\
          & \textbf{E} & 1.03  & 36.98 & 1.02  & 29.99 & 1.03  & 95.80 & 0.99  & 32.80 \\
          & \textbf{F} & 3.26  & 24.52 & 3.31  & 86.80 & 3.31  & 298.50 & 3.23  & 89.21 \\
          & \textbf{G} & 0.25  & 2.83  & 0.25  & 8.73  & 0.25  & 26.00 & 0.25  & 8.45 \\
          & \textbf{H} & 0.47  & 3.99  & 0.47  & 14.37 & 0.48  & 46.73 & 0.46  & 15.32 \\
          & \textbf{I} & 1.63  & 13.59 & 1.63  & 45.47 & 1.64  & 149.62 & 1.59  & 49.73 \\

\hline
\end{tabular}
\end{footnotesize}
\caption{\label{table: model3} {\footnotesize Comparison of the average CPU for  \ref{for:CoopP} and Benders Algorithm}}
\end{table}
\end{center}

\end{multicols}
\end{landscape}

\subsection{Comparing solutions and risk profiles within {problems}}

This section analyzes the results provided by the {two hierarchical} {problems} in terms of broker-dealer's net profit and risk and expected return attained by the investor.

{Figure \ref{Grap:CVaR_M1} compares the CVaR values obtained for the different risk profiles for \textbf{BLIFP}. Each piecewise curve reports the CVaR values for different $\alpha$-levels and $\mu_0$-levels and the nine markets profiles ($A,\ldots,I$). We observe that the CVaR always increases with the value of $\alpha$, since this implies assuming more risk. It can also be seen in these figures that, when the value of $\alpha$ increases, the CVaR for the different values of $\mu_0$ becomes closer for each value of $\alpha$. This can be explained because when $\alpha=1$, if the constraint that the expected return must be greater or equal to $0$ is satisfied, both problems become the same, then, the bigger the $\alpha$ the more similar the results for different values of $\mu_0$. Furthermore, for a given $\alpha$, the CVaR for smaller $\mu_0$ is higher because in these cases the constraint on the expected return enlarges the feasible region as compared with higher values of $\mu_0$.}

\begin{figure}[H]
\centering

\fbox{\includegraphics[scale=0.7]{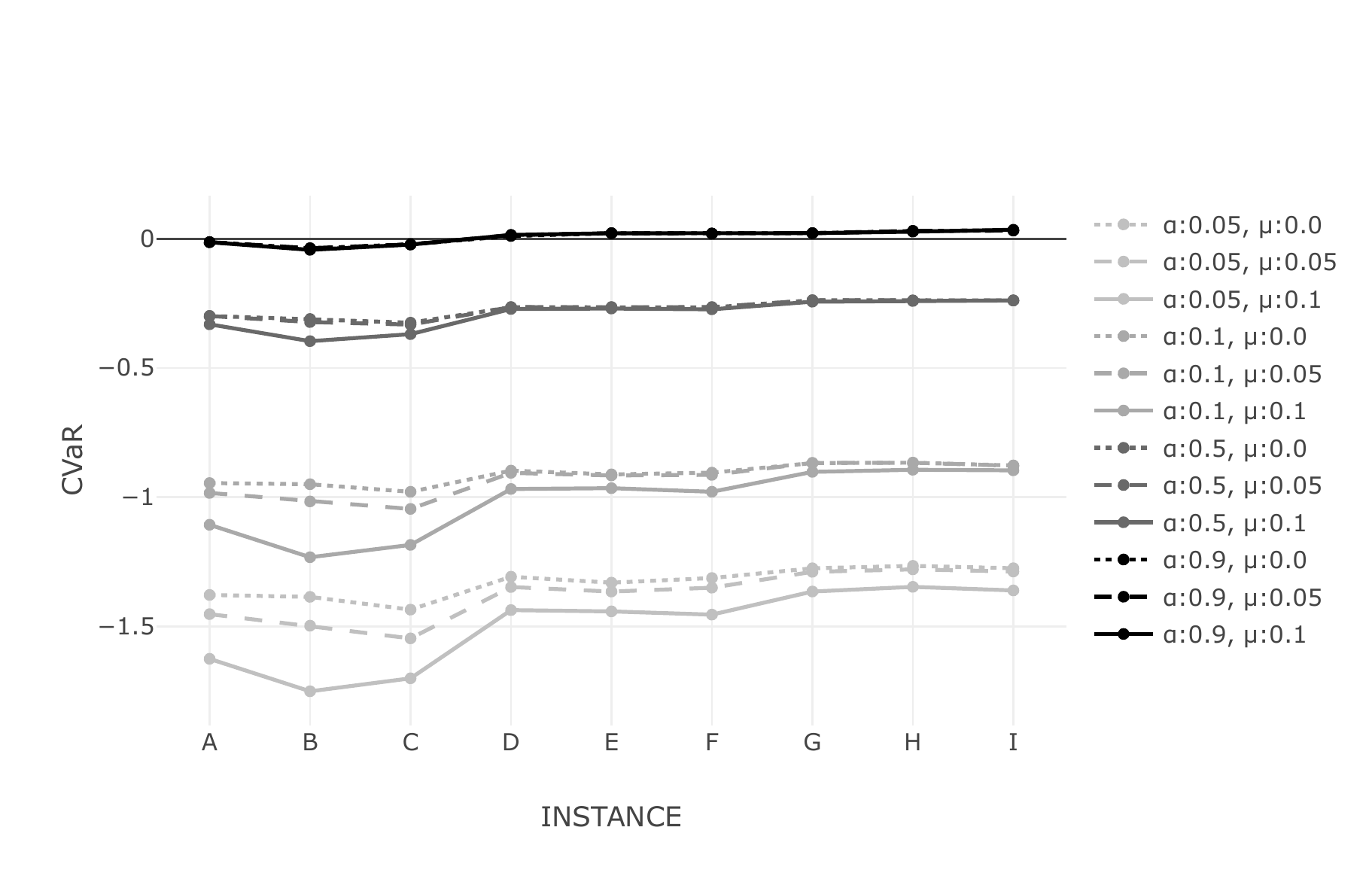}}

\caption{{\footnotesize Values of the CVaR for \textbf{BLIFP}, for different $\alpha$ and $\mu_0$ levels}}\label{Grap:CVaR_M1}
\end{figure}

{Figure \ref{Grap:BankProfit_M1} compares, with a similar organization as Figure \ref{Grap:CVaR_M1}, the broker-dealer net profit for different investor's risk profiles. Analogously, Figure \ref{Grap:ExpectedReturn_M1} represents the expected return for the investor.

We observe in Figure  \ref{Grap:BankProfit_M1} that the results of the broker-dealer net profit are bigger, in trend, for profiles with smaller values of $\alpha$, that is, for more risk-averse investments. In addition, we also show in Figure \ref{Grap:ExpectedReturn_M1} that, in general, bigger expected returns are obtained for higher values of $\alpha$. The reason for this is that by increasing $\alpha$ one is considering a wider range of values to compute the CVaR, and then the result is a  value closer to the expected return (note that when $\alpha=1$ the expected return is equal to the CVaR).}

\begin{figure}[H]
\centering

\fbox{\includegraphics[scale=0.7]{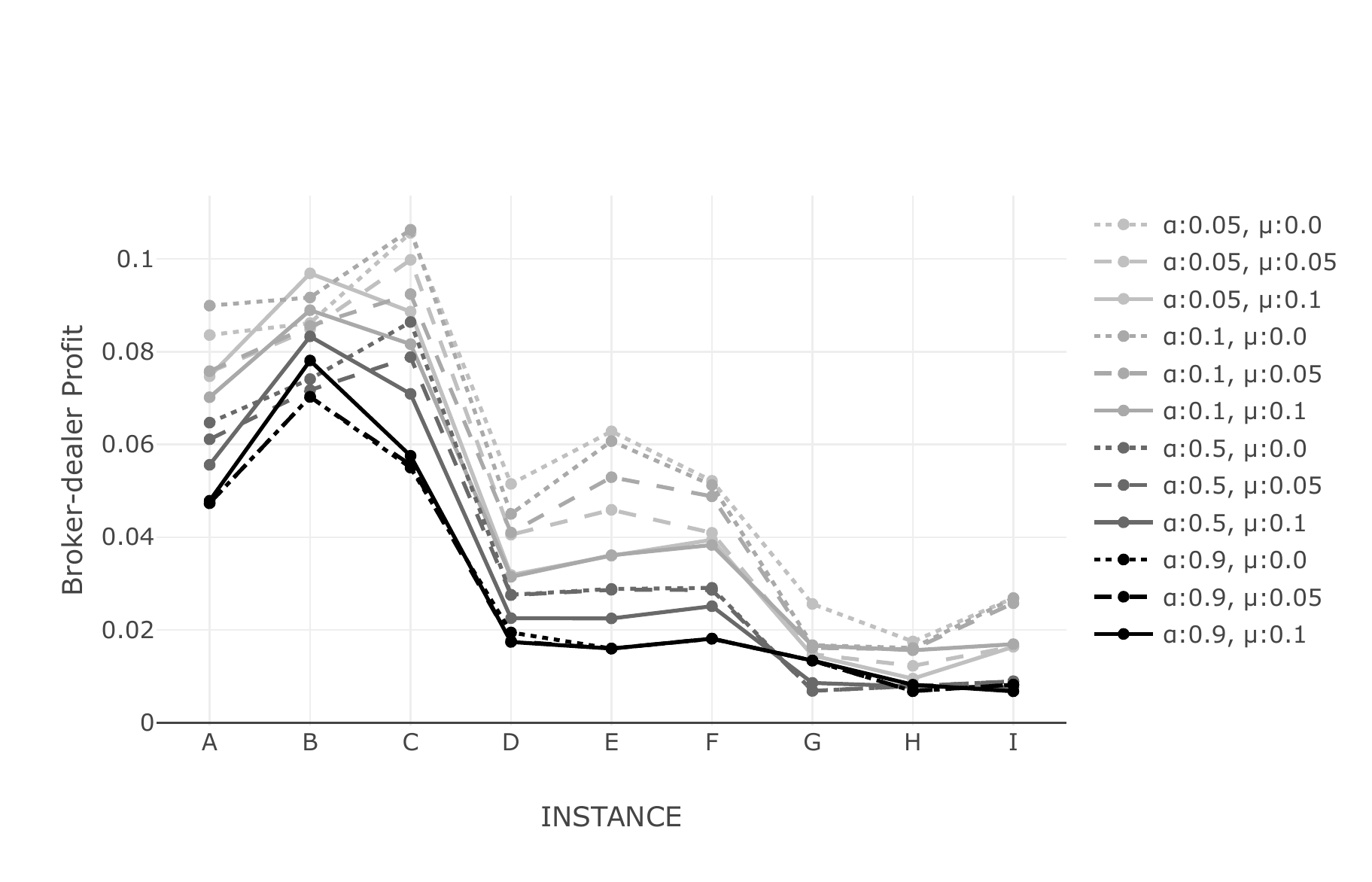}}

\caption{{\footnotesize Values of the broker-dealer profit for \textbf{BLIFP}, for different values $\alpha$ and $\mu_0$ levels}}\label{Grap:BankProfit_M1}
\end{figure}

\begin{figure}[H]
\centering

\fbox{\includegraphics[scale=0.7]{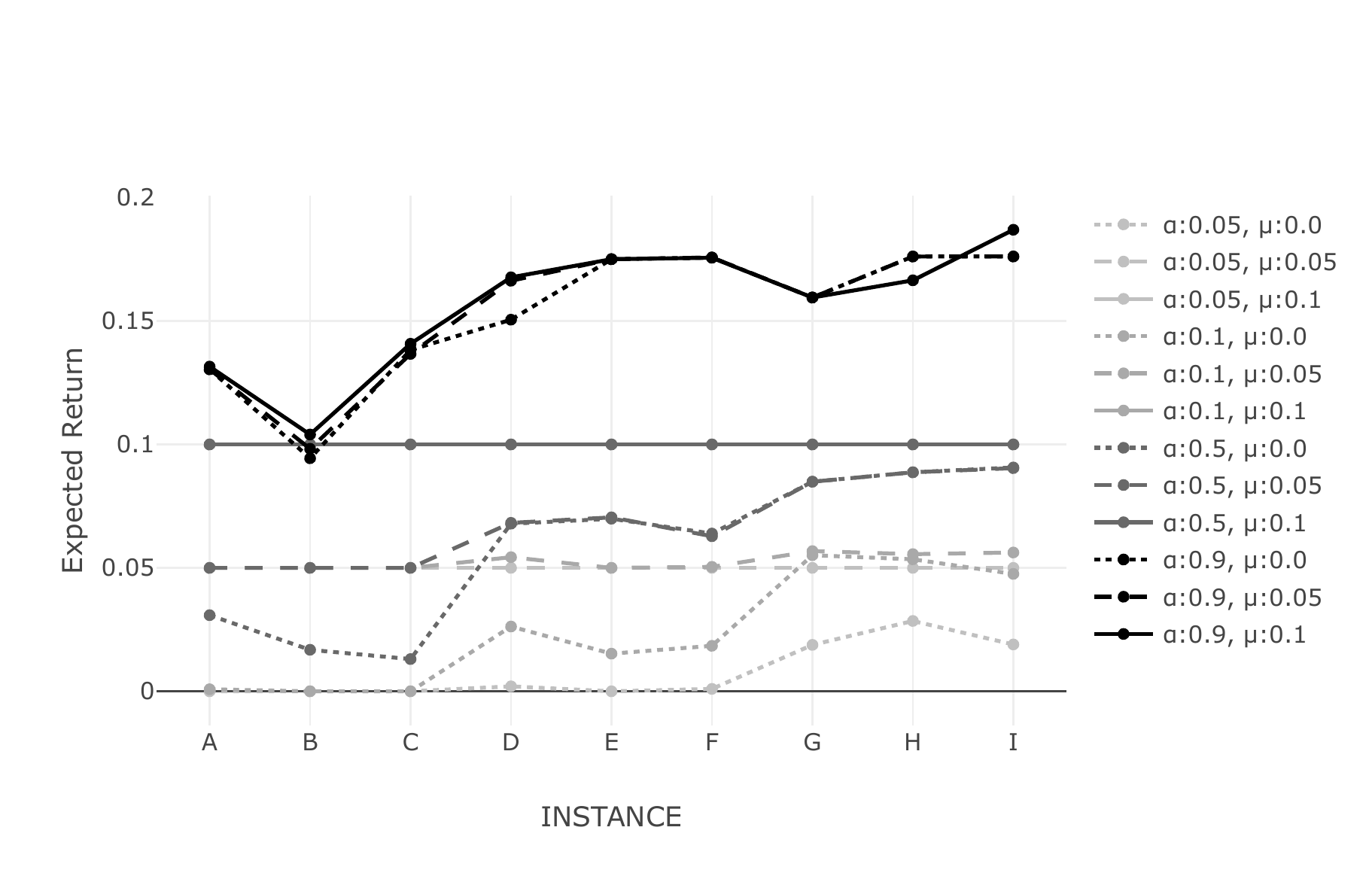}}

\caption{{\footnotesize Values of the expected return for \textbf{BLIFP}, for different $\alpha$ and $\mu_0$ levels}}\label{Grap:ExpectedReturn_M1}
\end{figure}

Finally, to conclude with the analysis of \textbf{BLIFP}, we remark that the smaller the cardinality of the set $B$ the better the CVaR and expected returns for the investor, but the worse the broker-dealer net profit. This is expected since we are reducing the number of securities where the broker-dealer could charge transaction costs.

{We proceed next to analyze the solutions of the second {problem}, namely \textbf{ILBFP}. We observe in Figure \ref{Grap:BankProfit_M2} the same trend  that in the previous model: more risk-averse investments produce lower CVaR for the investor (left-upper figure), and bigger profits for the broker-dealer (right-upper figure), and decreasing the cardinality of the set $B$ results in a reduction of the broker-dealer profit. The behavior of expected return (lower figure) is similar to those observed in Figure \ref{Grap:ExpectedReturn_M1} for the corresponding \textbf{BLIFP}. }

\vspace*{-0.08cm}

\begin{figure}[H]
\centering

\begin{multicols}{2}

\fbox{\includegraphics[scale=0.45]{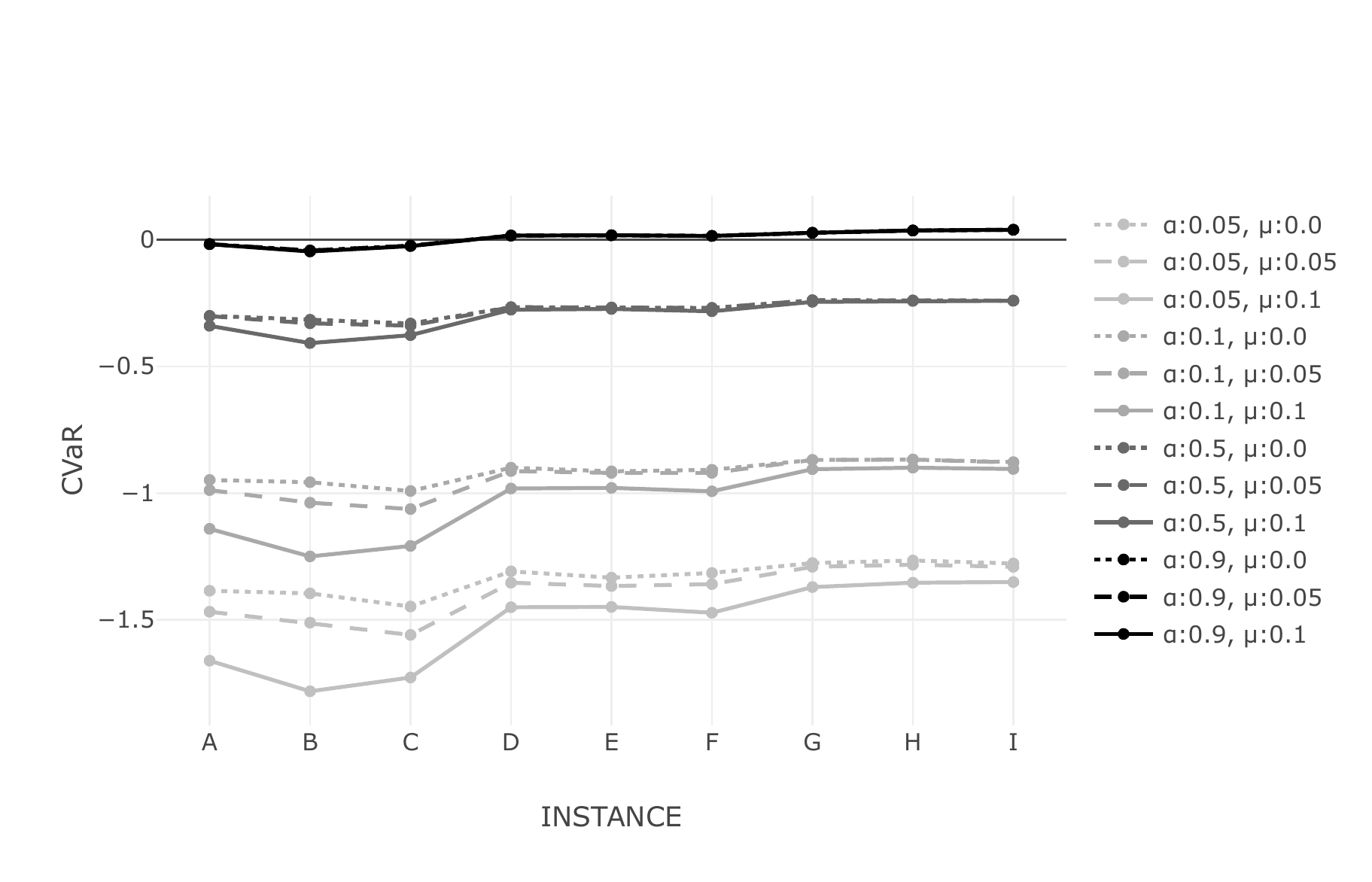}}

 \columnbreak

\fbox{\includegraphics[scale=0.45]{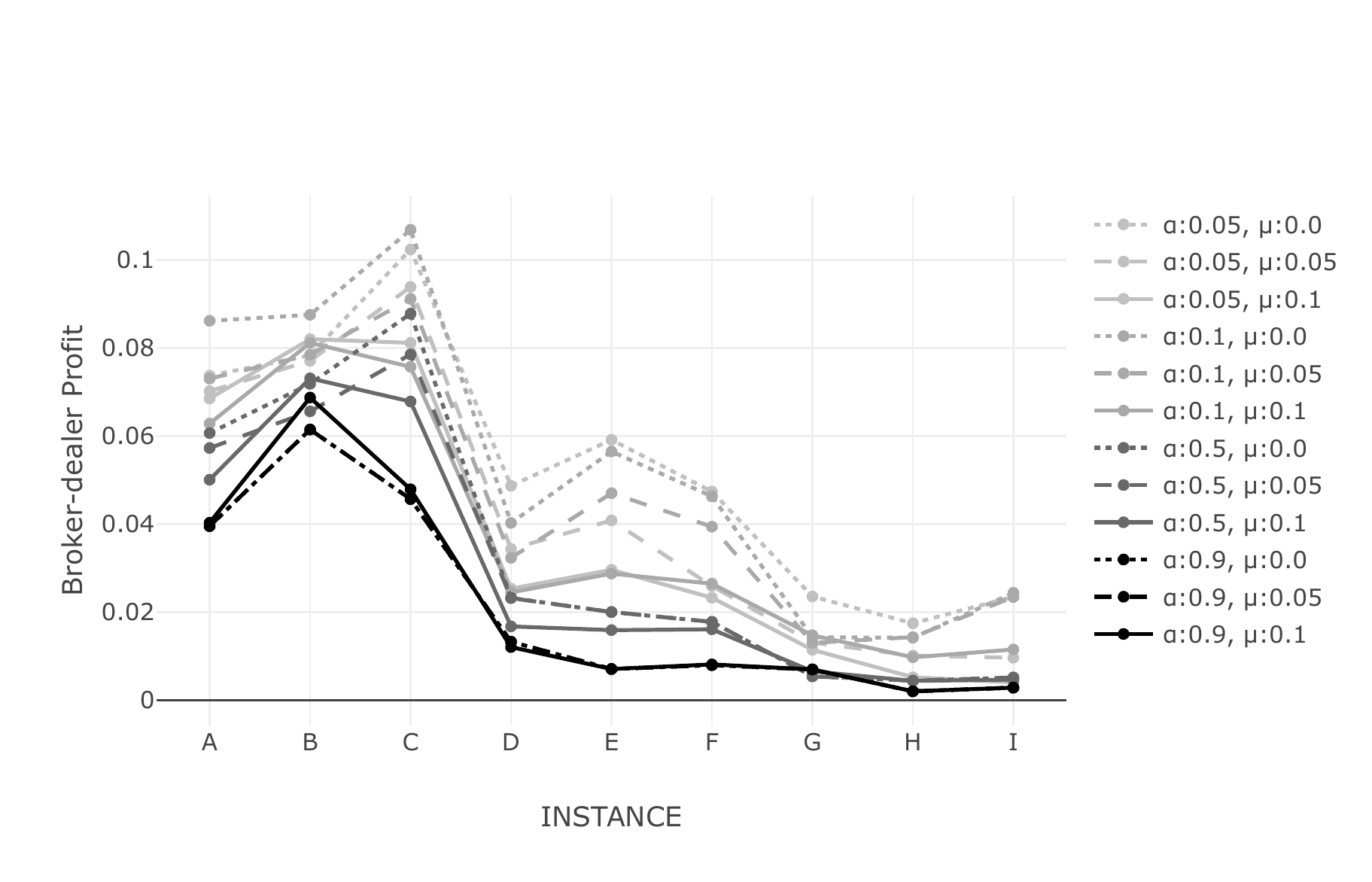}}

\end{multicols}

\begin{center}
\fbox{\includegraphics[scale=0.45]{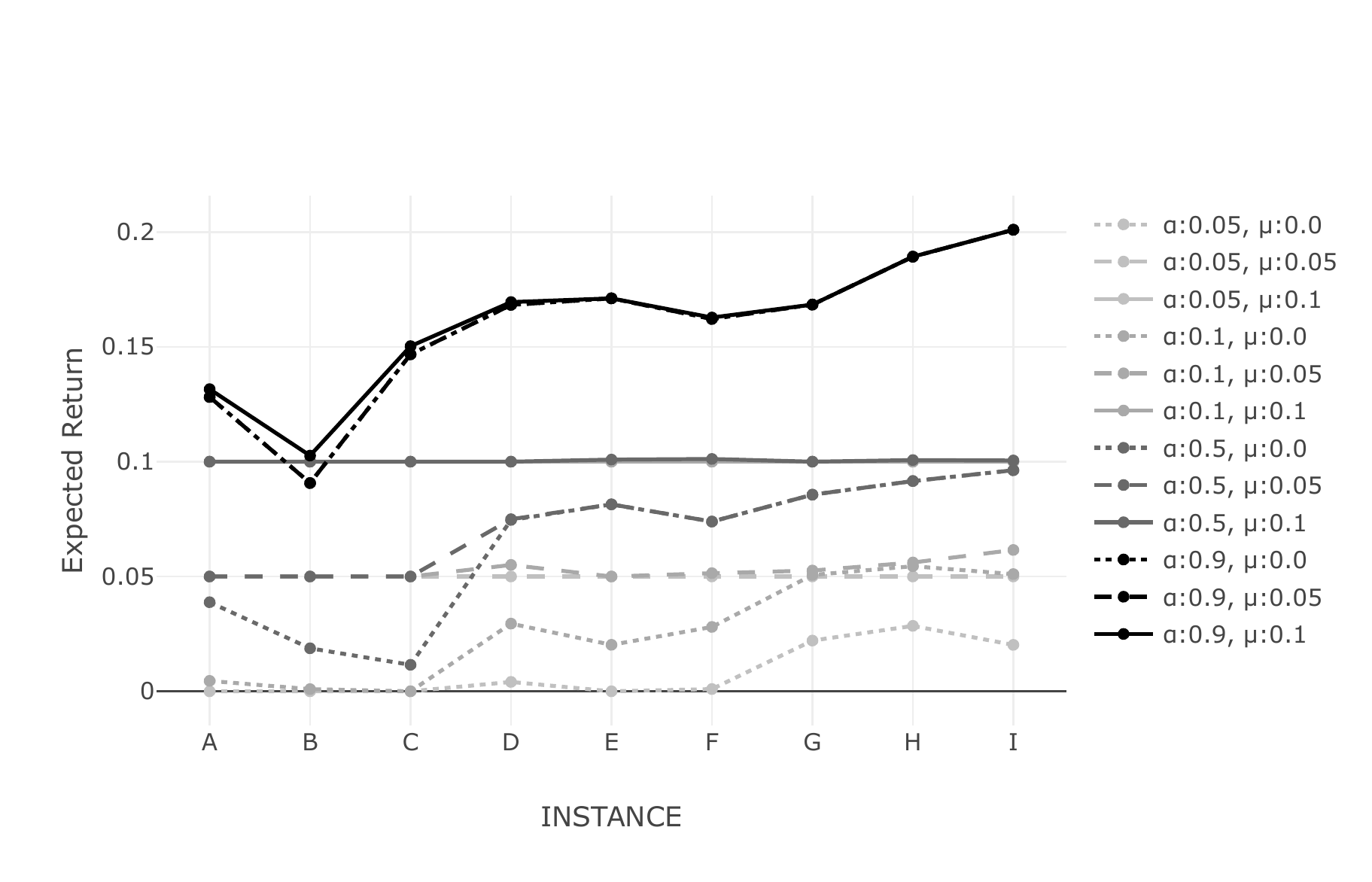}}
\end{center}


\caption{{\footnotesize {Values of the CVaR (left-above), broker-dealer profit (right-above) and expected return (below) for \textbf{ILBFP}, for different $\alpha$ and $\mu_0$ levels}}}\label{Grap:BankProfit_M2}
\end{figure}


{To finish this section, we consider the \textbf{MSWP} model. In this case, we have also included in our analysis the comparison of the objective function of this {problem, namely the broker-dealer net profit plus CVaR,} for the different risk profiles with respect to $\mu_0$ and $\alpha$, and type of market ($A,\ldots,I$). It can be seen in {the upper-right frame of} Figure \ref{Grap:Sum_M3}, that the objective value increases with the value of $\alpha$. The same {trend} is observed for the CVaR and the expected return (left figures). However, regarding the broker-dealer profit {we could not detect a clear pattern.}

The interested reader is referred {to the appendix,  that includes all comparisons and graphical outputs gathered in our study. Furthermore, one can find there a discrete Pareto front of \textbf{MSWP}} for different values of the parameter $\xi$.}

\begin{figure}[H]
\centering

\begin{multicols}{2}

\fbox{\includegraphics[scale=0.45]{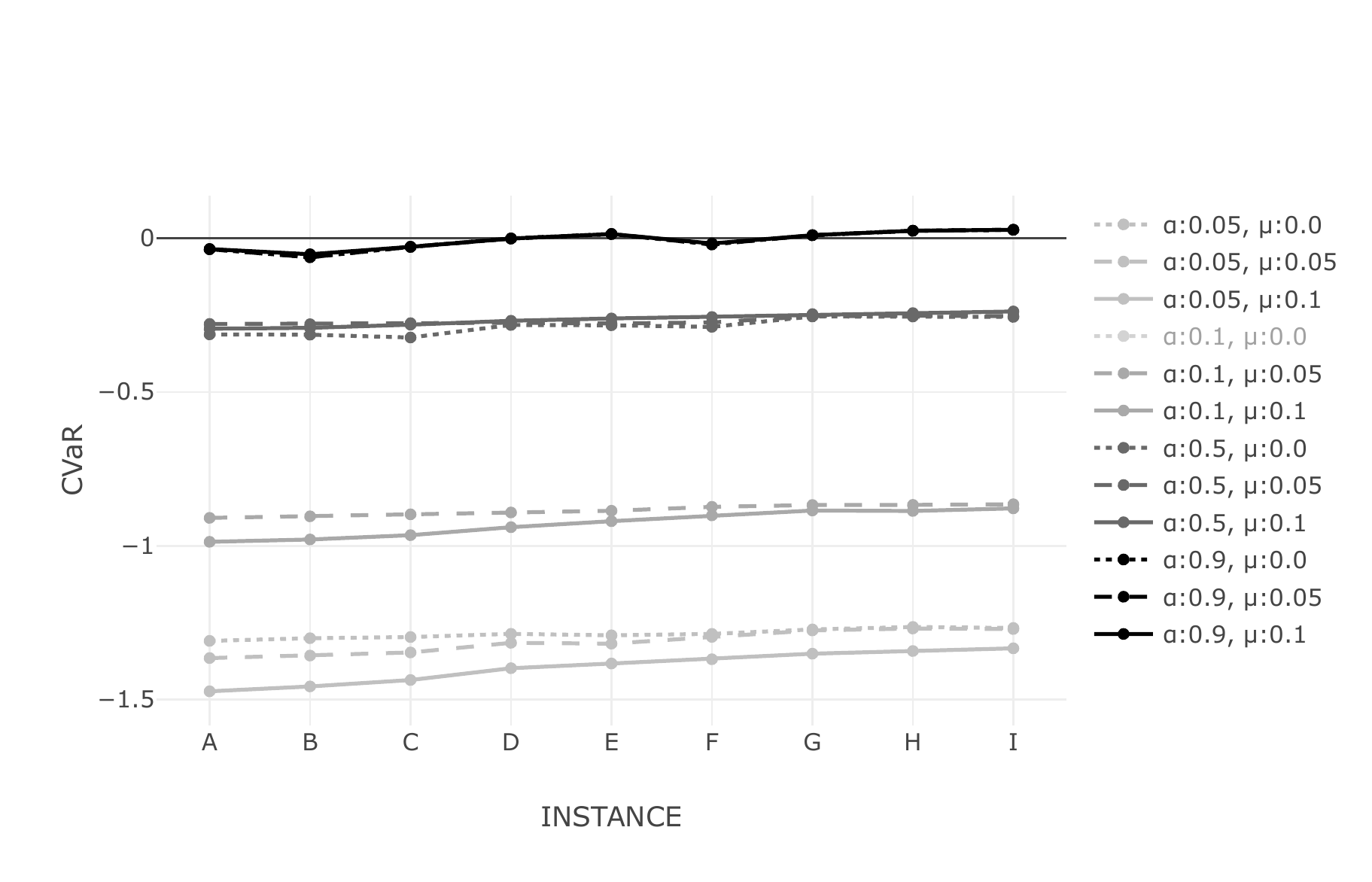}}

 \columnbreak

\fbox{\includegraphics[scale=0.45]{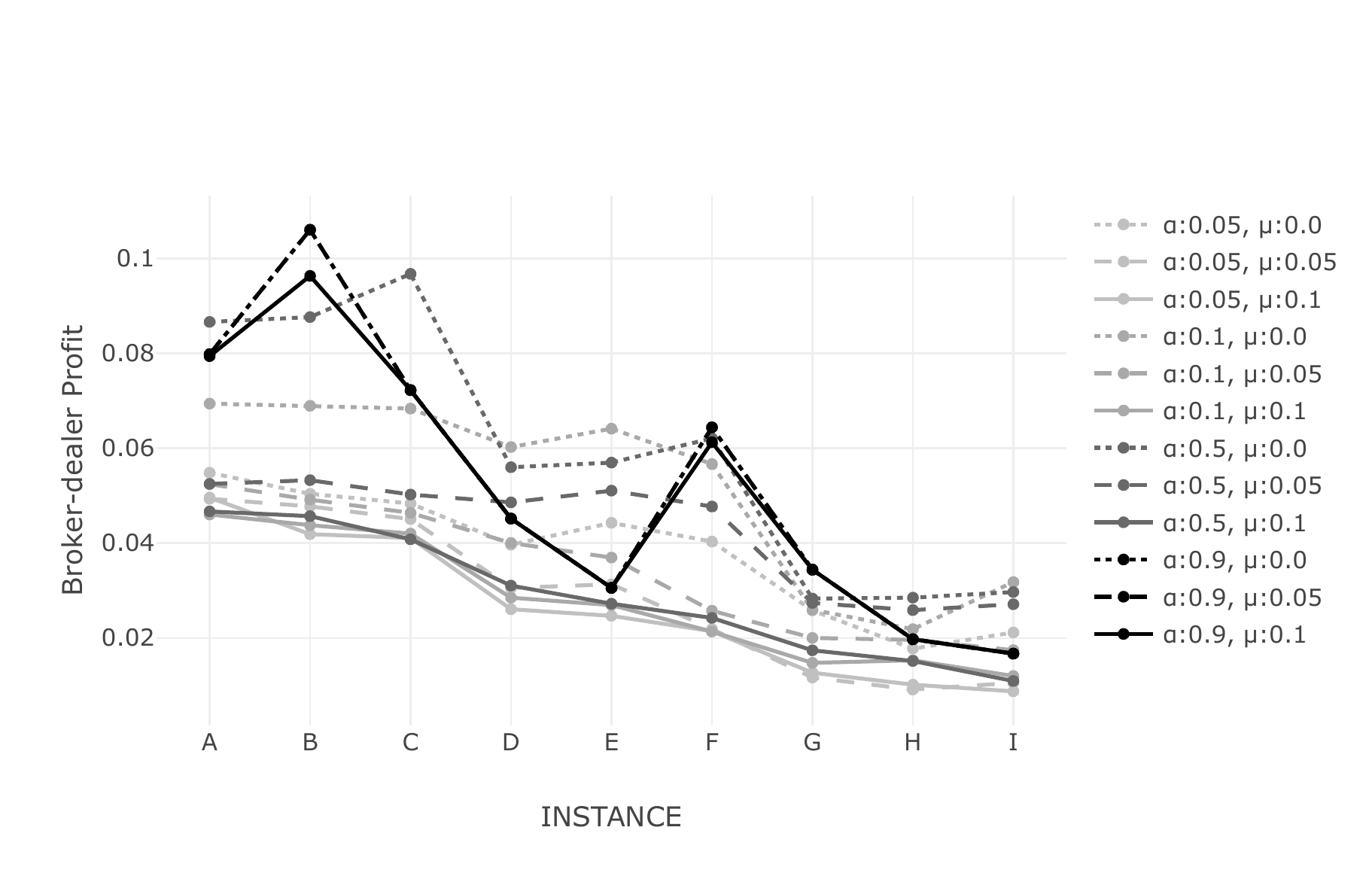}}

\end{multicols}

\begin{multicols}{2}

\begin{center}
\fbox{\includegraphics[scale=0.45]{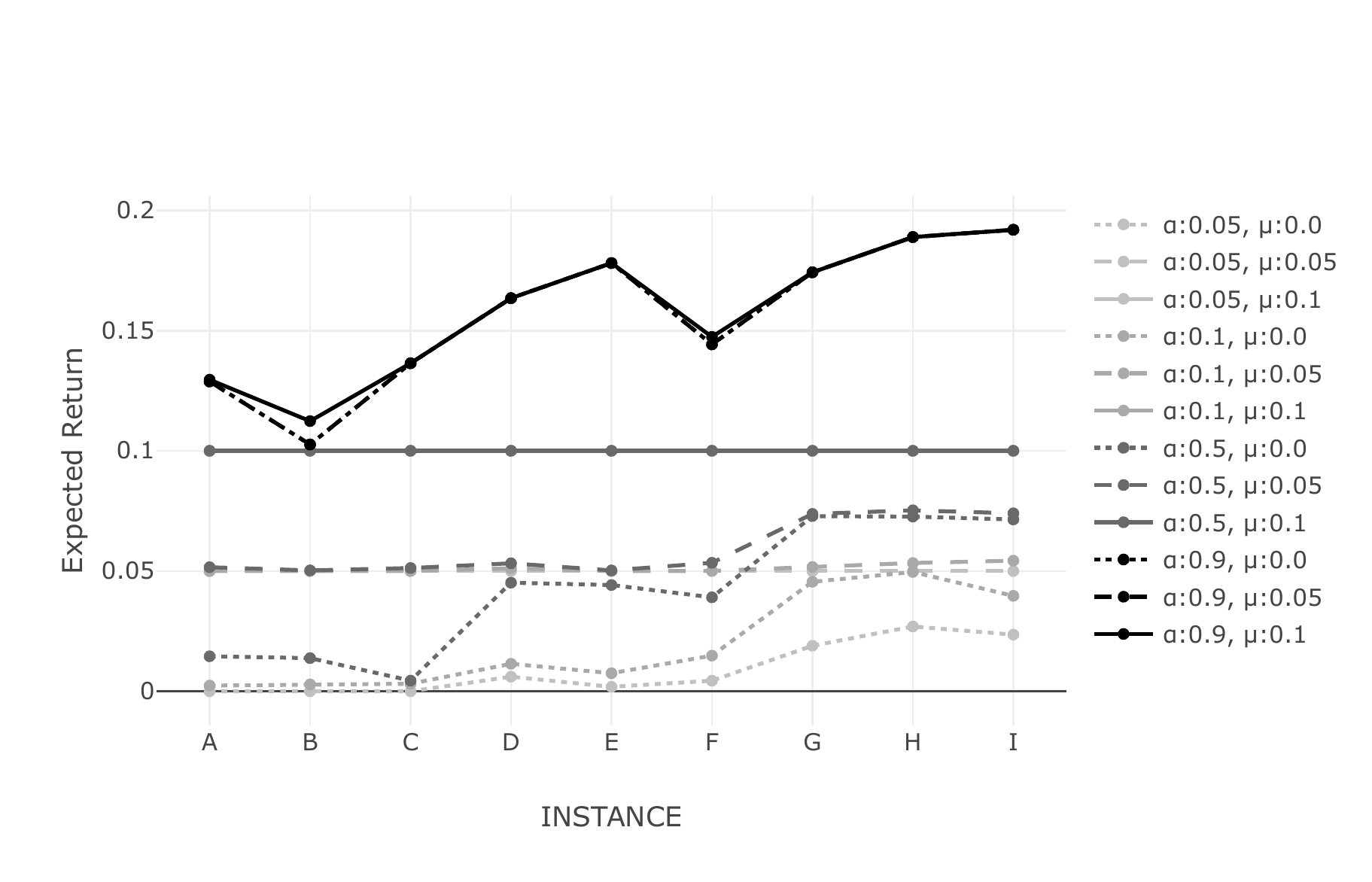}}
\end{center}

\begin{center}
\fbox{\includegraphics[scale=0.34]{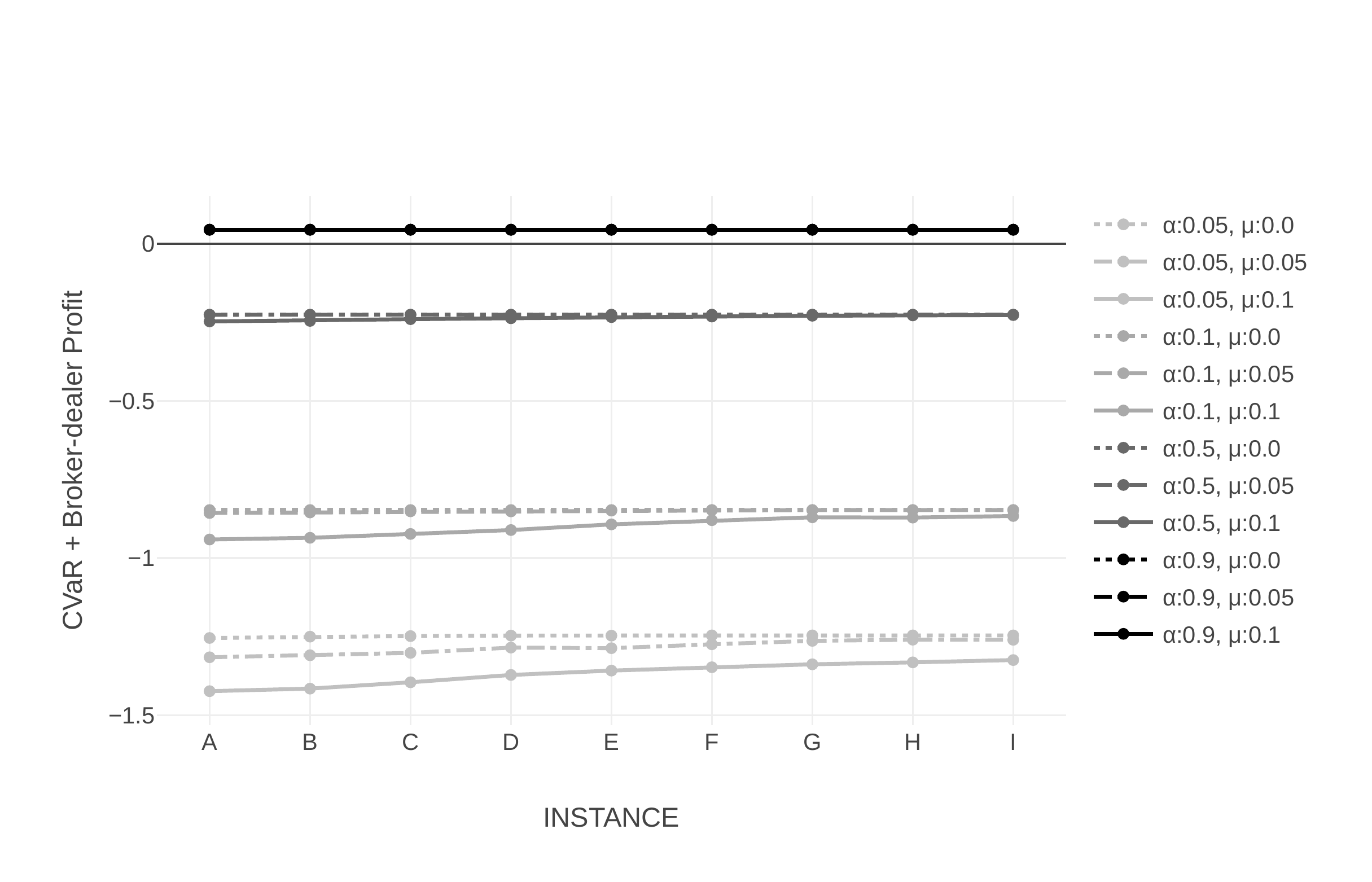}}
\end{center}

\end{multicols}
\caption{{\footnotesize {Values of the CVaR (upper-left), broker-dealer profit (upper-right), expected return (lower-left) and objective value (lower-right) for different $\alpha$ and $\mu_0$ levels in \textbf{MSWP},}}}\label{Grap:Sum_M3}
\end{figure}

%
%
%
%
%

\subsection{Comparing solutions across problems}

This last section of the computational results is devoted to comparing the solutions provided for the three problems considered in this paper, namely \textbf{BLIFP}, \textbf{ILBFP} and \textbf{MSWP}. The goal is to analyze the solution across {problems} with respect to the goals of the two parties: broker-dealer net profit, CVaR levels and expected returns. Due to page length limitations in the paper version, we have included in our figures only some comparisons for certain risk profiles. The interested reader is referred again to the appendix, where we report comparisons for a broader range of risk profiles.

\vspace{10mm}

\begin{figure}[H]
\centering

\begin{multicols}{2}



\fbox{\includegraphics[scale=0.45]{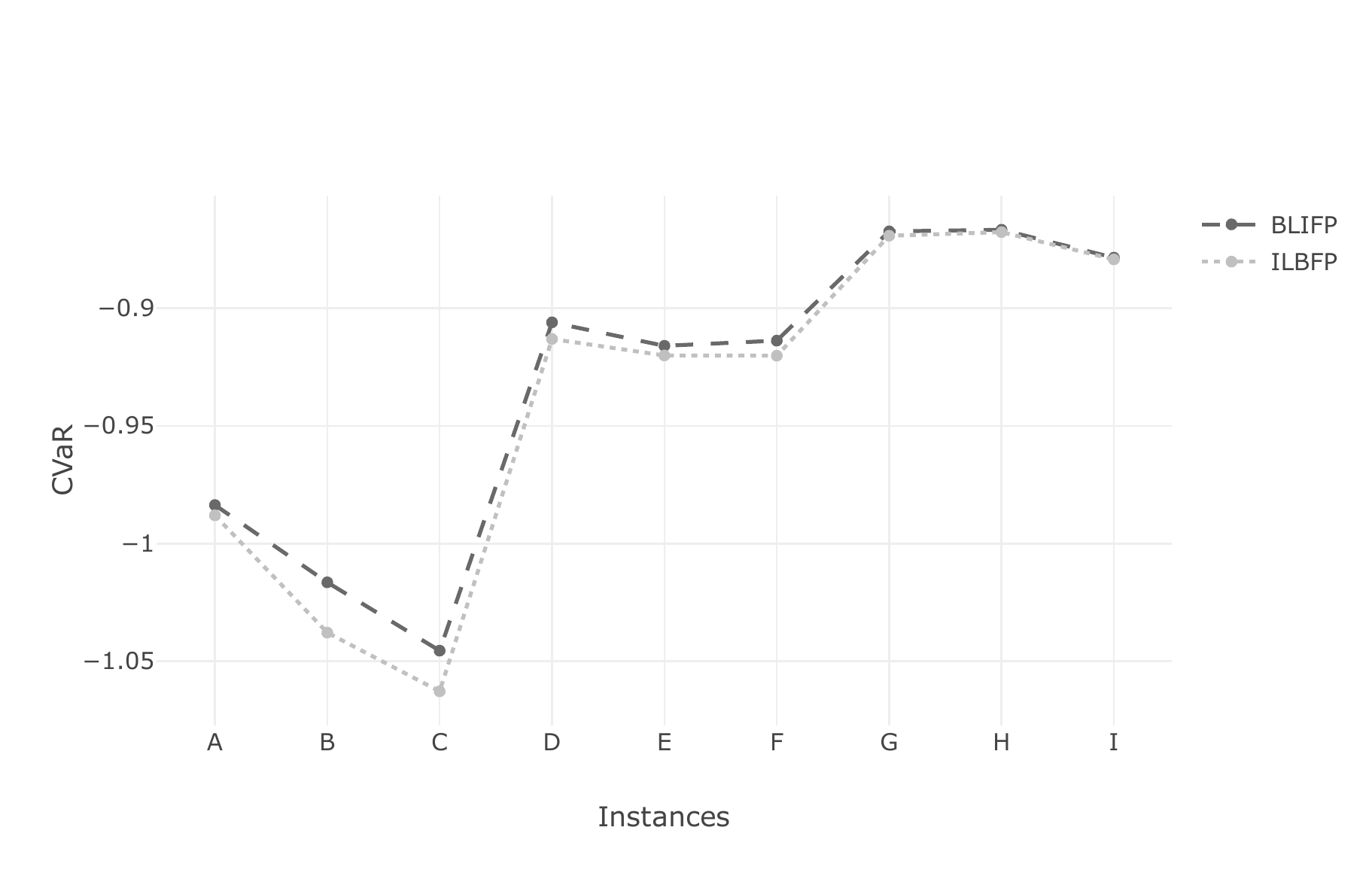}}

\fbox{\includegraphics[scale=0.45]{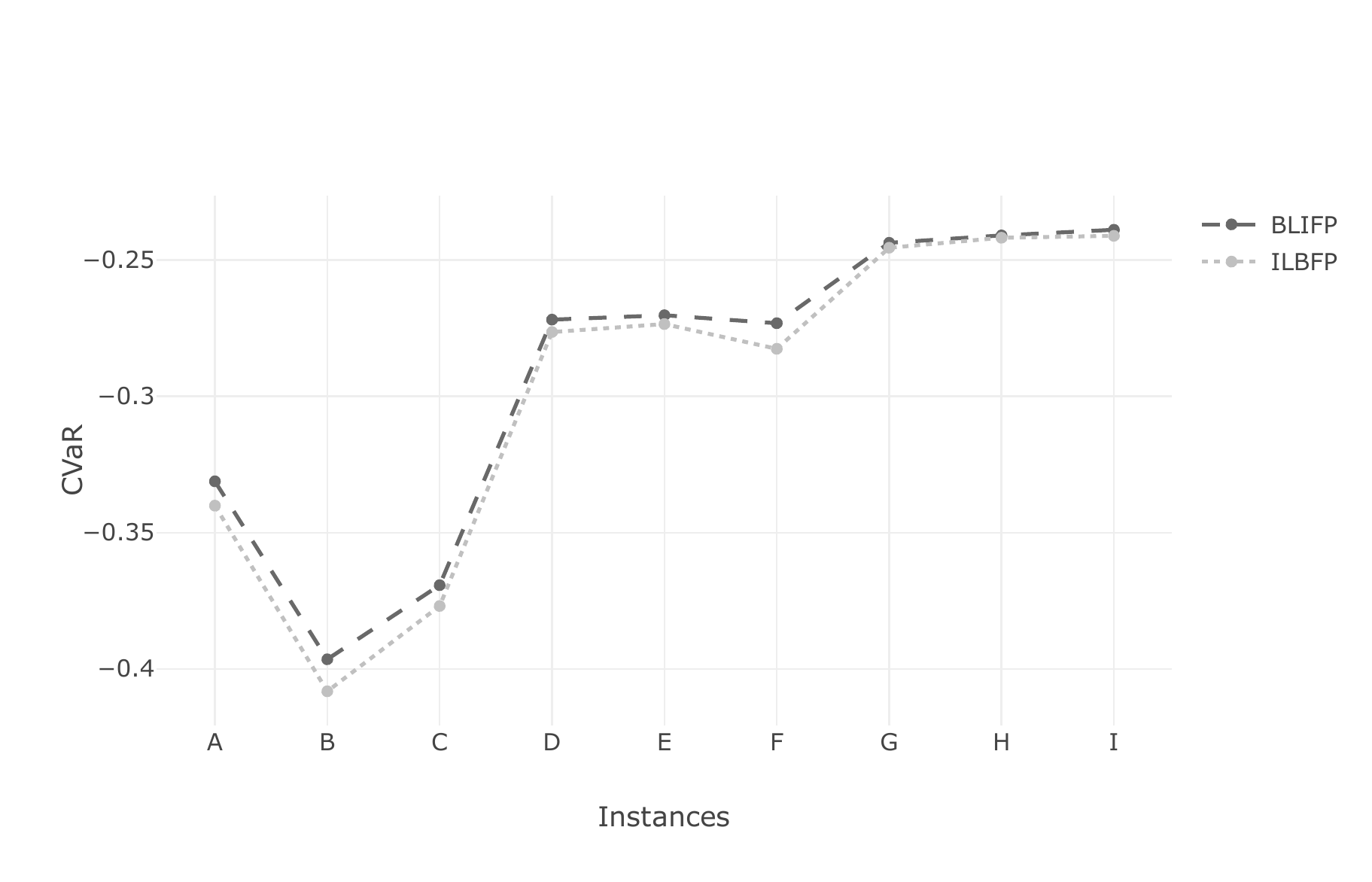}}

\end{multicols}
\caption{ {\footnotesize {Values of the CVaR for {\textbf{BLIFP}} and {\textbf{ILBFP}} for $\alpha=0.1$ and $\mu_0=0.05$ (left) and for $\alpha=0.5$ and $\mu_0=0.1$ (right)}}}\label{Grap:CVaR_compare}
\end{figure}

Figure \ref{Grap:CVaR_compare} shows a comparison of the CVaR values attained in \textbf{BLIFP} and \textbf{ILBFP} for different risk profiles {($\alpha=0.1$ and $\mu_0=0.05$  and $\alpha=0.5$ and $\mu_0=0.1$, in the right and left figures, respectively)}. We can observe in Figure \ref{Grap:CVaR_compare} that for each risk profile, the CVaR values are always higher in \textbf{BLIFP} than in \textbf{ILBFP}. Analogously, Figure \ref{Grap:BankProfit_compare} compares the values of the broker-dealer profit for the two hierarchical problems. It is also remarkable that \textbf{BLIFP} always results in higher profit values for each risk profile and all type of instances. In these comparisons, we do not include the values for the social welfare problem because they are not comparable due to the existence of multiple solutions (with the same value for the objective function but a very different balance between the distribution of the CVaR and the broker-dealer profit). As we mentioned above, we emphasize that in all our experiments, \textbf{BLIFP} always gives higher profit for the broker-dealer and better CVaR for the investor than \textbf{ILBFP}. {In this regard,  it seems beneficial for the two parties to accept that the investor knows the transaction costs on the securities before setting his portfolio.}

\begin{figure}[H]
\centering

\begin{multicols}{2}



\fbox{\includegraphics[scale=0.45]{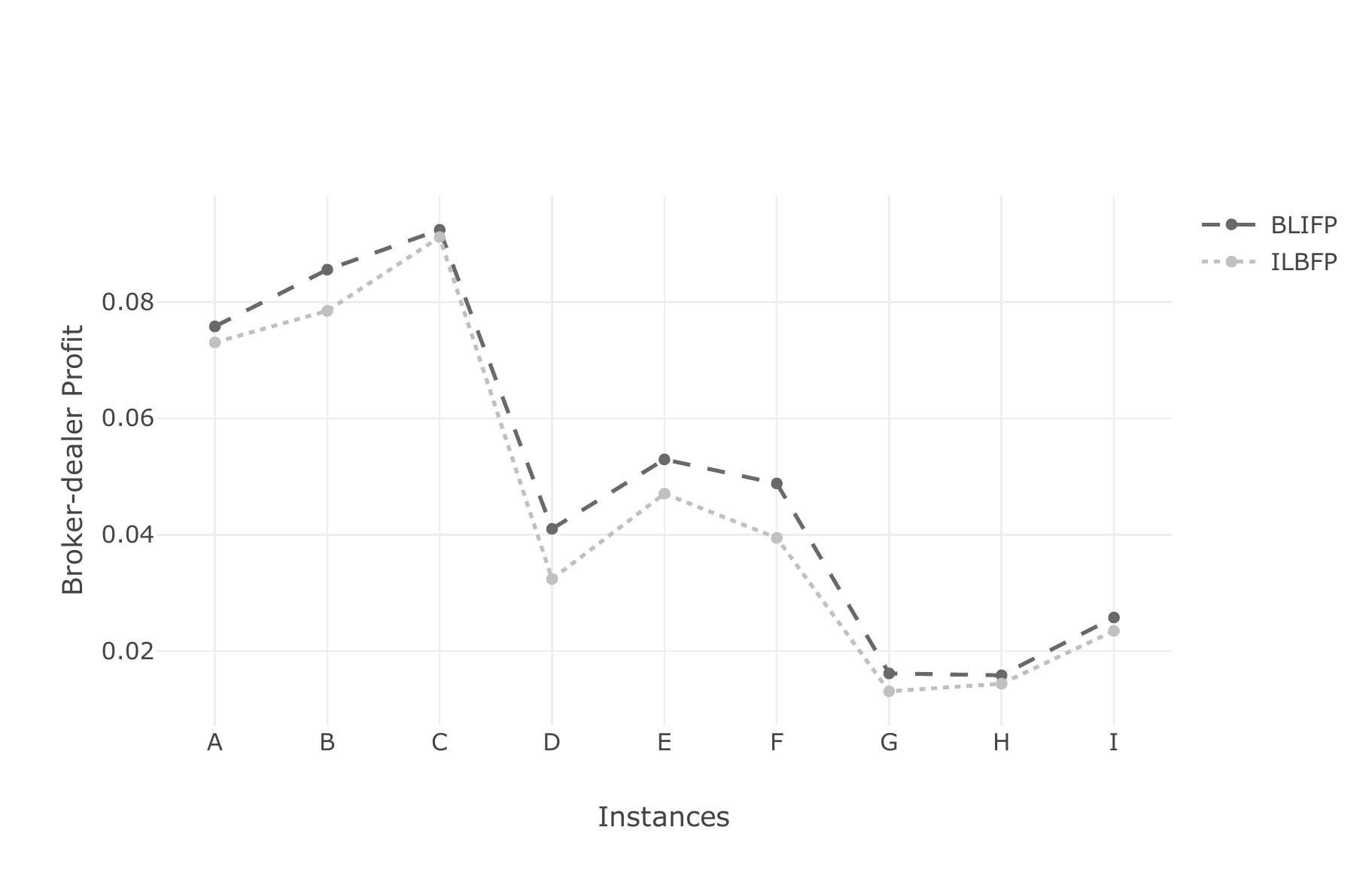}}

\fbox{\includegraphics[scale=0.45]{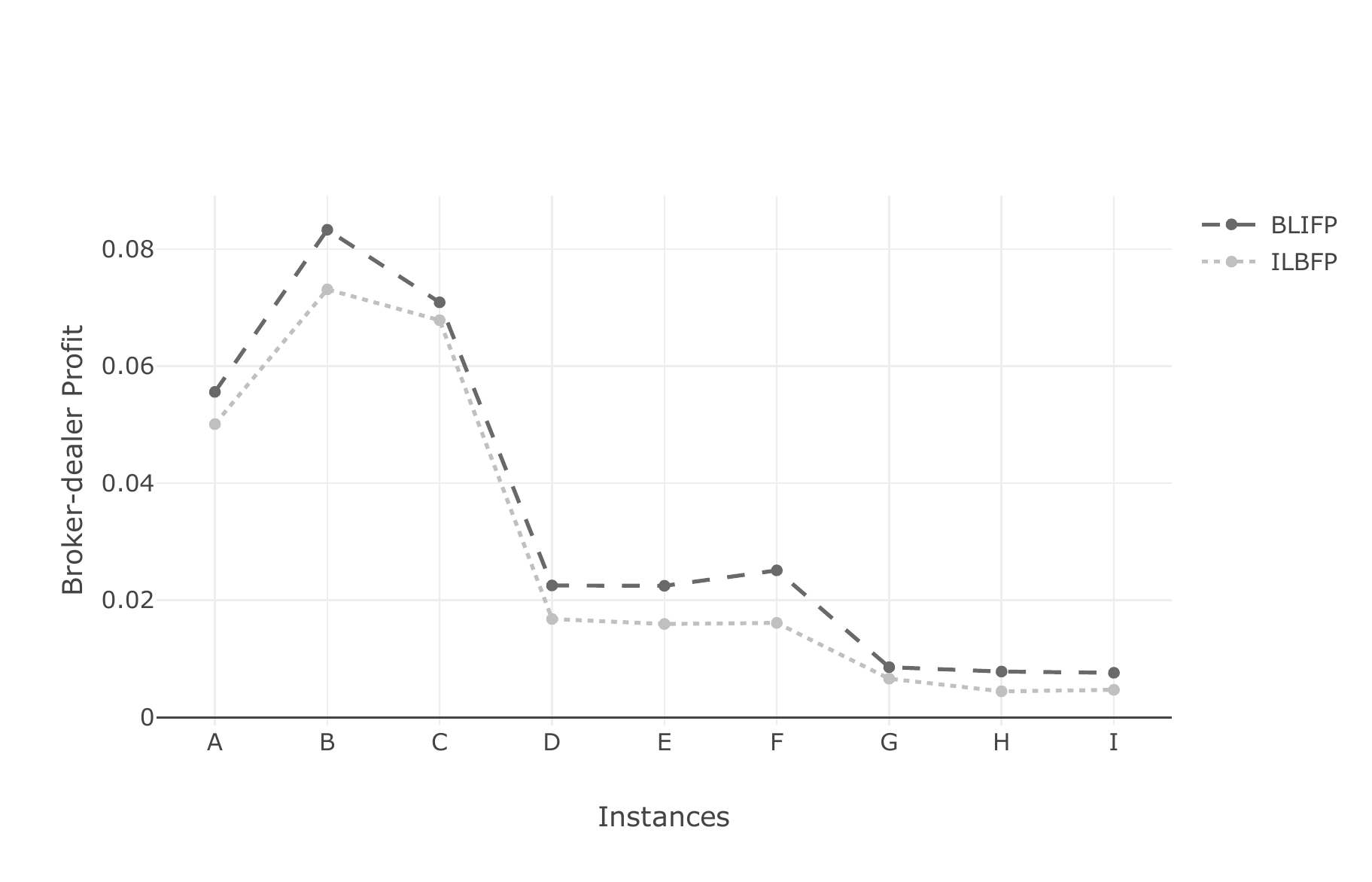}}

\end{multicols}
\caption{ {\footnotesize {Values of the broker-dealer profit for  {\textbf{BLIFP}} and {\textbf{ILBFP}},  for $\alpha=0.1$ and $\mu_0=0.05$ (left) and for $\alpha=0.5$ and $\mu_0=0.1$ (right)}}}\label{Grap:BankProfit_compare}
\end{figure}

The last comparisons across models refer to the value of the sum of {broker-dealer profit plus the CVaR of the investor}, in Figure \ref{Grap:Sum_compare}, and the expected return value, in Figure \ref{Grap:ExpectedReturn_compare}. These two figures show the corresponding values attained by the three proposed problems, \textbf{BLIFP}, \textbf{ILBFP} and \textbf{MSWP}, for the different instances ($A,\ldots,I$) and two different risk profiles (see figures captions). As theoretically proved in Proposition \ref{prop:social_welfare}, we can observe in Figure \ref{Grap:Sum_compare} that the value of the sum of the broker-dealer profit plus the CVaR {of the investor} is always greater for the social welfare model (\textbf{MSWP}) than for the other two hierarchical problems, namely \textbf{BLIFP} and \textbf{ILBFP}. Finally, we compare {the obtained expected return values for the three problems}. From Figure \ref{Grap:ExpectedReturn_compare}, we can not conclude any dominating relationship among the problems with respect to the expected return value and therefore, the numerical experiments do not prescribe any preference relationship regarding the expected return.

\begin{figure}[H]
\centering

\begin{multicols}{2}



\fbox{\includegraphics[scale=0.45]{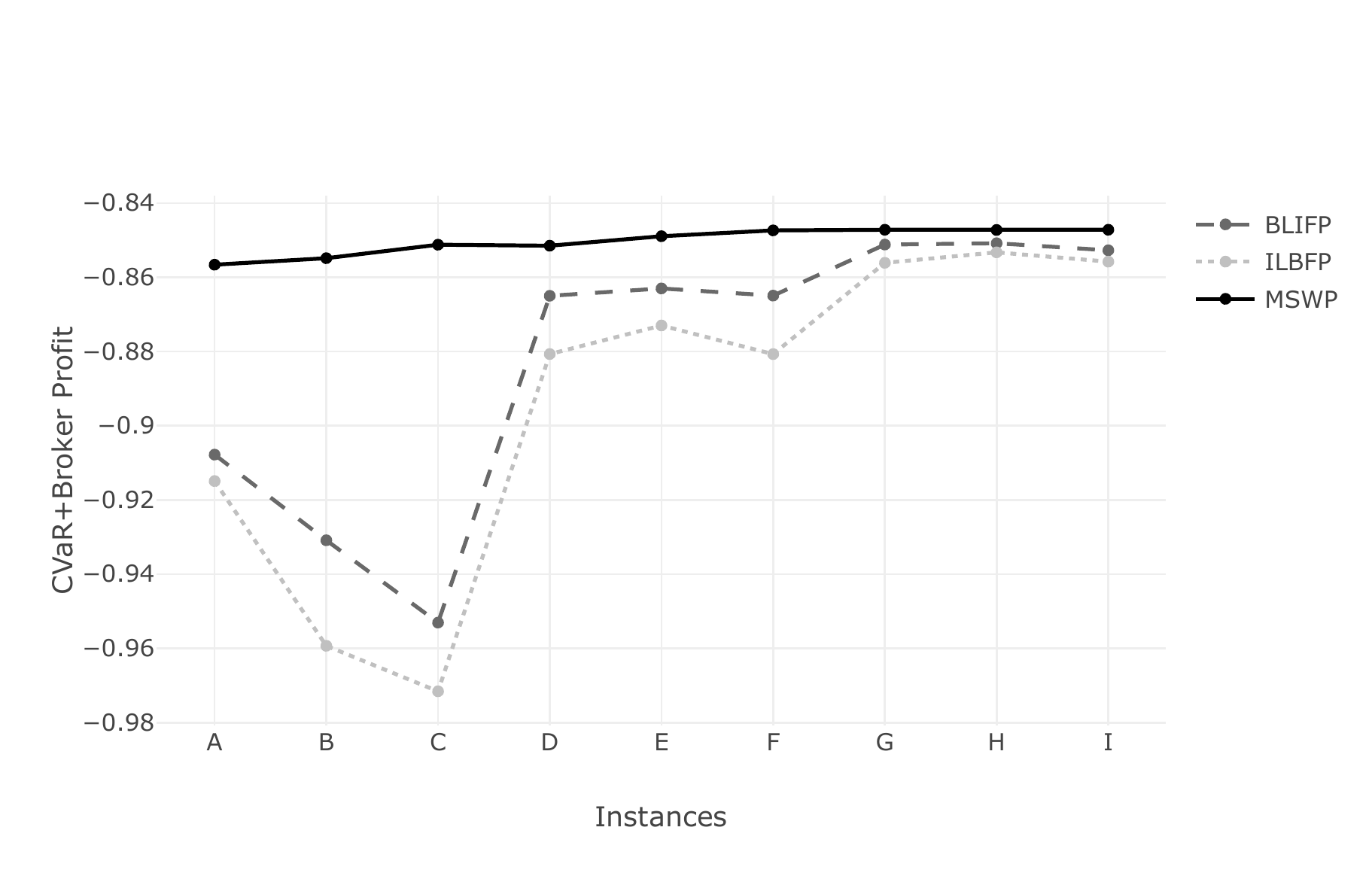}}

\fbox{\includegraphics[scale=0.45]{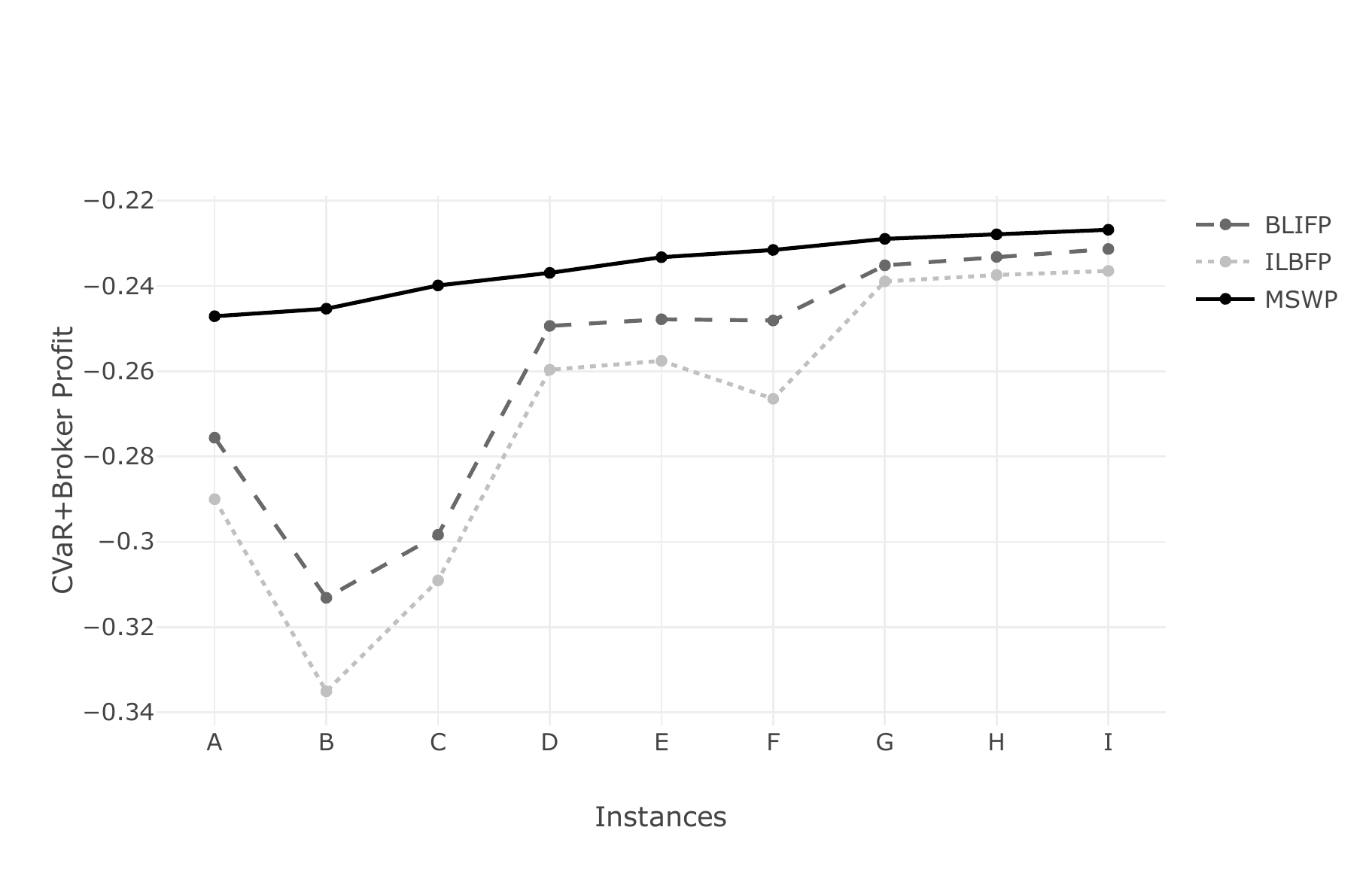}}

\end{multicols}
\caption{ {\footnotesize {Values of the broker-dealer profit + CVaR for the three {problems}, for $\alpha=0.1$ and $\mu_0=0.05$ (left) and for $\alpha=0.5$ and $\mu_0=0.1$ (right)}}}\label{Grap:Sum_compare}
\end{figure}

\begin{figure}[H]
\centering

\begin{multicols}{2}

\fbox{\includegraphics[scale=0.45]{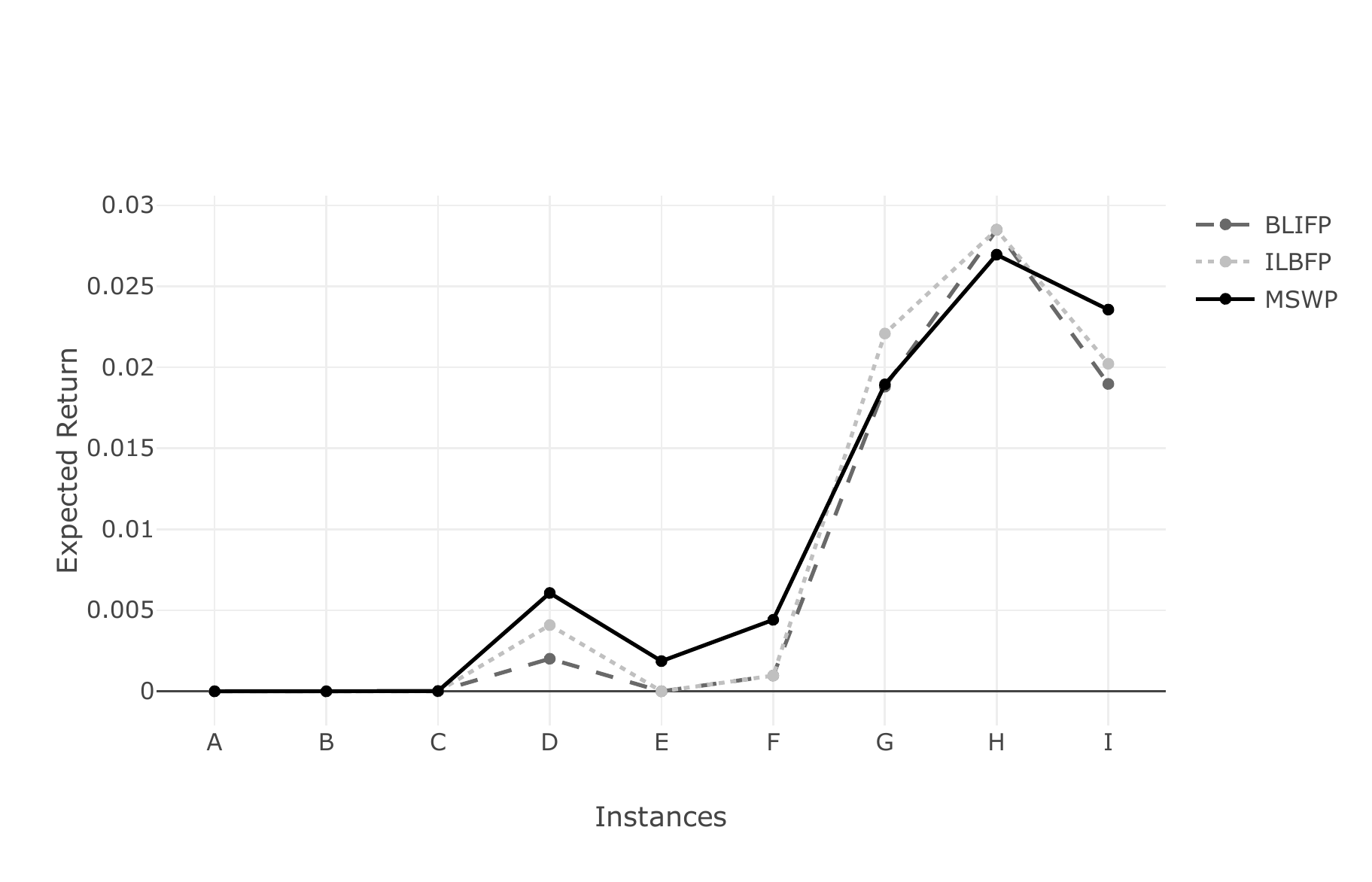}}

 \columnbreak

\fbox{\includegraphics[scale=0.45]{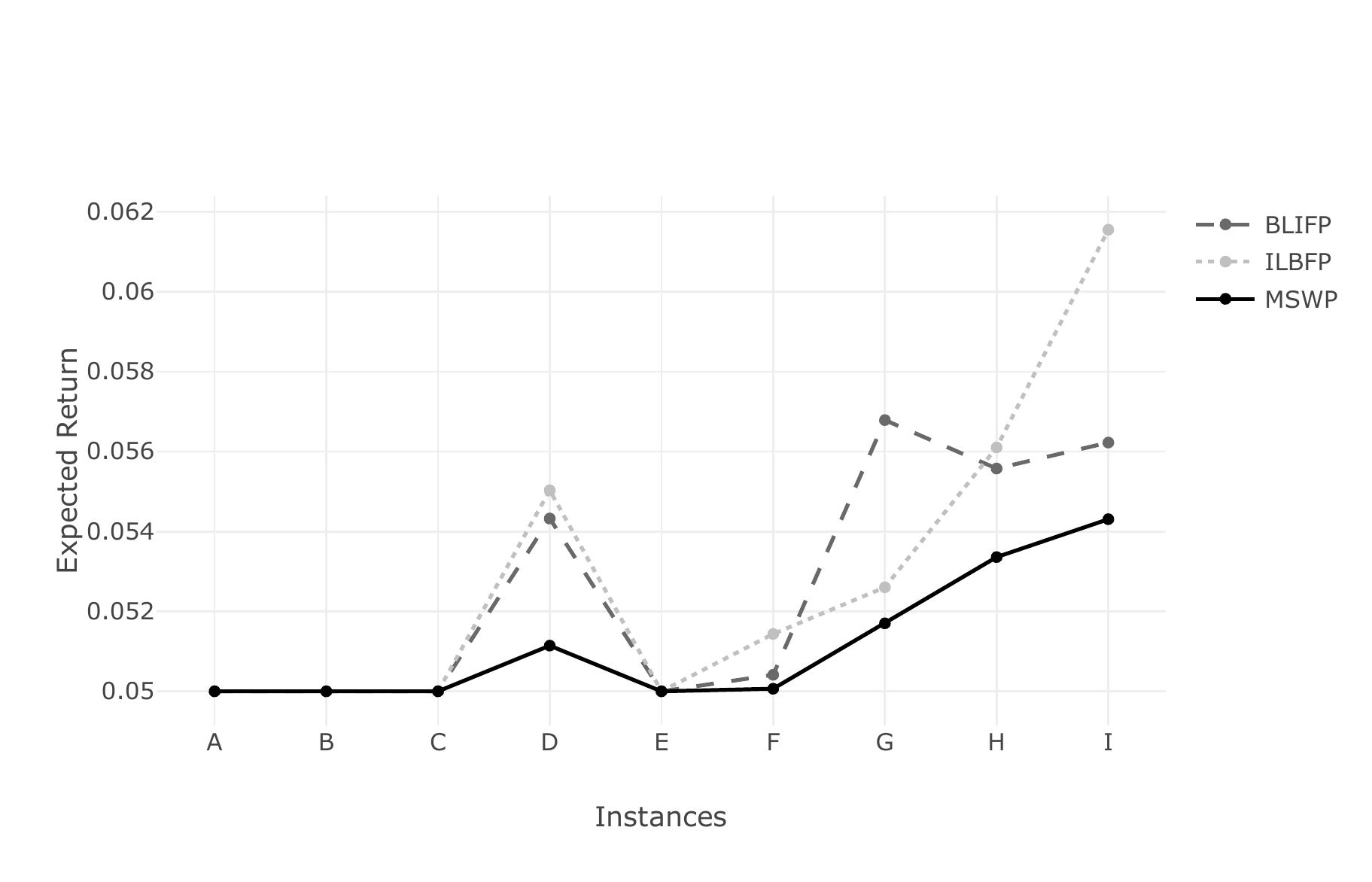}}


\end{multicols}
\caption{ {\footnotesize {Values of the expected return for the three {problems}, for $\alpha=0.05$ and $\mu_0=0.0$,  (left) and for $\alpha=0.1$ and $\mu_0=0.05$ (right)}}}\label{Grap:ExpectedReturn_compare}
\end{figure}

\section{Concluding remarks and extensions}\label{sec:Conclusions}

We have presented {three} single-period portfolio optimization problem{s} with transaction costs, considering two different decision-makers: the investor and the financial intermediary. Including the financial intermediaries {(broker-dealers)} as decision-makers leads to the incorporation of the transaction costs as decision variables in the portfolio selection problem. The action of both decision-makers was assumed to be hierarchical. We have considered the situations where each {of these decision-makers is leader and have analyzed them.} This hierarchical structure has been modeled using bilevel optimization. In addition, a social welfare model has also been studied.

In all cases, it has been assumed that the broker-dealer has to choose the unit transaction costs, for each security, from a discrete set of {possible costs}, maximizing its benefits, and that the investor aims to minimize the risk (optimizing his CVaR), ensuring a given expected return. Considering continuous sets of possible {values for the transaction costs} could be an interesting future research line.

In the considered models we assumed proportional transactions cost; however, other transaction costs structures such as fixed transaction costs or convex piecewise linear costs have been considered in the literature (for further details on transaction costs structures we refer the reader to \cite{man15_2}). These costs structures could be incorporated in our models by slightly modifying the resolution methods  and increasing the complexity of problem-solving. For instance, in order to incorporate fixed fees and commissions, we should include some binary variables determining whether the investor chooses a security or not, and then accounting for its contribution to the transaction costs. The general tools from MILP can be adapted to solve the problem with this new structure of costs. This could be another interesting future research line.

In order to solve the three proposed {problems}, MILP and LP formulations, as well as algorithms, have been proposed. By making variations in the set of {costs}, and in the parameters to model the CVaR and the expected return, $\alpha$ and $\mu_0$, different broker-dealer and investor profiles can be considered.

In our analysis in Sections \ref{Sect:Bank-leader} and \ref{Investor-leader}, all the problems have been presented, for simplicity, with only one follower. Nevertheless, they could be easily extended to more than one. In particular, in Section \ref{Sect:Bank-leader}, the problem has been studied from the broker-dealer point of view, that is, the broker-dealer aims to maximize its benefit by assuming that once the {costs} for the securities are set, a single investor will choose his portfolio according to the described goals. We remark that the same procedure could be applied to several followers (investors). In fact, in that {problem}, $F$ different profiles of followers (risk-averse, risk-taker, etc.) could be considered, and the broker-dealer's goal would be maximizing the overall benefit for any linear function of its {costs}. This approach would allow the broker-dealer to improve the decision-making process in the cases where the same {costs} have to be set for all the investors, but different investor's profiles are considered.


{A detailed computational study has been conducted using data from the Dow Jones Industrial Average.} We have compared the solution methods, the solutions and the risk profiles within {problems}, and the solutions across {them}. From our computational experience, we have observed that the broker-dealer-leader investor-follower {problem} results in better solutions for both, the broker-dealer and the investor, in comparison with the investor-leader broker-dealer-follower {problem}. Furthermore, the social welfare model {problem}, as theoretically proved, in higher aggregated benefits.

\newpage

\appendix
\section{Appendix}

\subsection{Comparing solutions and risk profiles within {problems}}

\subsubsection*{Discrete Pareto front for \textbf{MSWP}}

\begin{figure}[H]
\centering

\begin{multicols}{2}

\fbox{\includegraphics[scale=0.35]{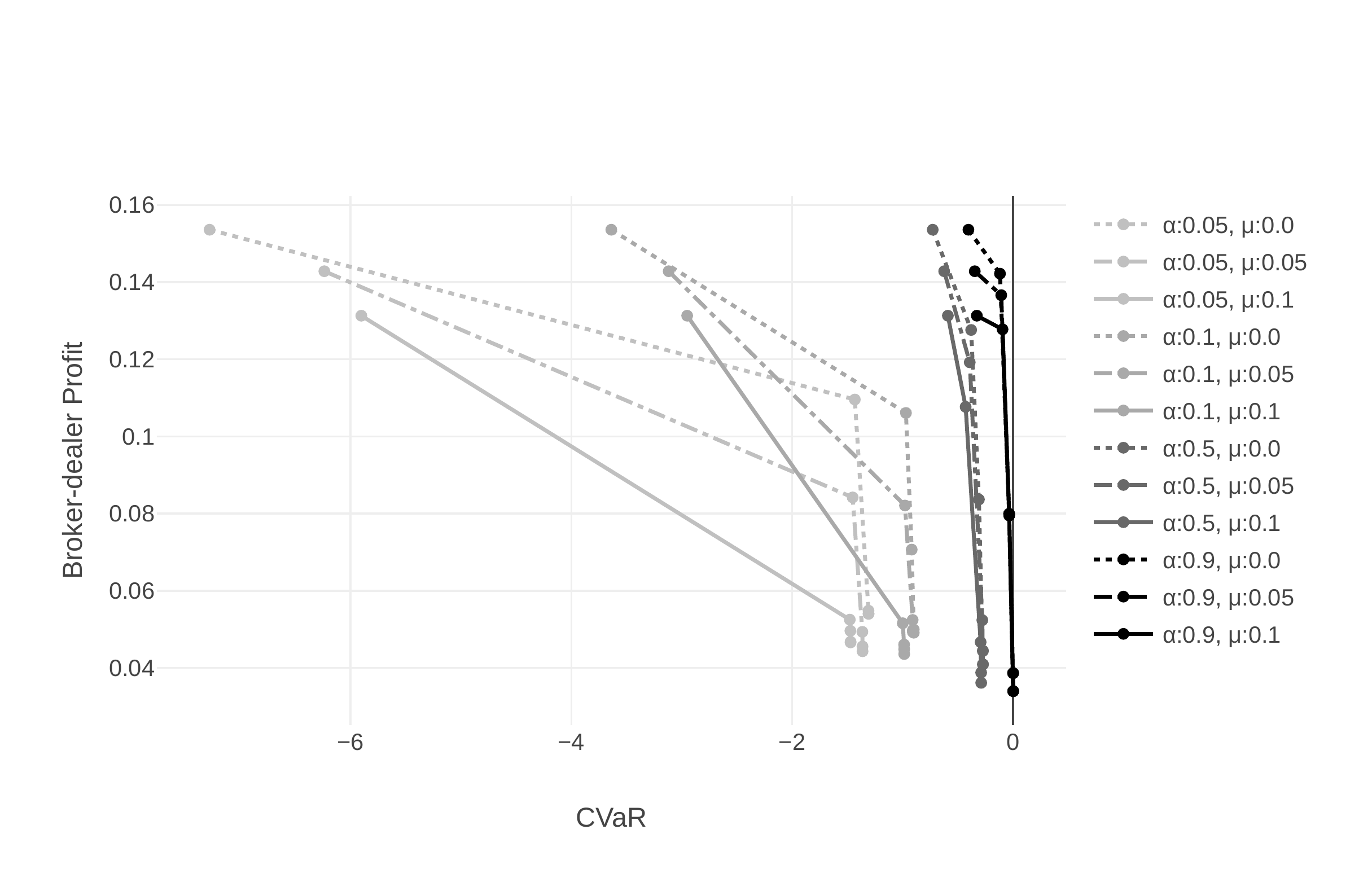}}

\fbox{\includegraphics[scale=0.35]{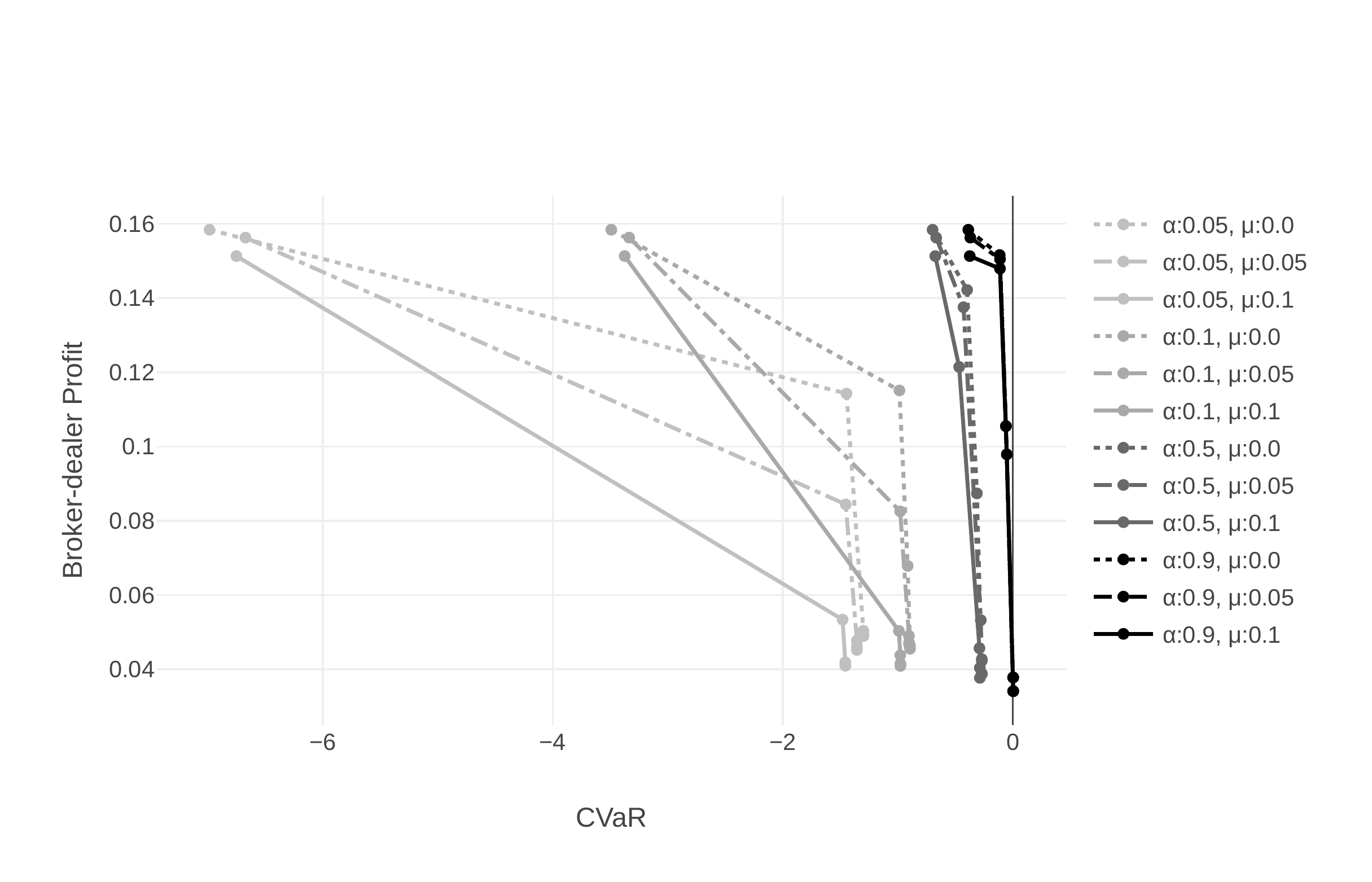}}

\end{multicols}
\caption{ {\footnotesize Discrete Pareto front of the \textbf{MSWP} for different values of $\xi\in \{0,0.25,0.5,0.75,1\}$ (from left to ritght) and different risk profiles, for instances type \textbf{A} (left), \textbf{B} (right).}}
\end{figure}

\begin{figure}[H]
\centering

\begin{multicols}{2}

\fbox{\includegraphics[scale=0.35]{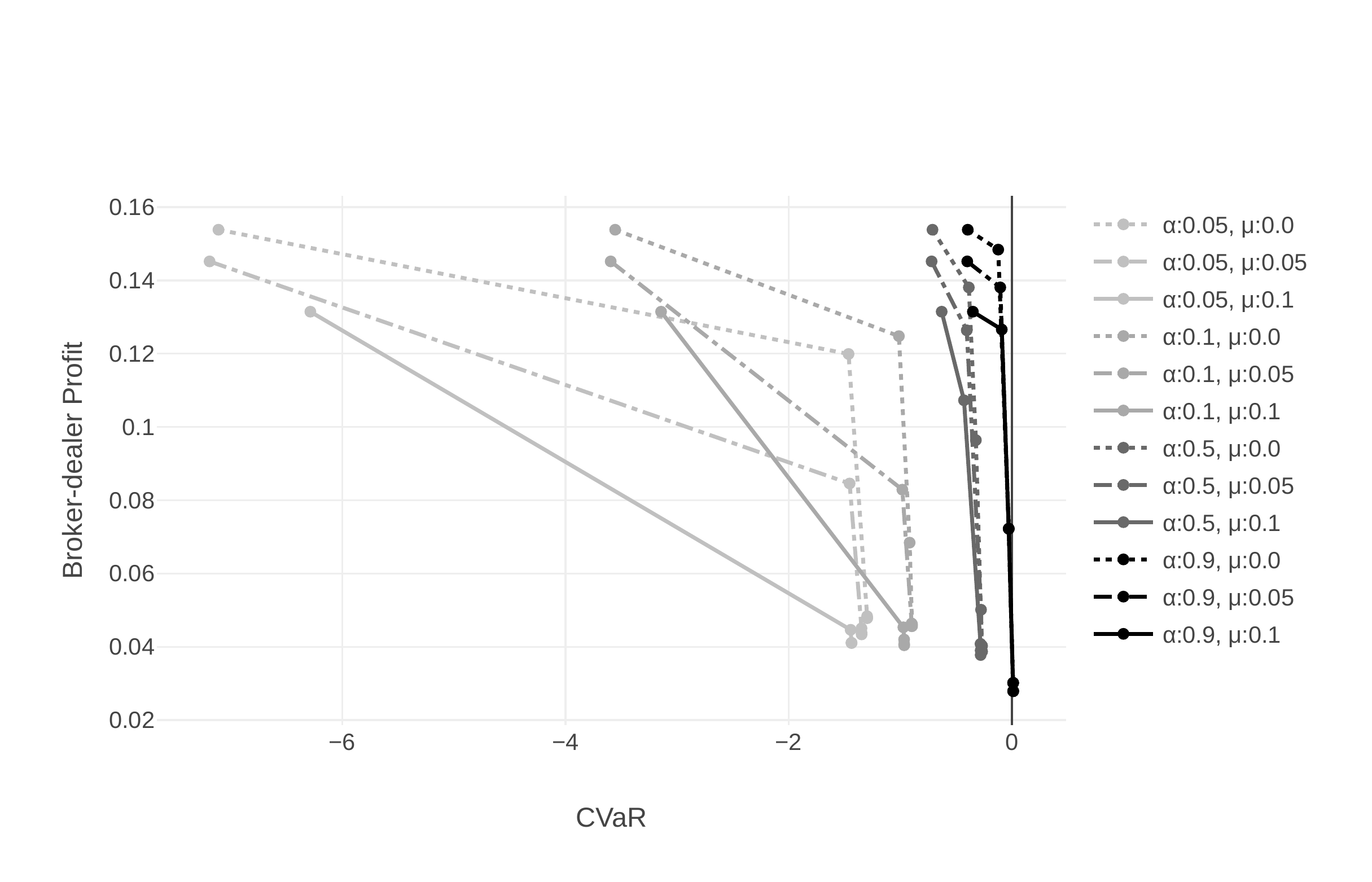}}

\fbox{\includegraphics[scale=0.35]{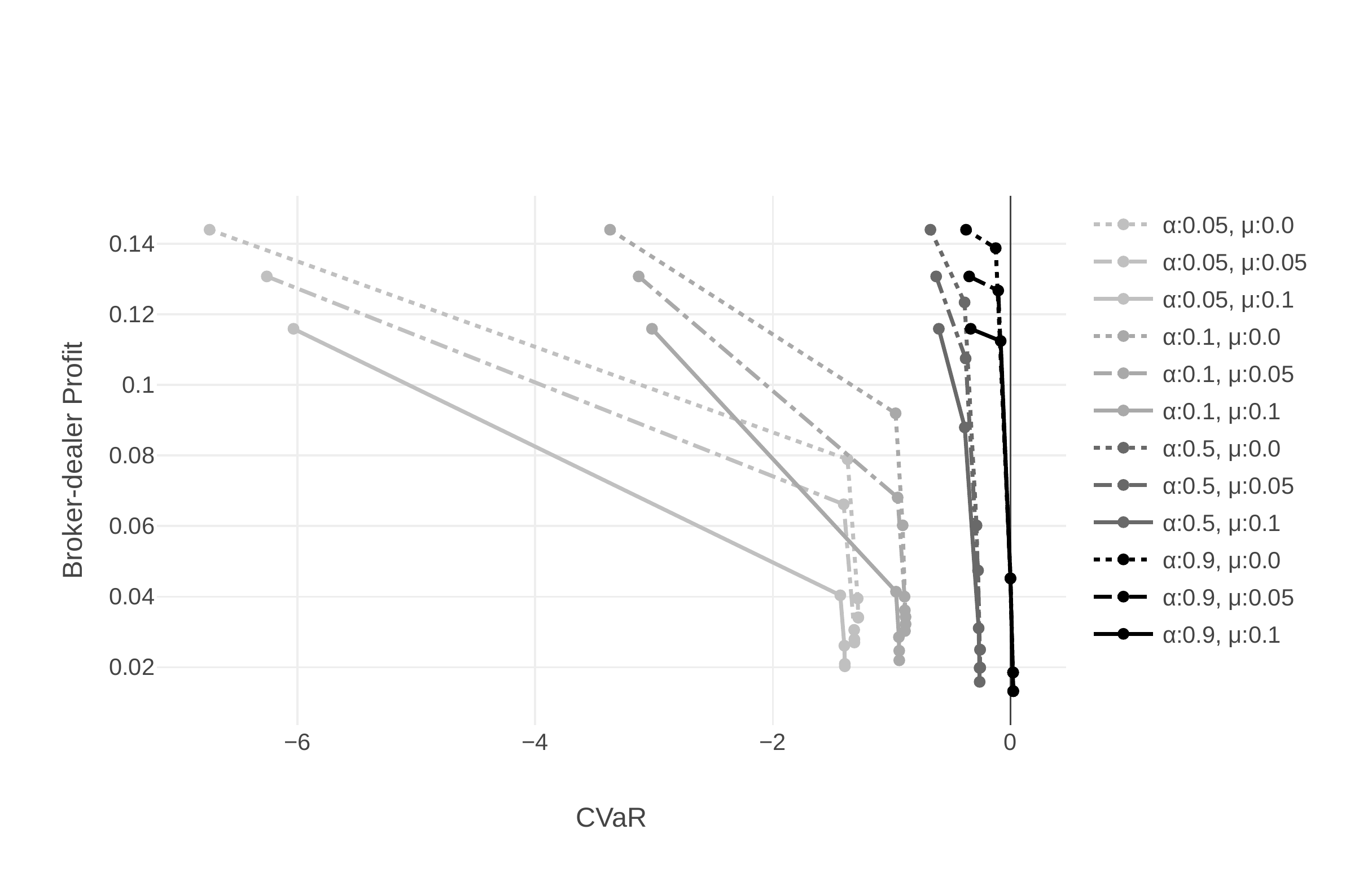}}

\end{multicols}
\caption{ {\footnotesize Discrete Pareto front of the \textbf{MSWP} for different values of $\xi\in \{0,0.25,0.5,0.75,1\}$ (from left to ritght) and different risk profiles, for instances type \textbf{C} (left), \textbf{D} (right).}}
\end{figure}

\begin{figure}[H]
\centering

\begin{multicols}{2}

\fbox{\includegraphics[scale=0.35]{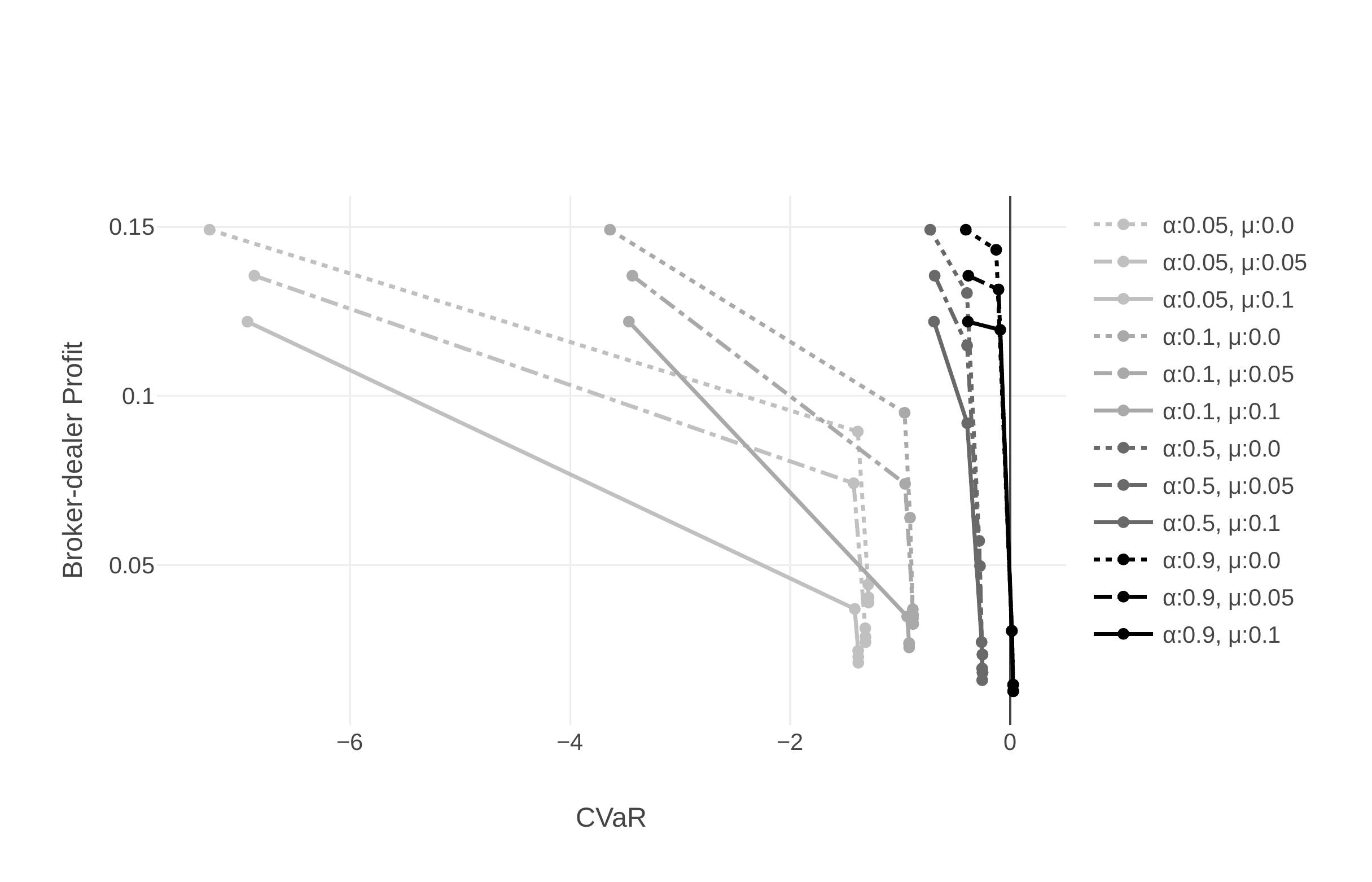}}

\fbox{\includegraphics[scale=0.35]{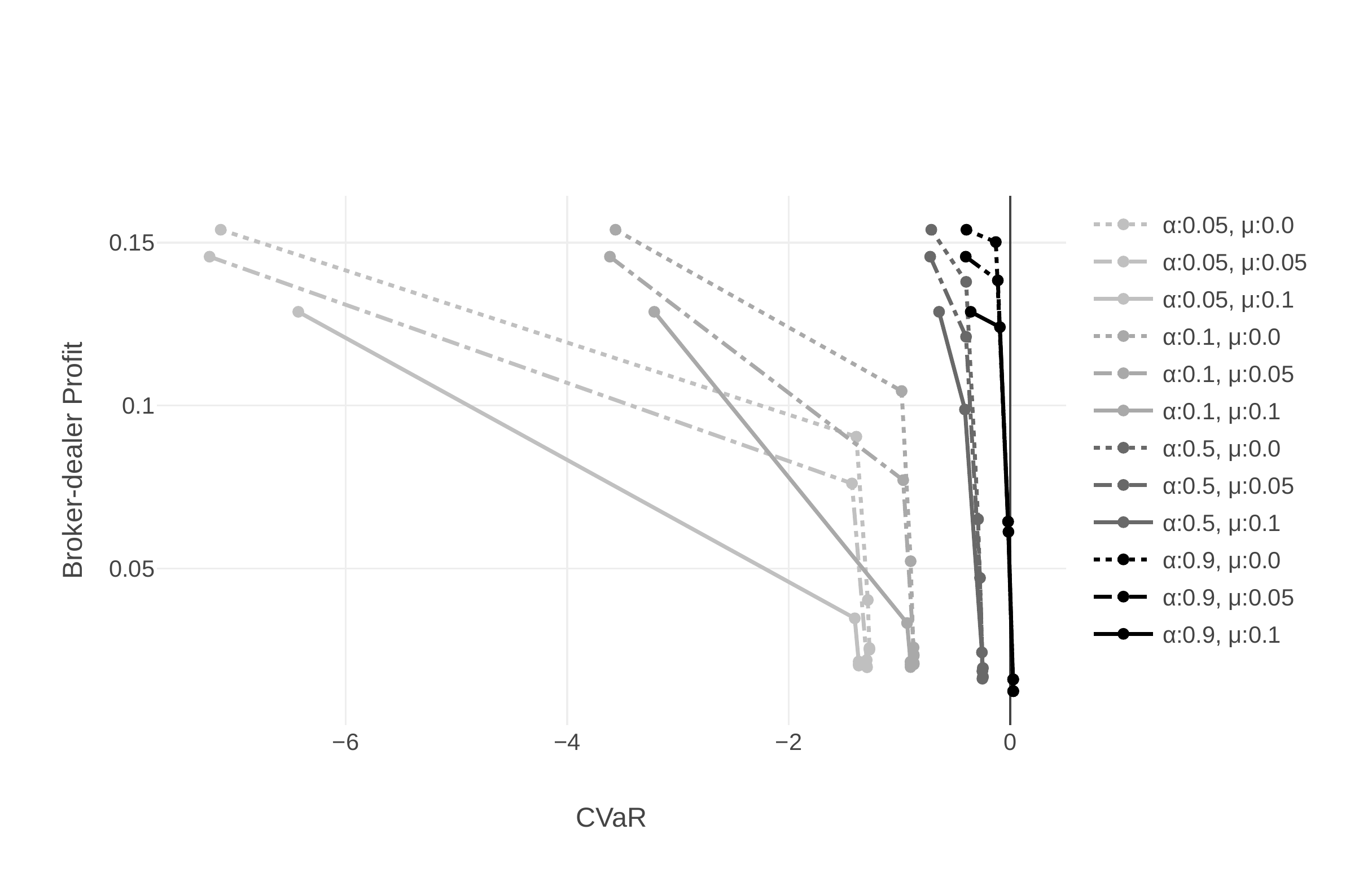}}

\end{multicols}
\caption{ {\footnotesize Discrete Pareto front of the \textbf{MSWP} for different values of $\xi\in \{0,0.25,0.5,0.75,1\}$ (from left to ritght) and different risk profiles, for instances type \textbf{E} (left), \textbf{F} (right).}}
\end{figure}

\begin{figure}[H]
\centering

\begin{multicols}{2}

\fbox{\includegraphics[scale=0.35]{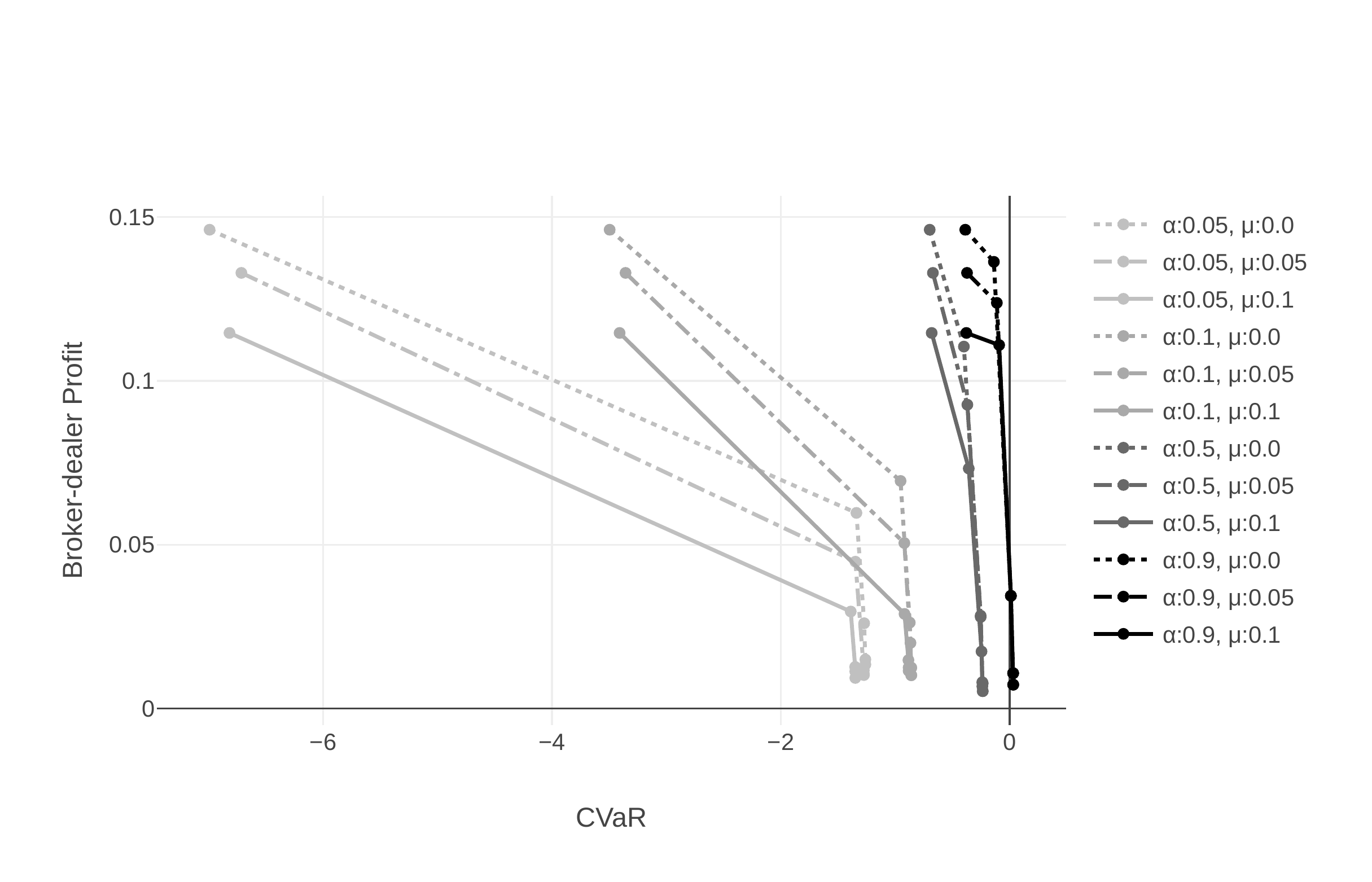}}

\fbox{\includegraphics[scale=0.35]{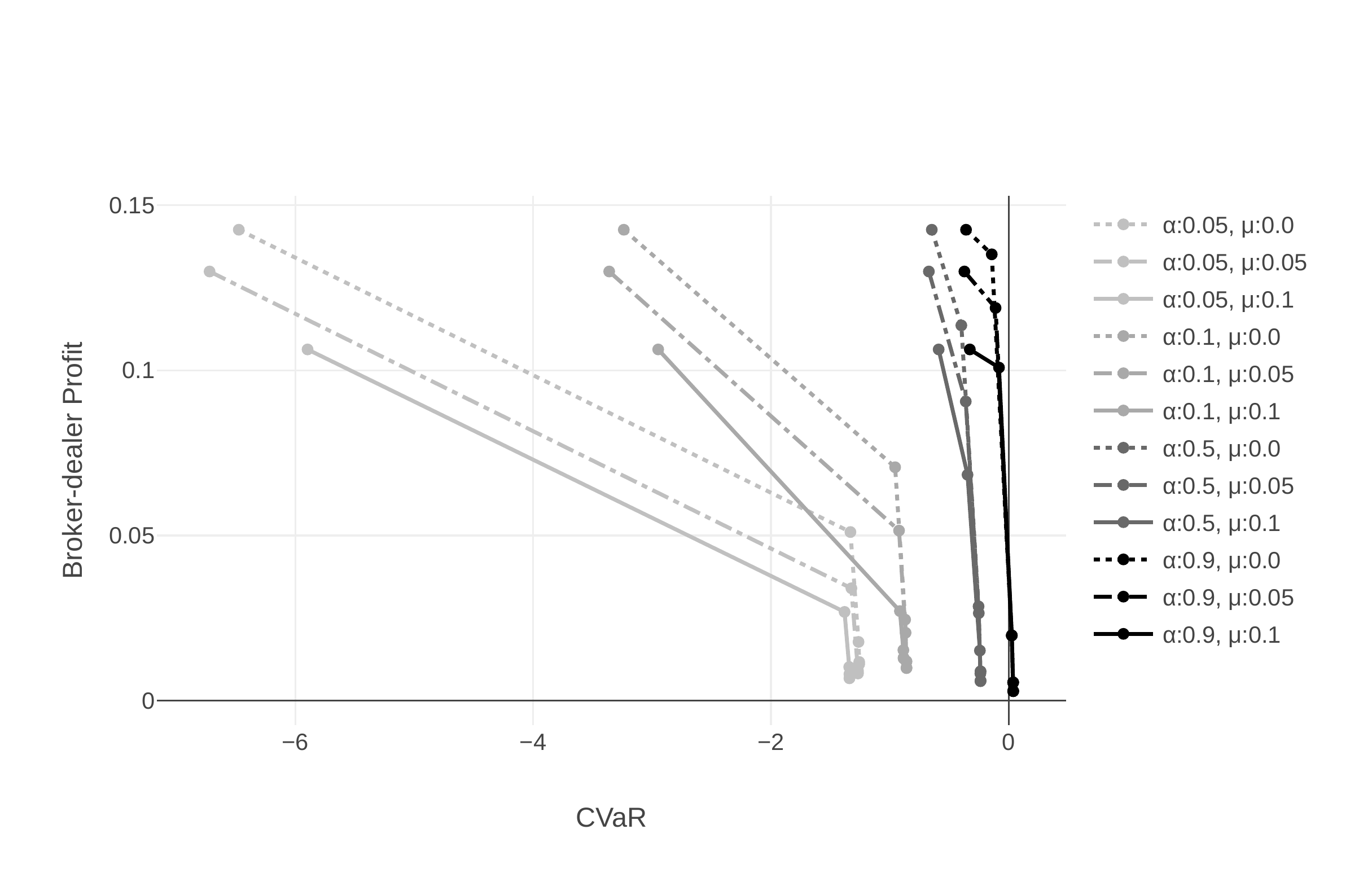}}

\end{multicols}
\caption{ {\footnotesize Discrete Pareto front of the \textbf{MSWP} for different values of $\xi\in \{0,0.25,0.5,0.75,1\}$ (from left to ritght) and different risk profiles, for instances type \textbf{G} (left), \textbf{H} (right).}}
\end{figure}

\begin{figure}[H]
\centering

\fbox{\includegraphics[scale=0.35]{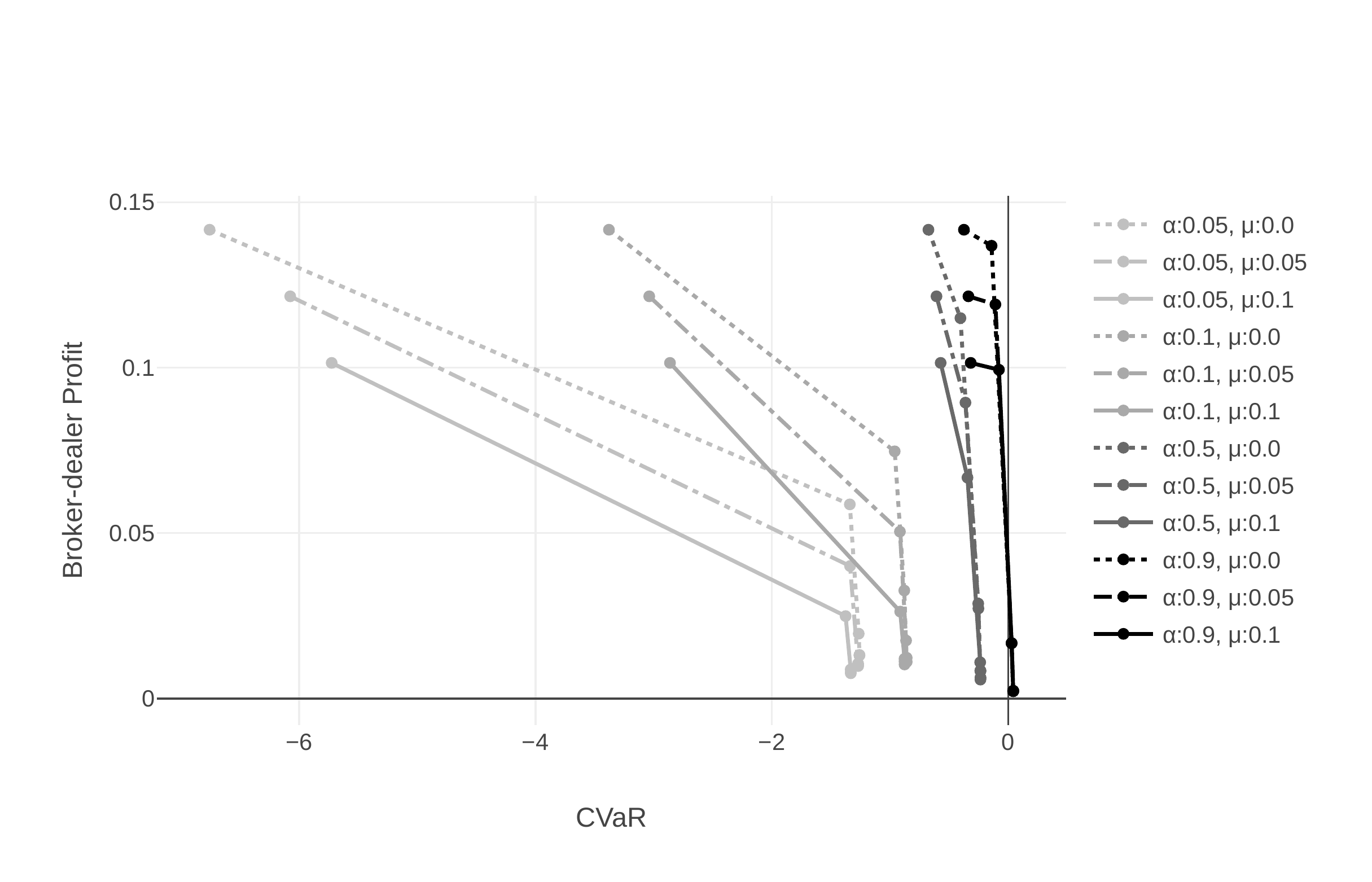}}

\caption{ {\footnotesize Discrete Pareto front of the \textbf{MSWP} for different values of $\xi\in \{0,0.25,0.5,0.75,1\}$ (from left to ritght) and different risk profiles, for instances type \textbf{I}}}
\end{figure}

\subsection{Comparing solutions across problems}

\subsubsection*{CVaR}
\begin{figure}[H]
\centering

\begin{multicols}{2}

\fbox{\includegraphics[scale=0.35]{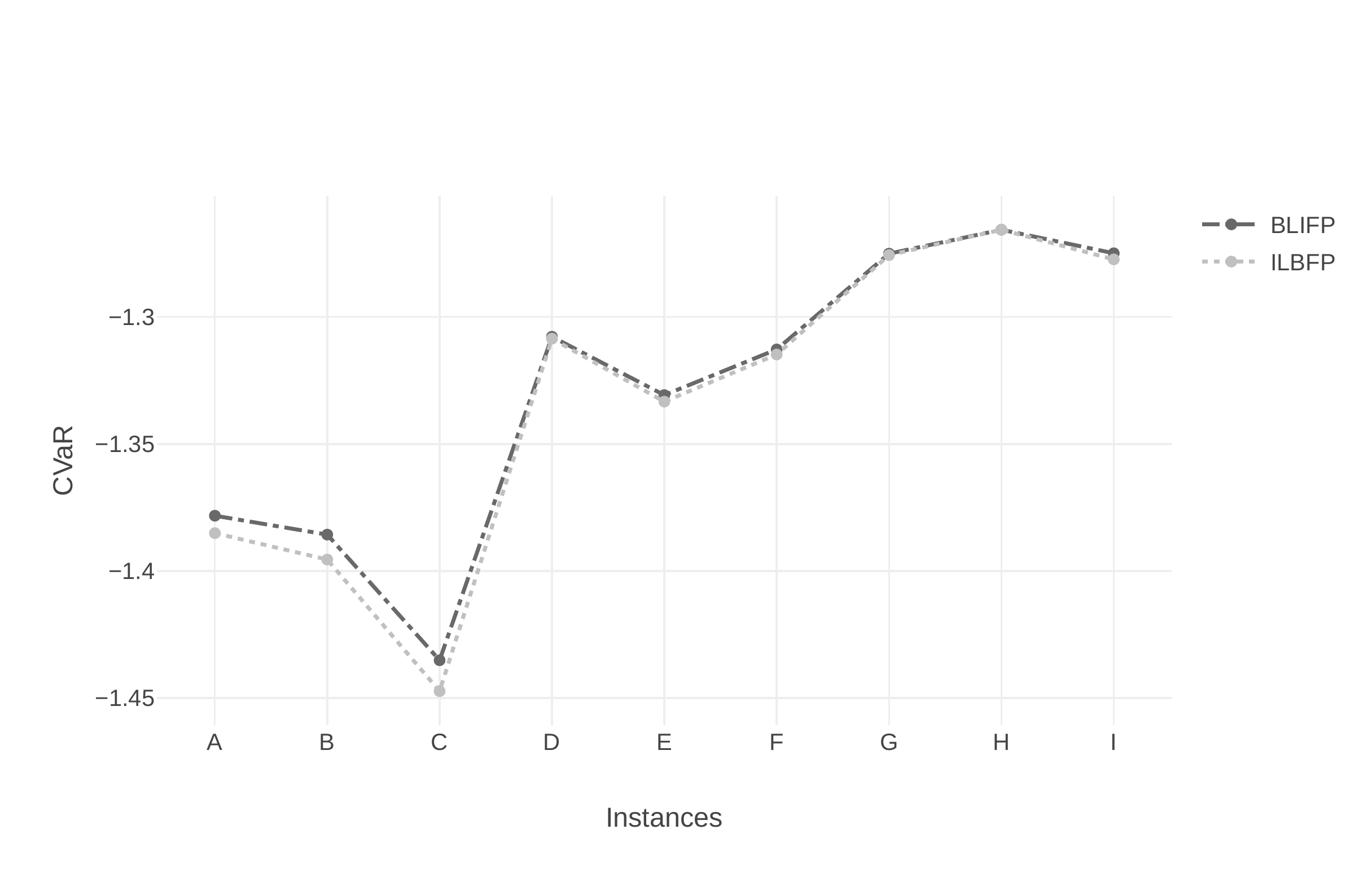}}

\fbox{\includegraphics[scale=0.35]{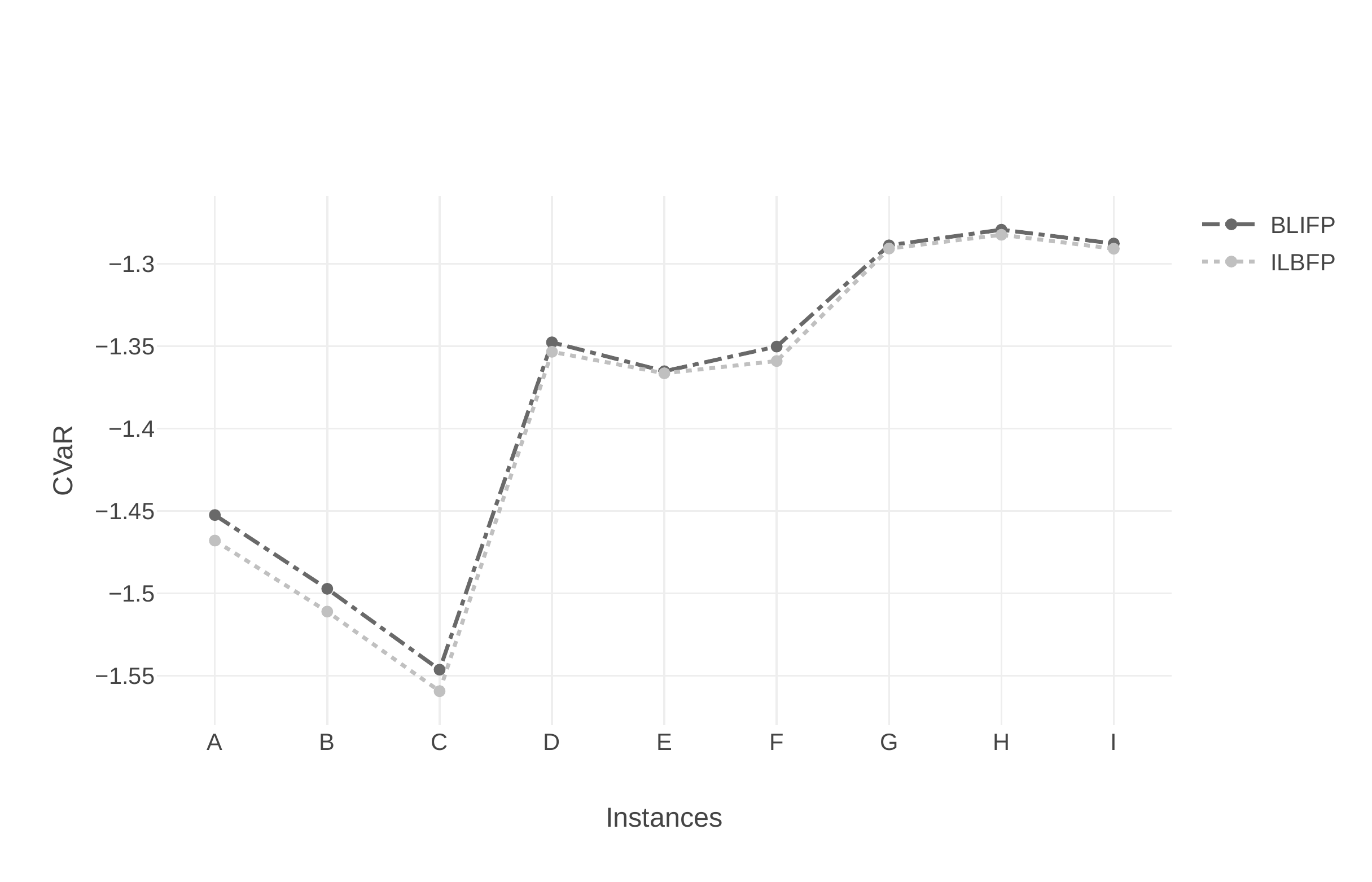}}

\end{multicols}
\caption{ {\footnotesize Values of the CVaR for {\textbf{BLIFP}} and {\textbf{ILBFP}} for $\alpha=0.05$ and $\mu_0=0$ (left) and for $\alpha=0.05$ and $\mu_0=0.05$ (right)}}\label{Grap:CVaR_compare}
\end{figure}

\begin{figure}[H]
\centering

\begin{multicols}{2}

\fbox{\includegraphics[scale=0.35]{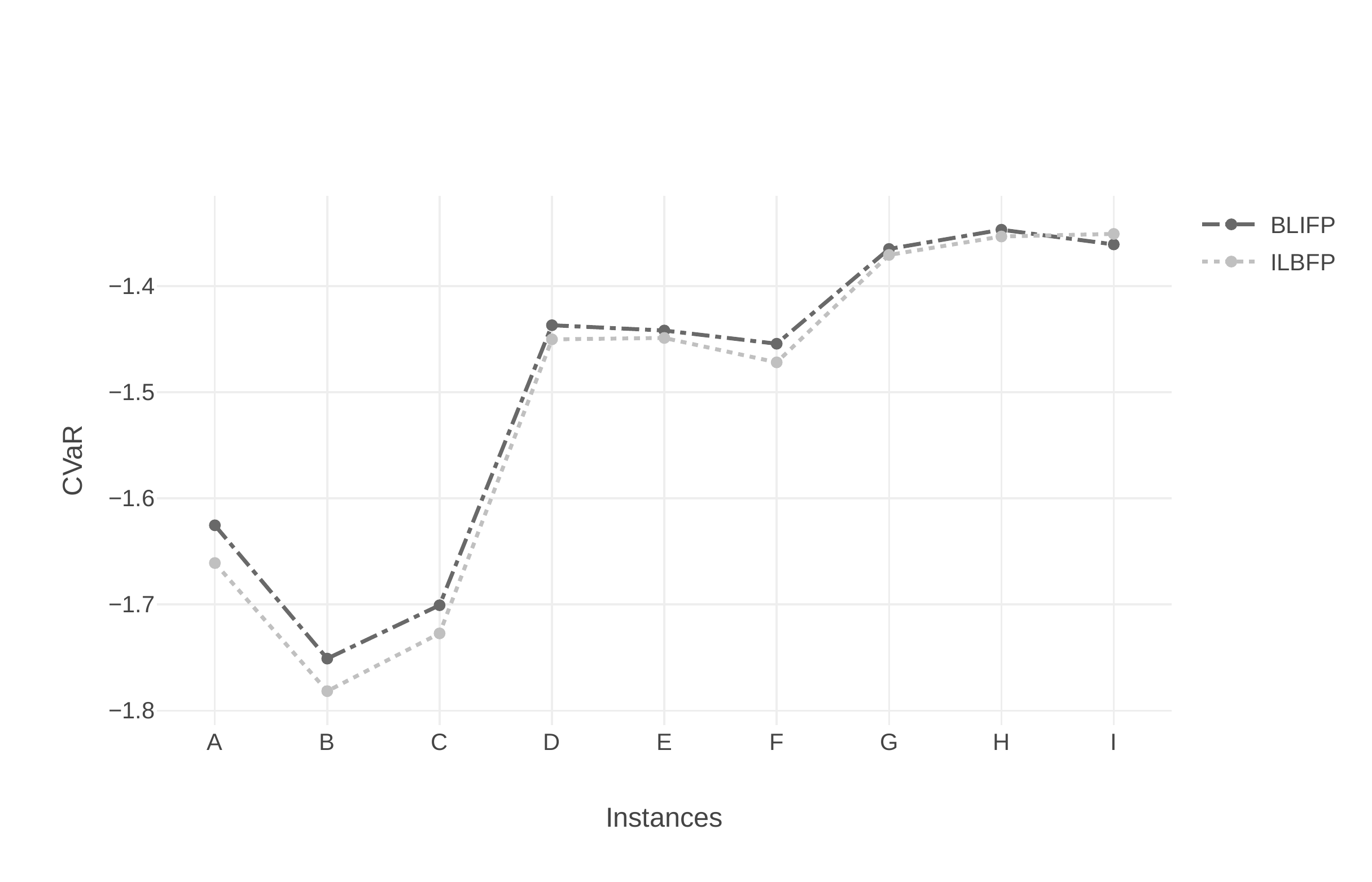}}

\fbox{\includegraphics[scale=0.35]{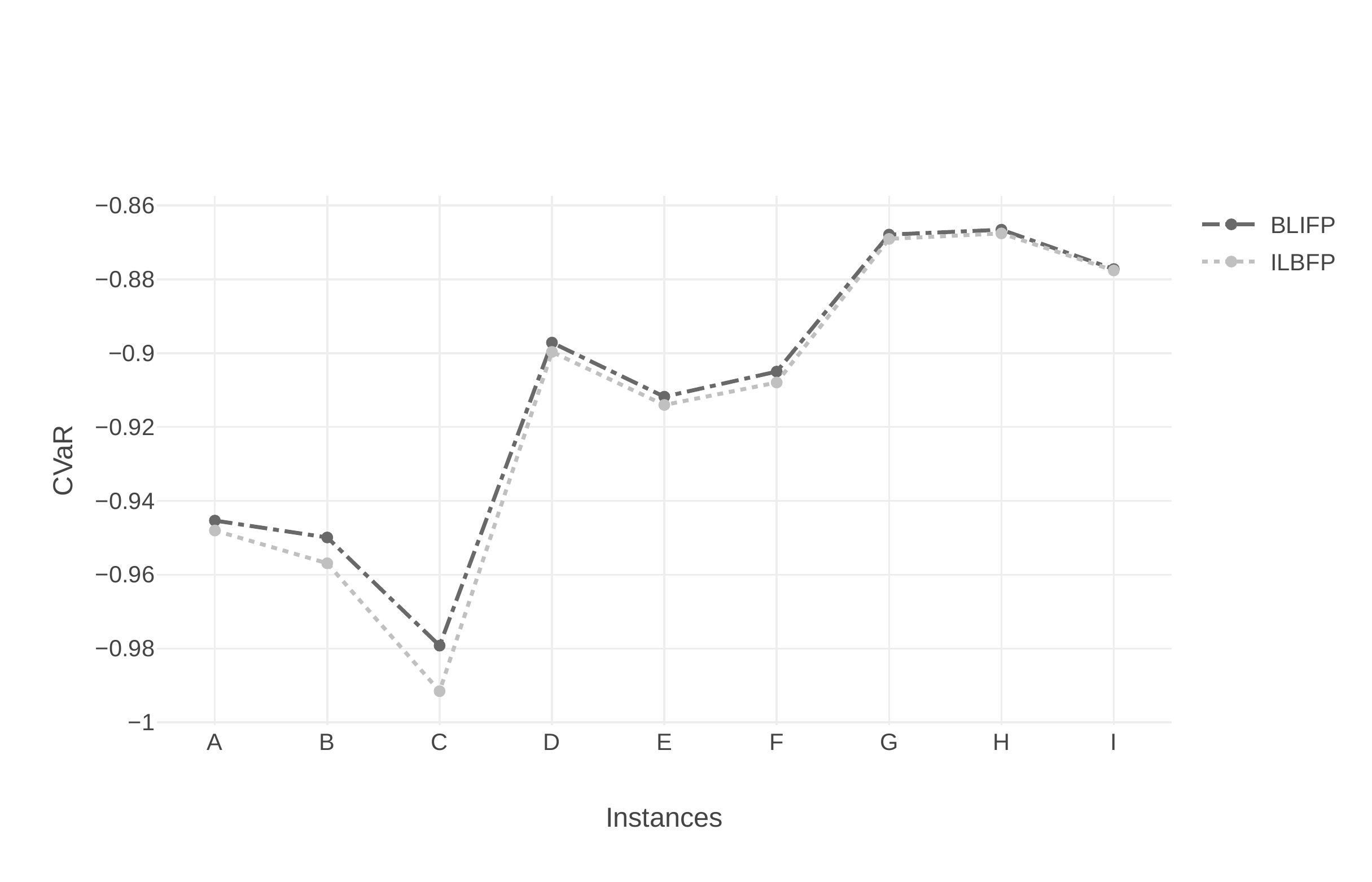}}

\end{multicols}
\caption{ {\footnotesize Values of the CVaR for {\textbf{BLIFP}} and {\textbf{ILBFP}} for $\alpha=0.05$ and $\mu_0=0.1$ (left) and for $\alpha=0.1$ and $\mu_0=0$ (right)}}
\end{figure}

\begin{figure}[H]
\centering

\begin{multicols}{2}

\fbox{\includegraphics[scale=0.35]{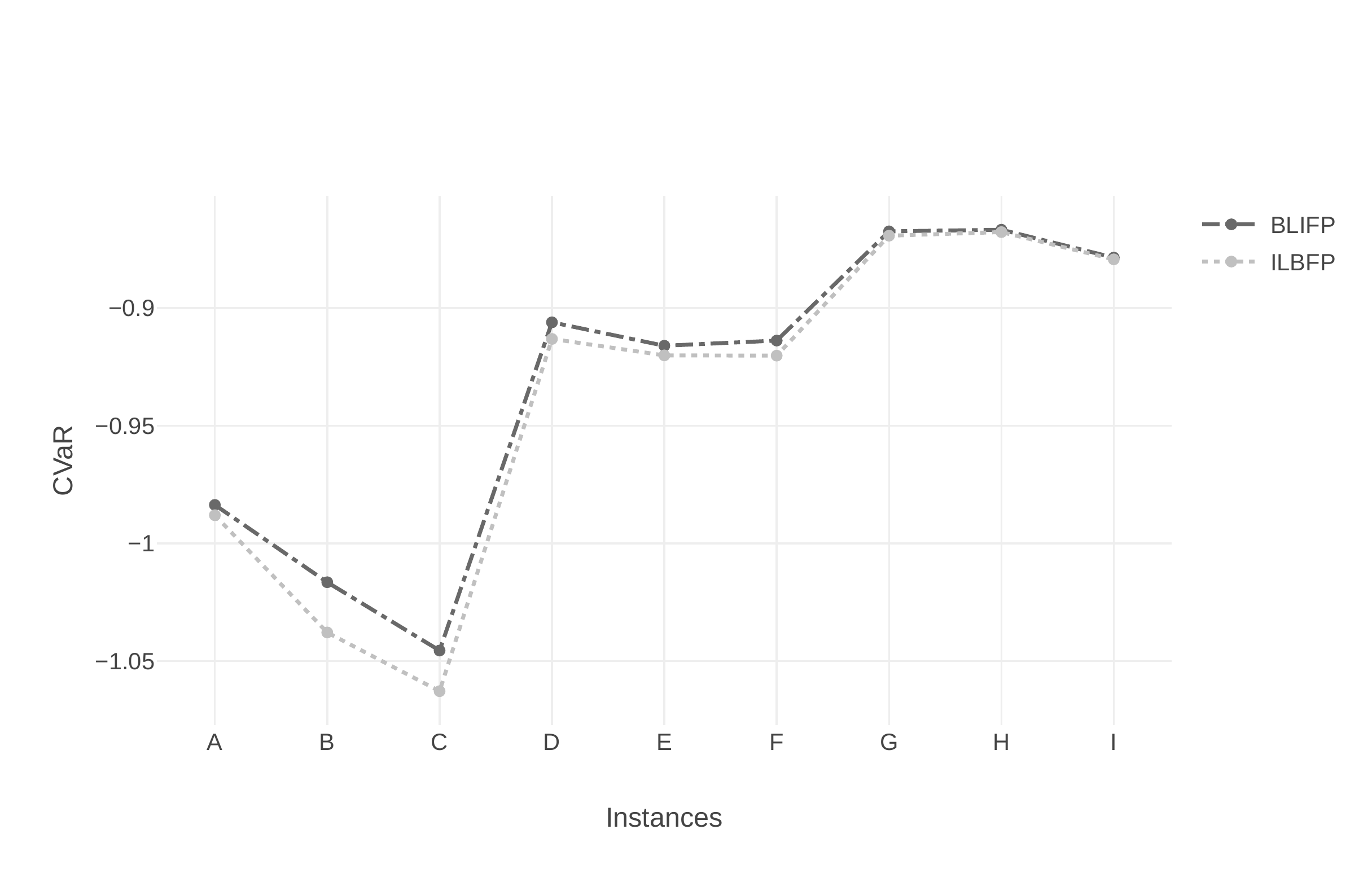}}

\fbox{\includegraphics[scale=0.35]{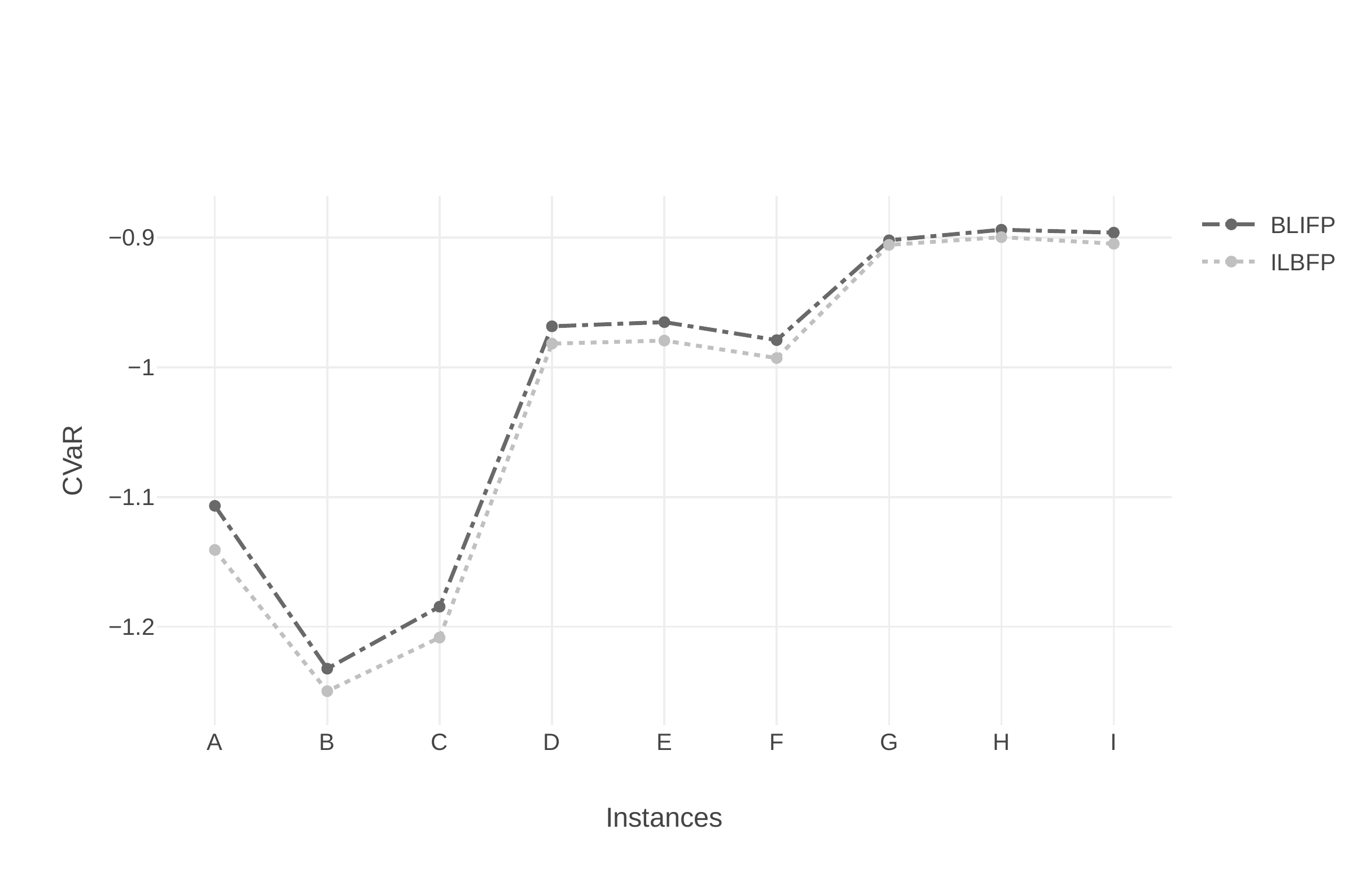}}

\end{multicols}
\caption{ {\footnotesize Values of the CVaR for {\textbf{BLIFP}} and {\textbf{ILBFP}} for $\alpha=0.1$ and $\mu_0=0.05$ (left) and for $\alpha=0.1$ and $\mu_0=0.1$ (right)}}
\end{figure}

\begin{figure}[H]
\centering

\begin{multicols}{2}

\fbox{\includegraphics[scale=0.35]{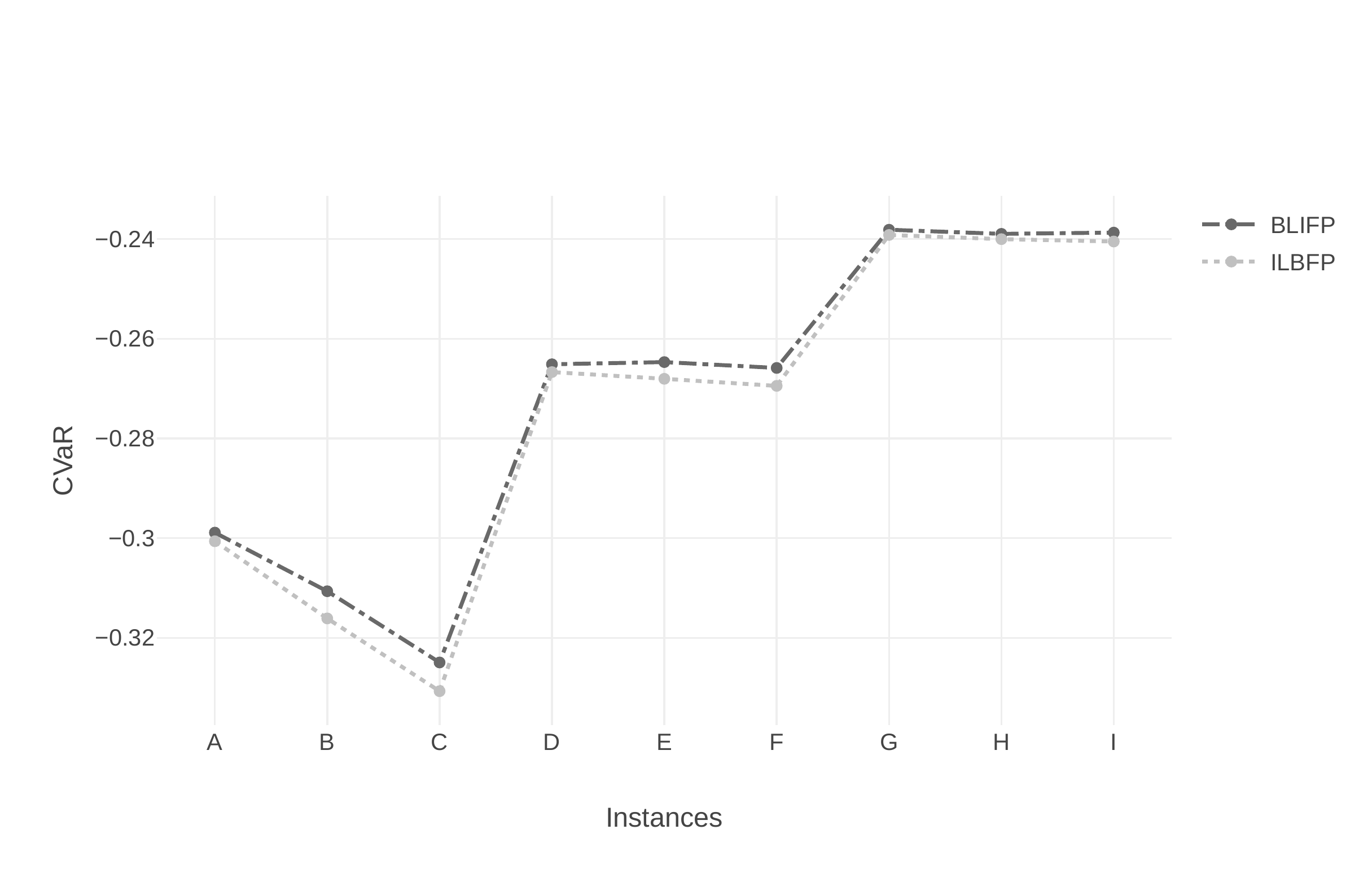}}

\fbox{\includegraphics[scale=0.35]{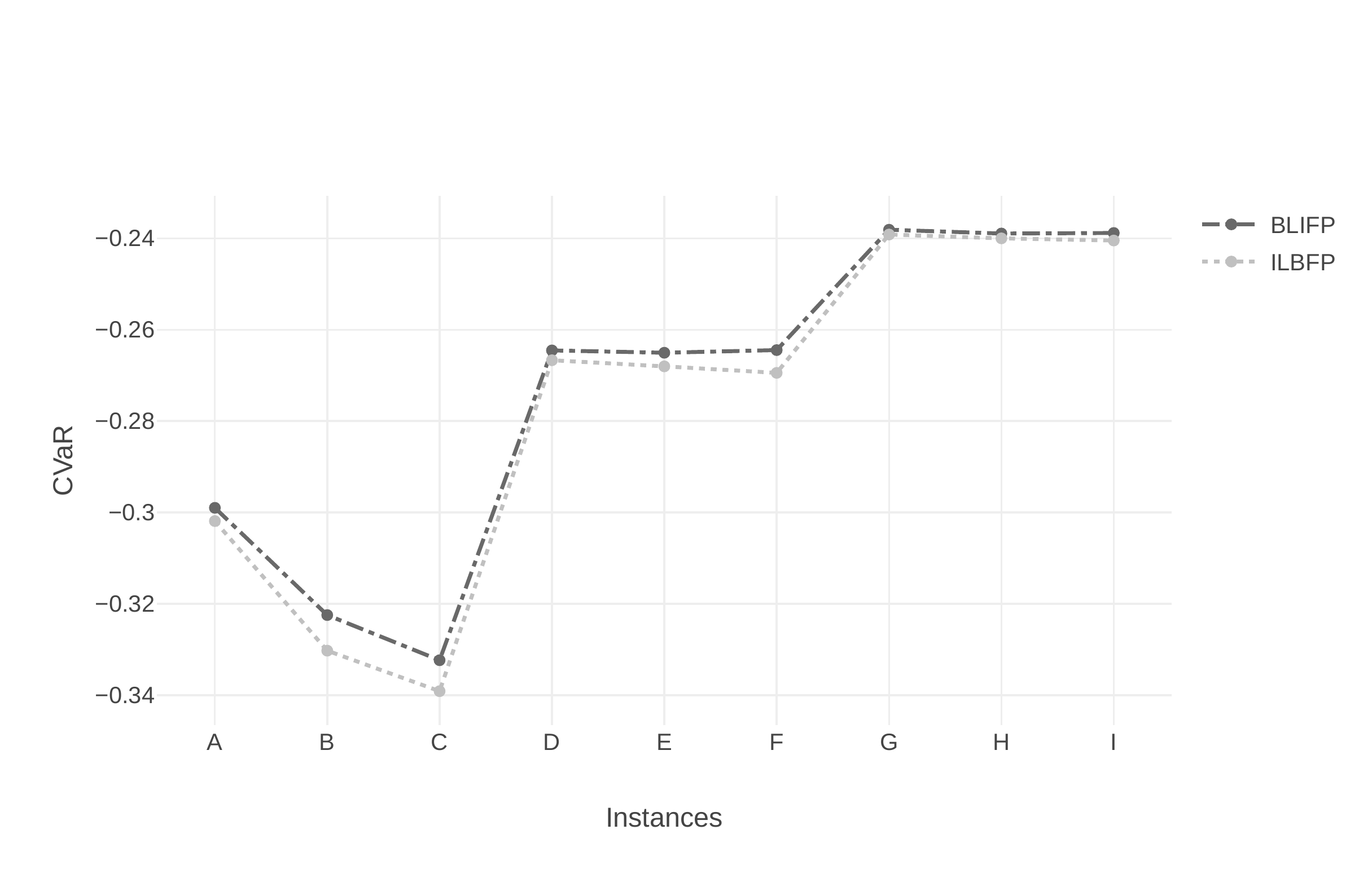}}

\end{multicols}
\caption{ {\footnotesize Values of the CVaR for {\textbf{BLIFP}} and {\textbf{ILBFP}} for $\alpha=0.5$ and $\mu_0=0$ (left) and for $\alpha=0.5$ and $\mu_0=0.05$ (right)}}
\end{figure}

\begin{figure}[H]
\centering

\begin{multicols}{2}

\fbox{\includegraphics[scale=0.35]{C05_010.pdf}}

\fbox{\includegraphics[scale=0.35]{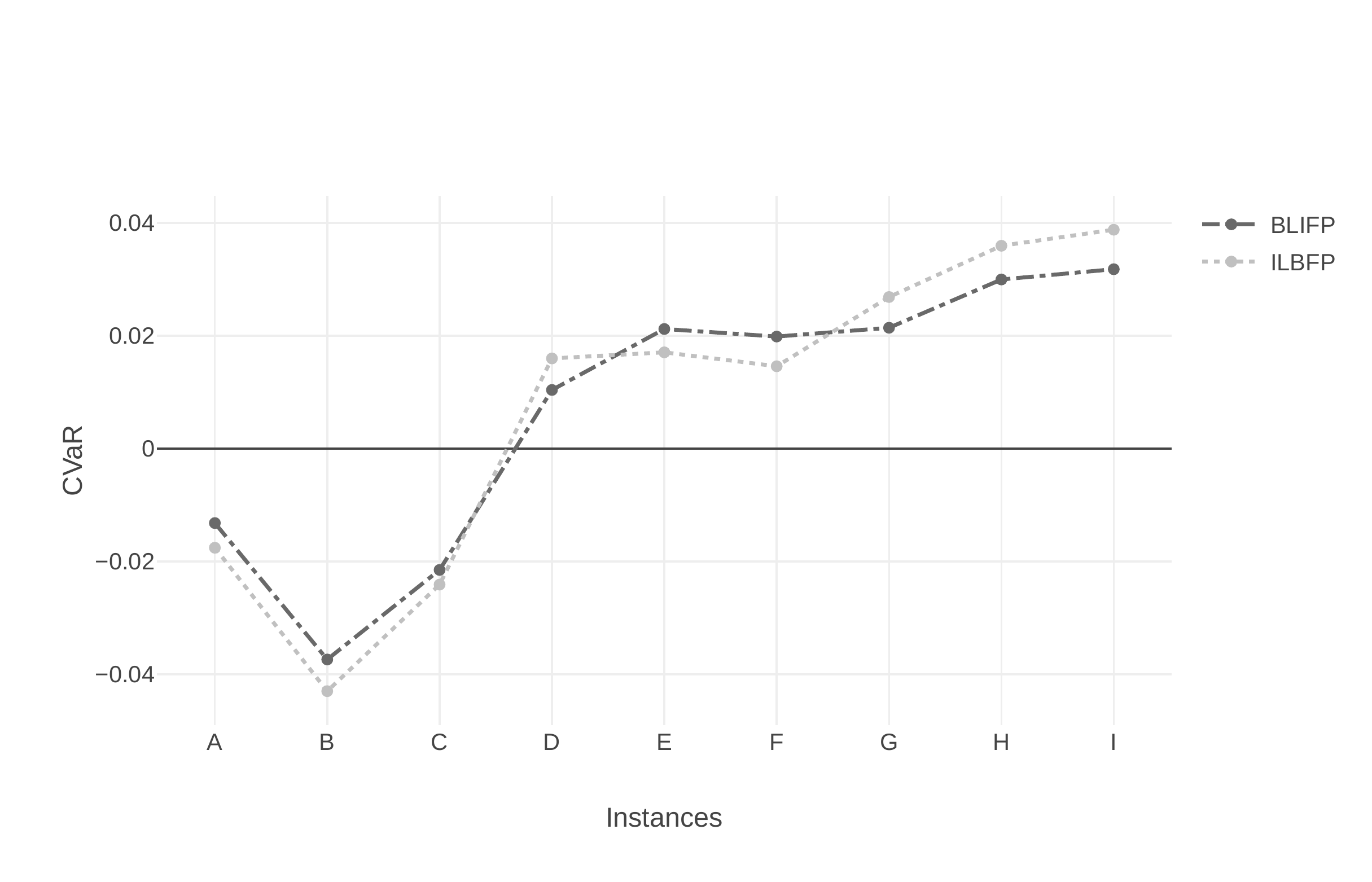}}

\end{multicols}
\caption{ {\footnotesize Values of the CVaR for {\textbf{BLIFP}} and {\textbf{ILBFP}} for $\alpha=0.5$ and $\mu_0=0.1$ (left) and for $\alpha=0.9$ and $\mu_0=0$ (right)}}
\end{figure}

\begin{figure}[H]
\centering

\begin{multicols}{2}

\fbox{\includegraphics[scale=0.35]{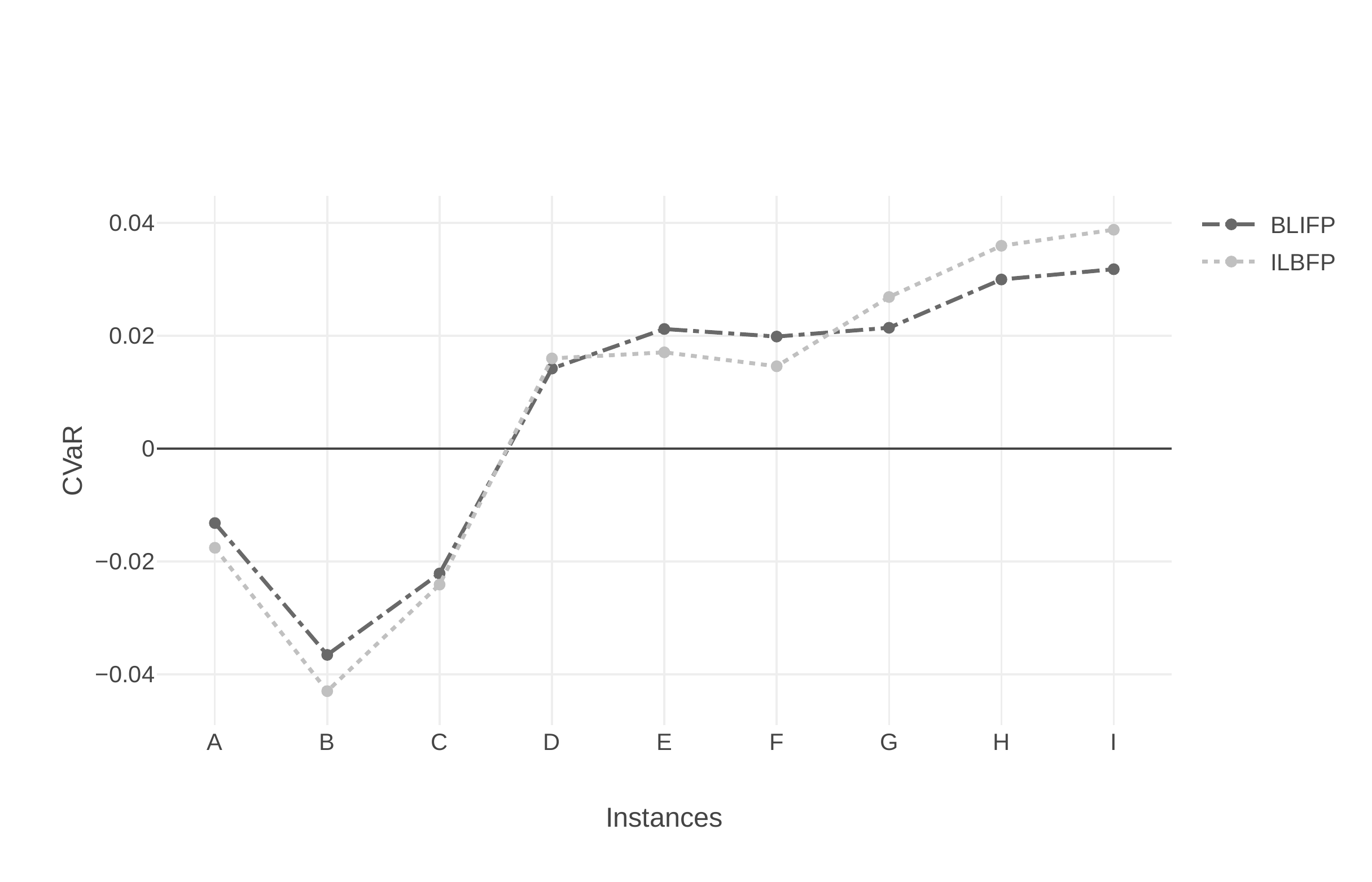}}

\fbox{\includegraphics[scale=0.35]{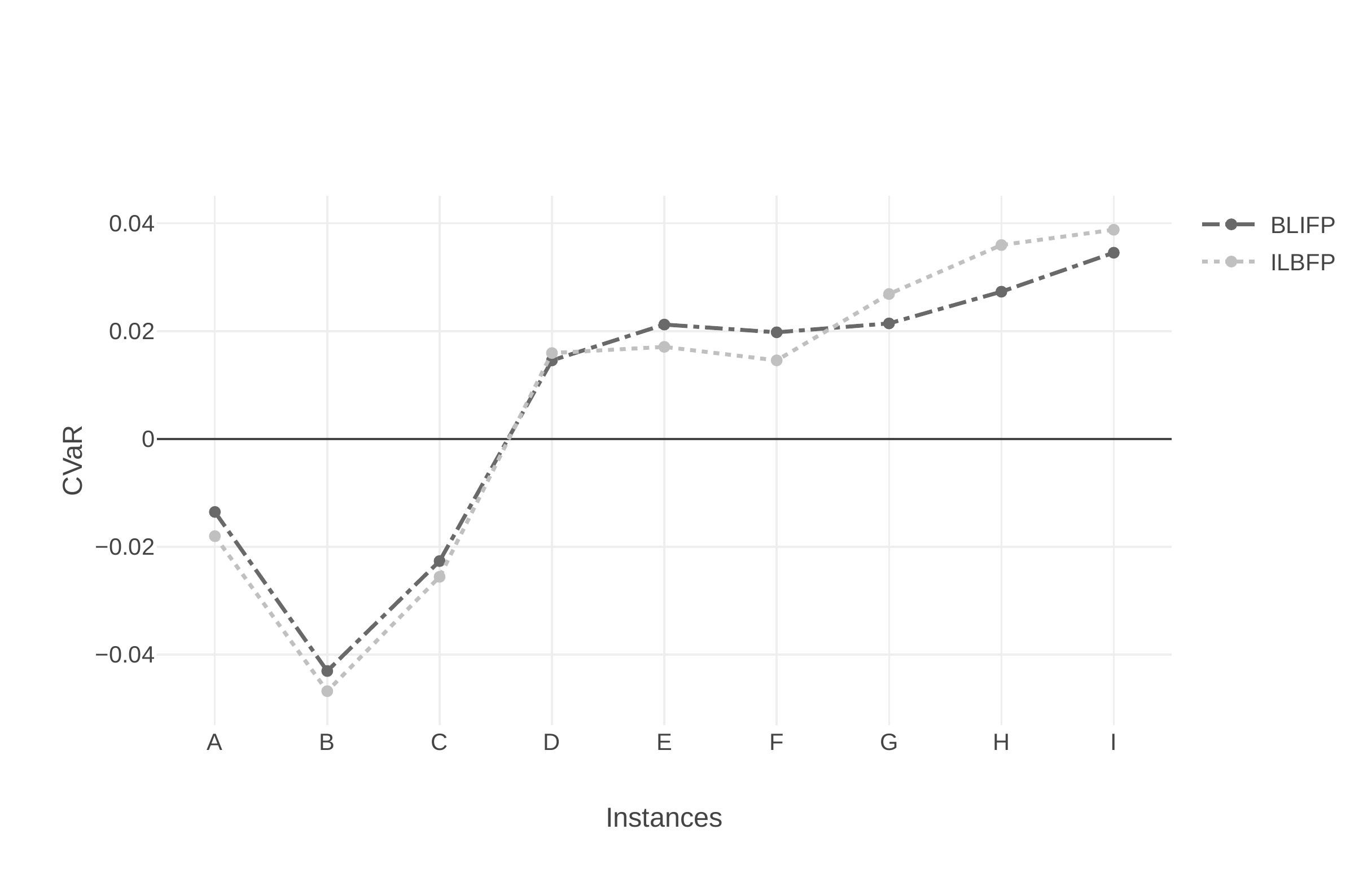}}

\end{multicols}
\caption{ {\footnotesize Values of the CVaR for {\textbf{BLIFP}} and {\textbf{ILBFP}} for $\alpha=0.9$ and $\mu_0=0.05$ (left) and for $\alpha=0.$ and $\mu_0=0.1$ (right)}}
\end{figure}

\subsubsection*{Broker-dealer profit}
\begin{figure}[H]
\centering

\begin{multicols}{2}

\fbox{\includegraphics[scale=0.35]{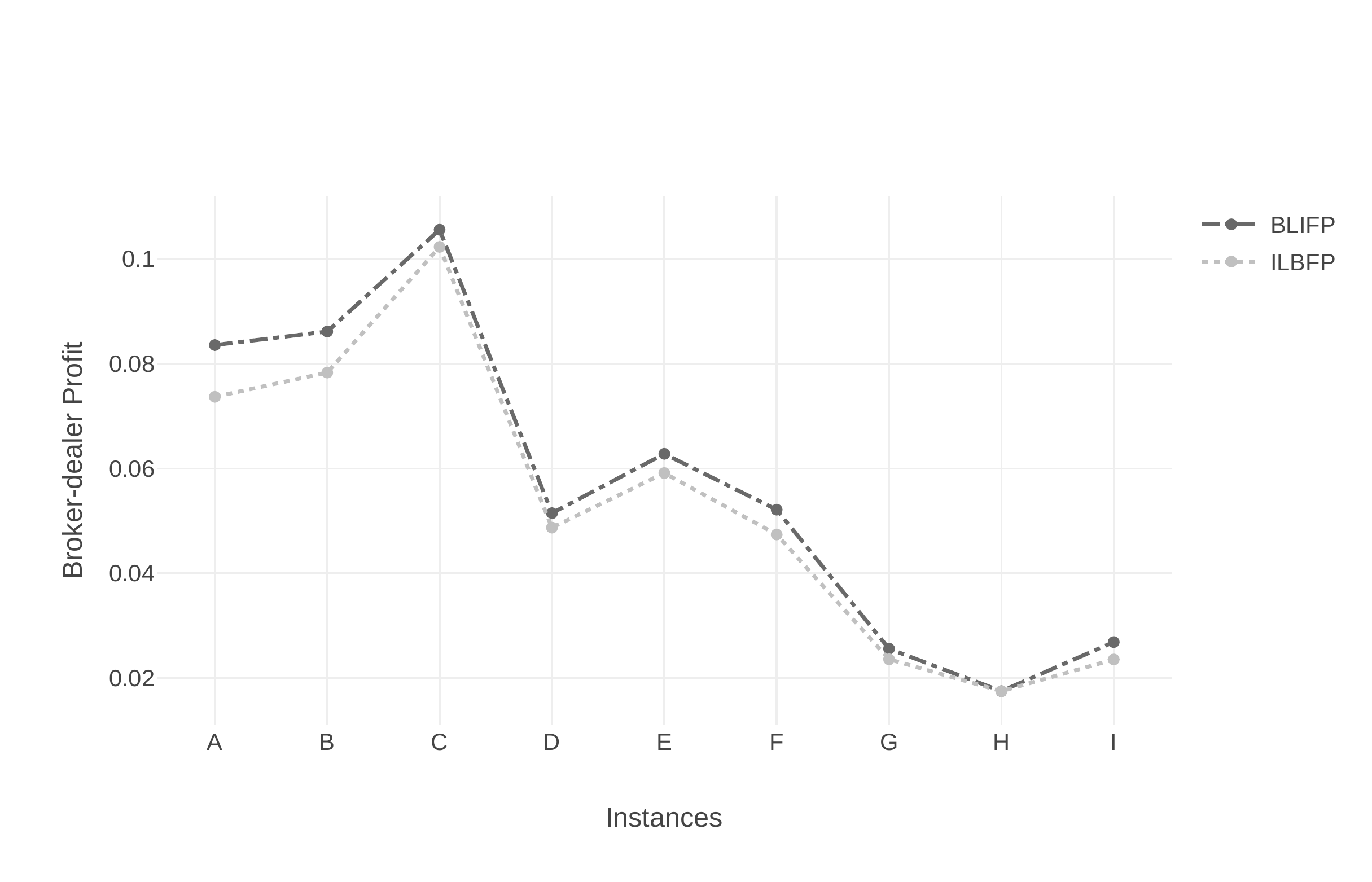}}

\fbox{\includegraphics[scale=0.35]{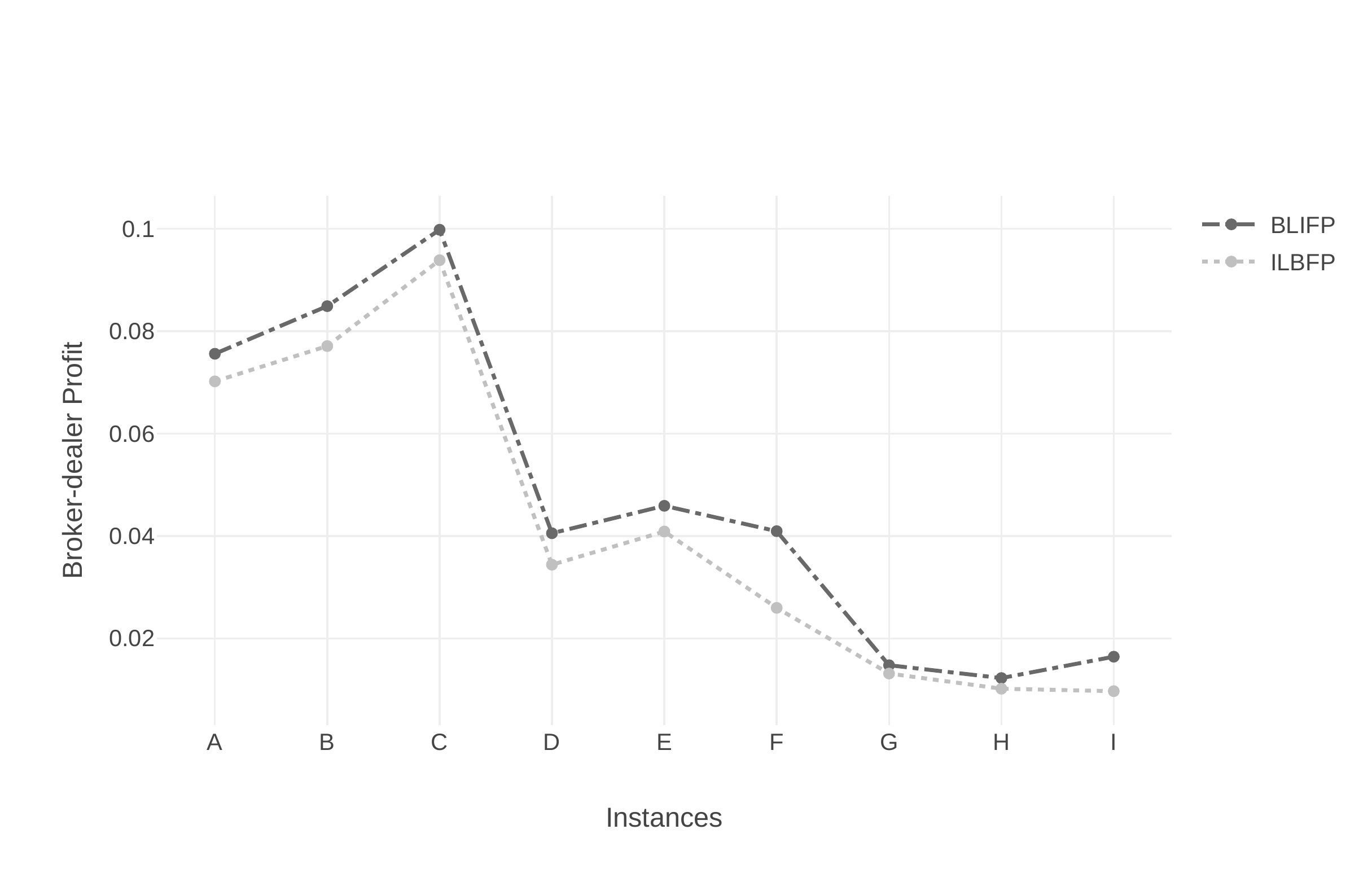}}

\end{multicols}
\caption{ {\footnotesize Values of the broker-dealer profit for {\textbf{BLIFP}} and {\textbf{ILBFP}} for $\alpha=0.05$ and $\mu_0=0$ (left) and for $\alpha=0.05$ and $\mu_0=0.05$ (right)}}\label{Grap:broker-dealer profit_compare}
\end{figure}

\begin{figure}[H]
\centering

\begin{multicols}{2}

\fbox{\includegraphics[scale=0.35]{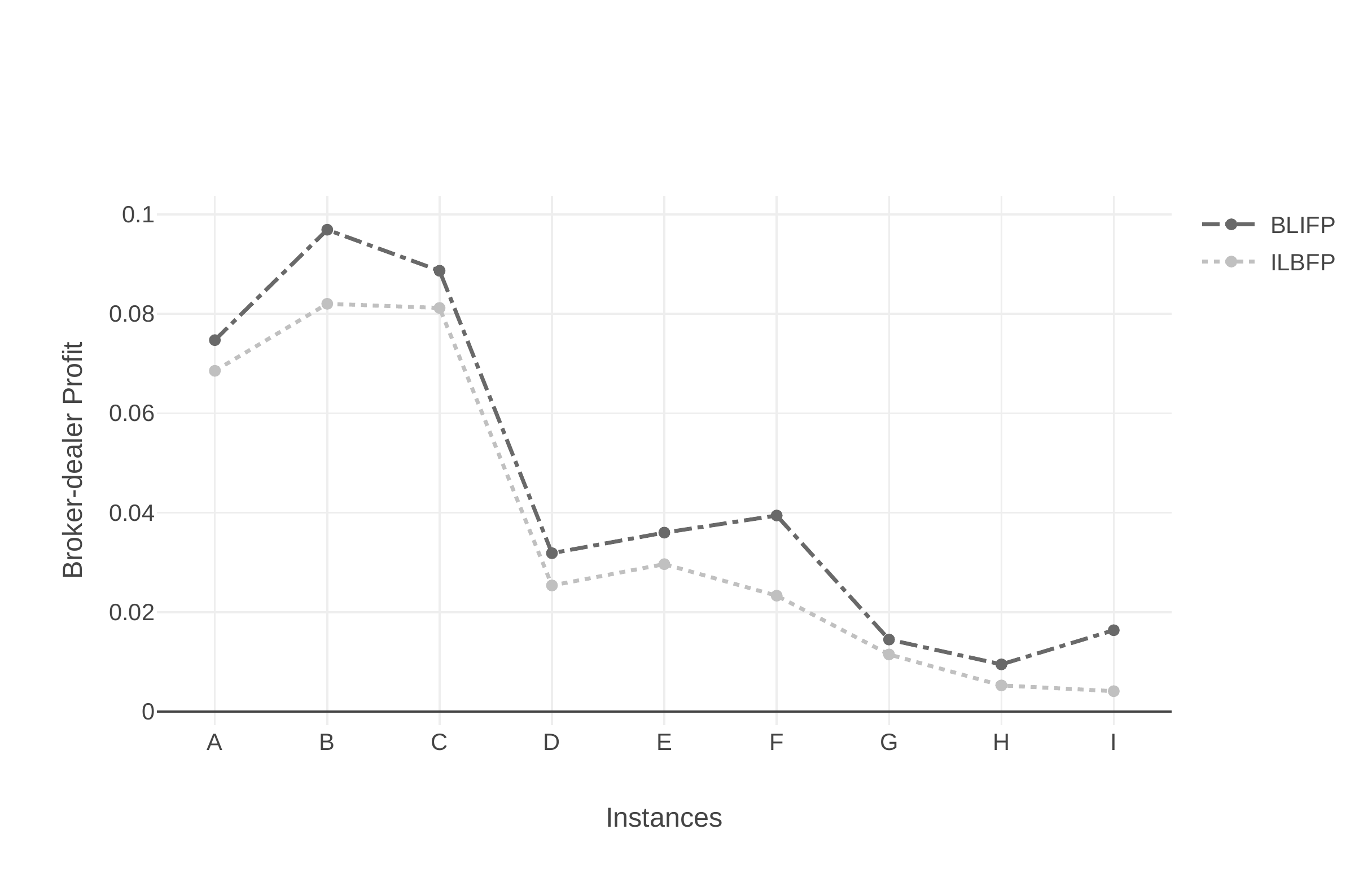}}

\fbox{\includegraphics[scale=0.35]{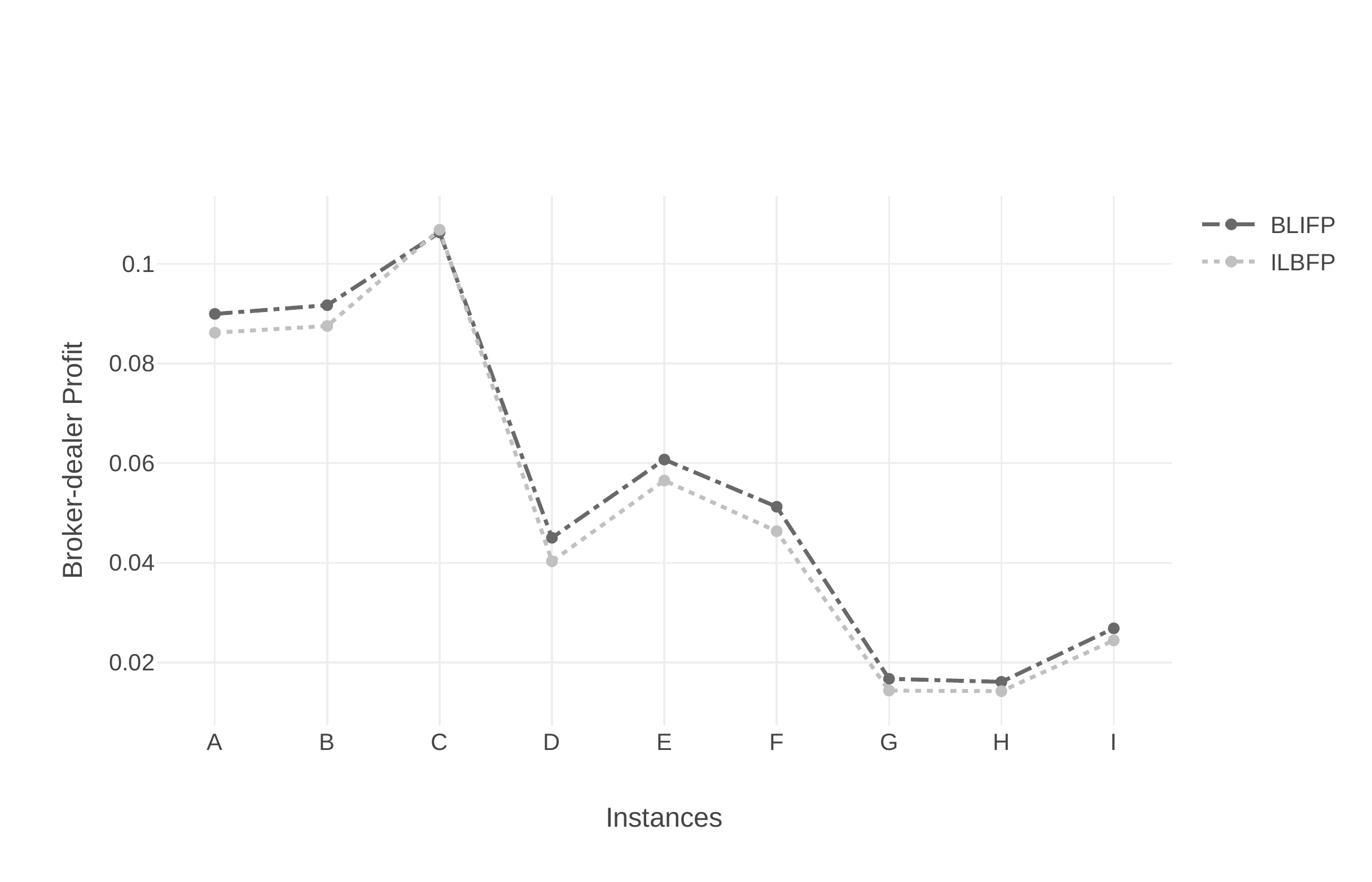}}

\end{multicols}
\caption{ {\footnotesize Values of the broker-dealer profit for {\textbf{BLIFP}} and {\textbf{ILBFP}} for $\alpha=0.05$ and $\mu_0=0.1$ (left) and for $\alpha=0.1$ and $\mu_0=0$ (right)}}
\end{figure}

\begin{figure}[H]
\centering

\begin{multicols}{2}

\fbox{\includegraphics[scale=0.35]{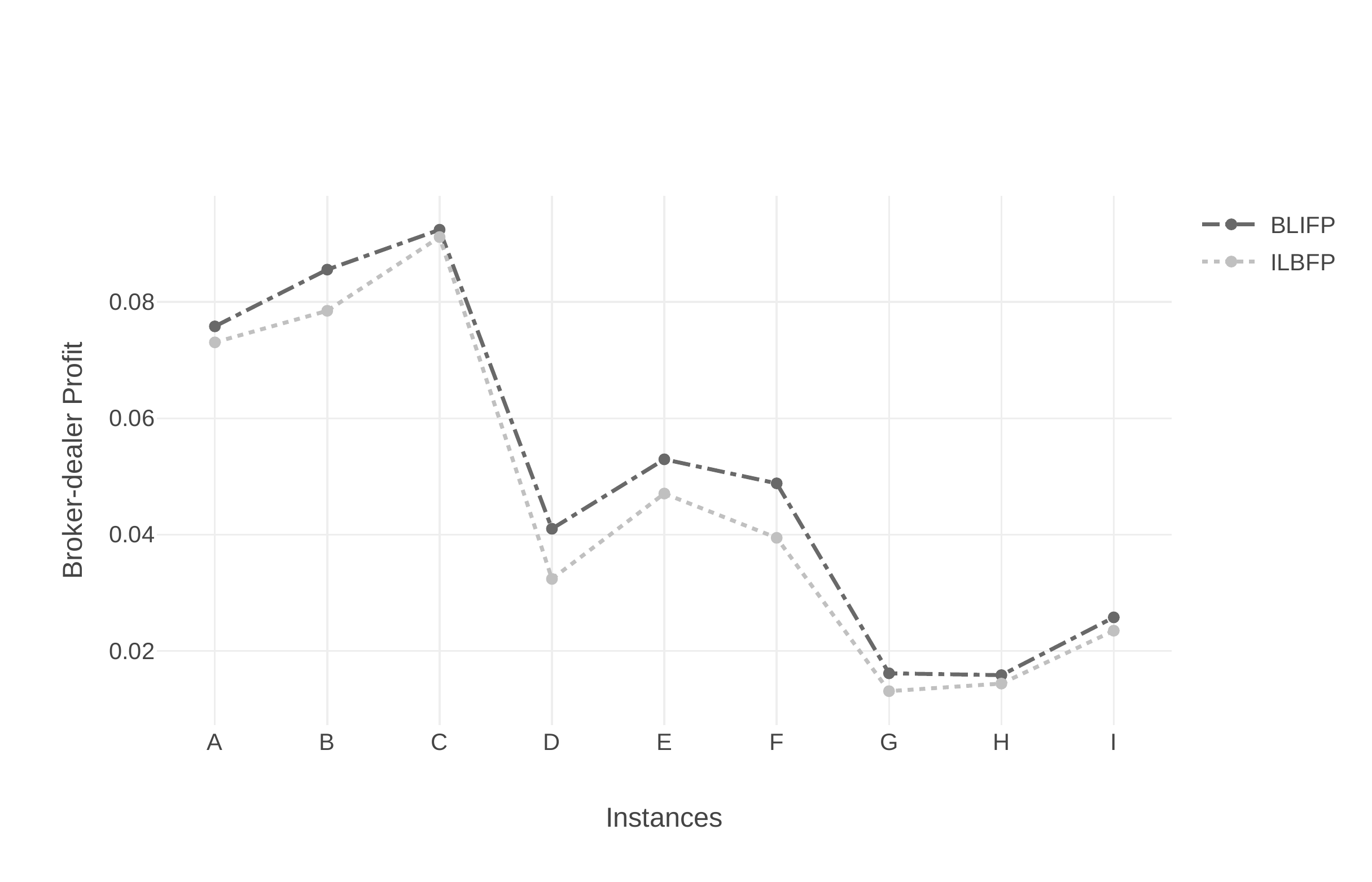}}

\fbox{\includegraphics[scale=0.35]{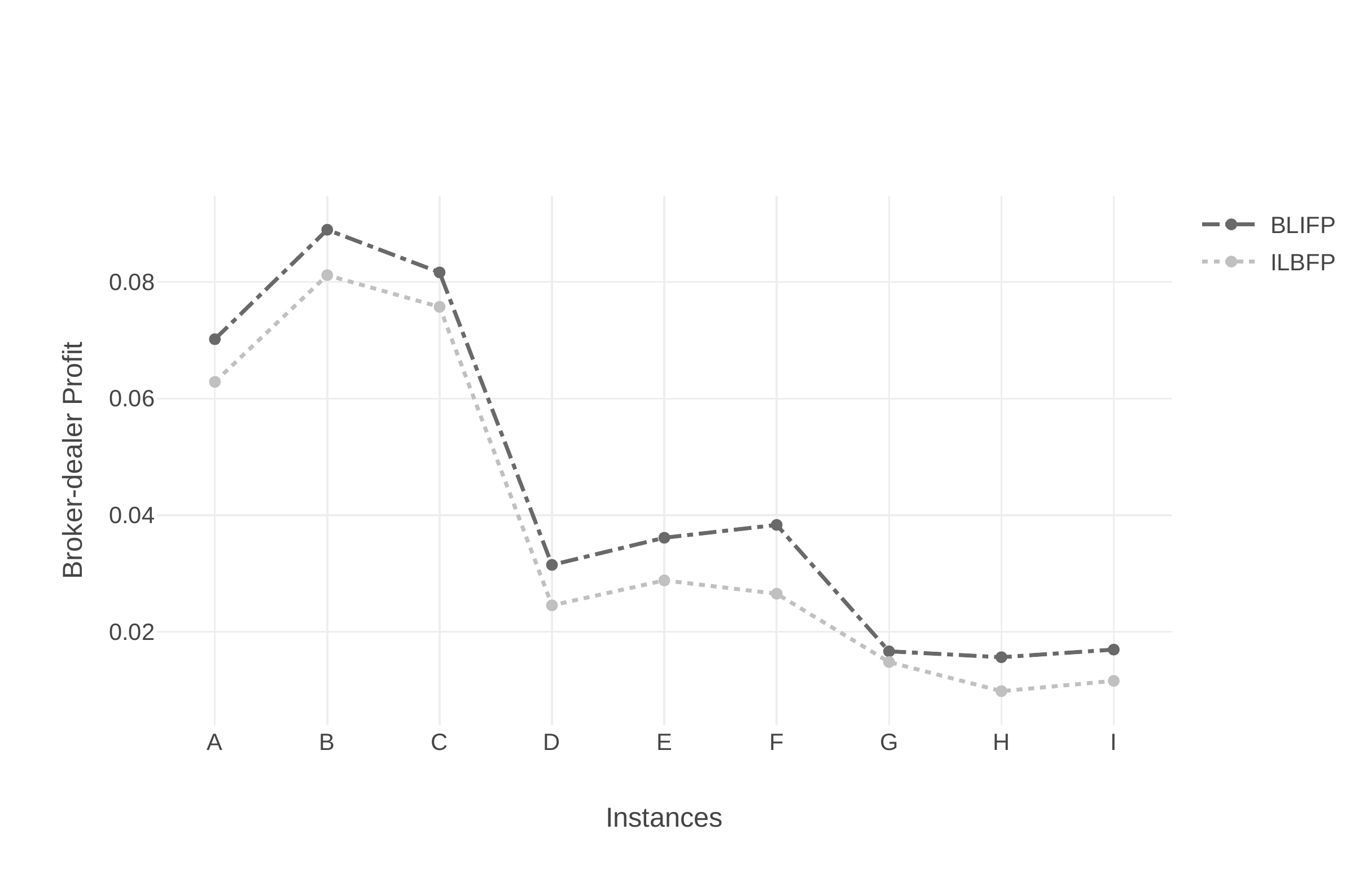}}

\end{multicols}
\caption{ {\footnotesize Values of the broker-dealer profit for {\textbf{BLIFP}} and {\textbf{ILBFP}} for $\alpha=0.1$ and $\mu_0=0.05$ (left) and for $\alpha=0.1$ and $\mu_0=0.1$ (right)}}
\end{figure}

\begin{figure}[H]
\centering

\begin{multicols}{2}

\fbox{\includegraphics[scale=0.35]{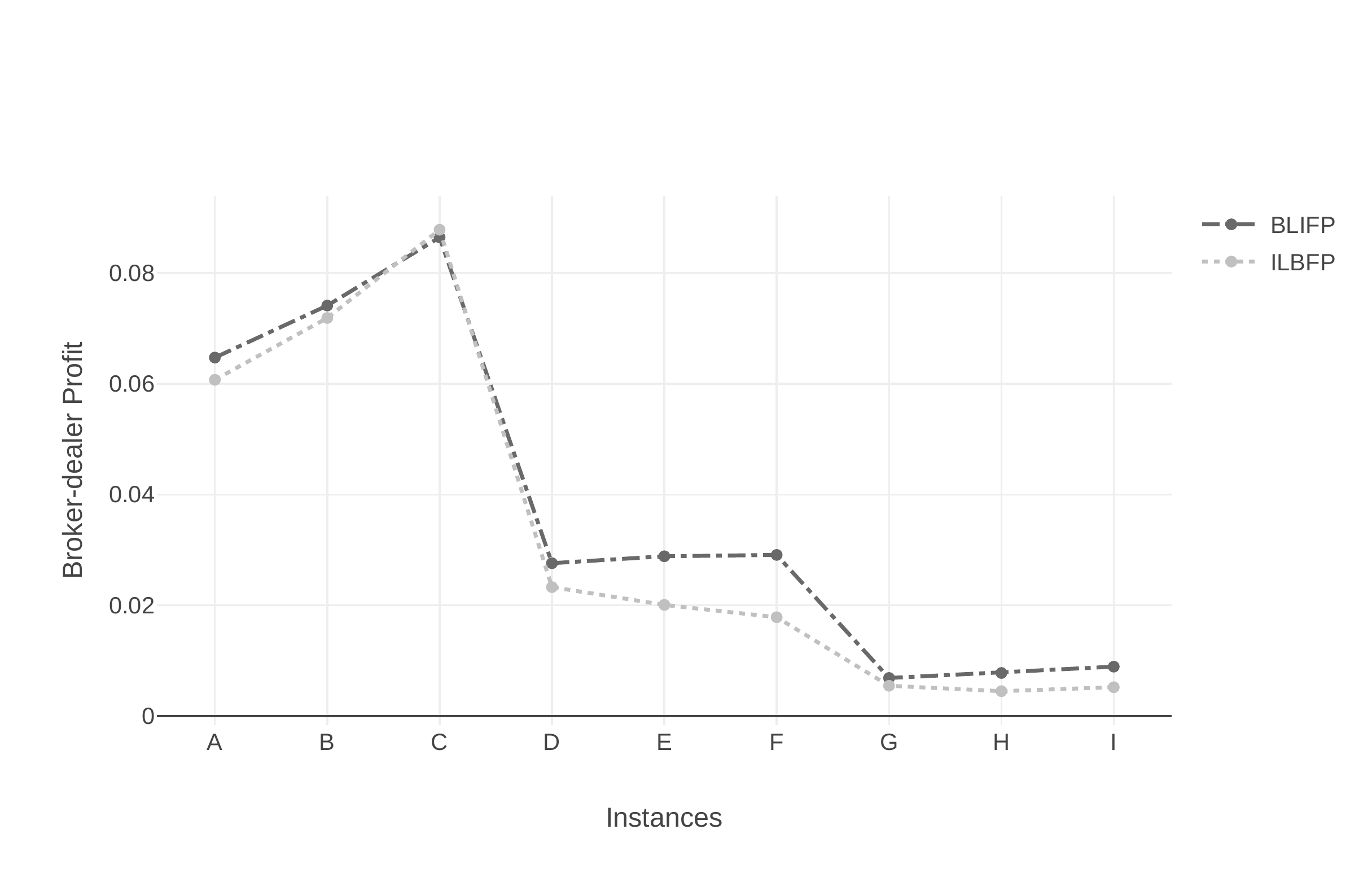}}

\fbox{\includegraphics[scale=0.35]{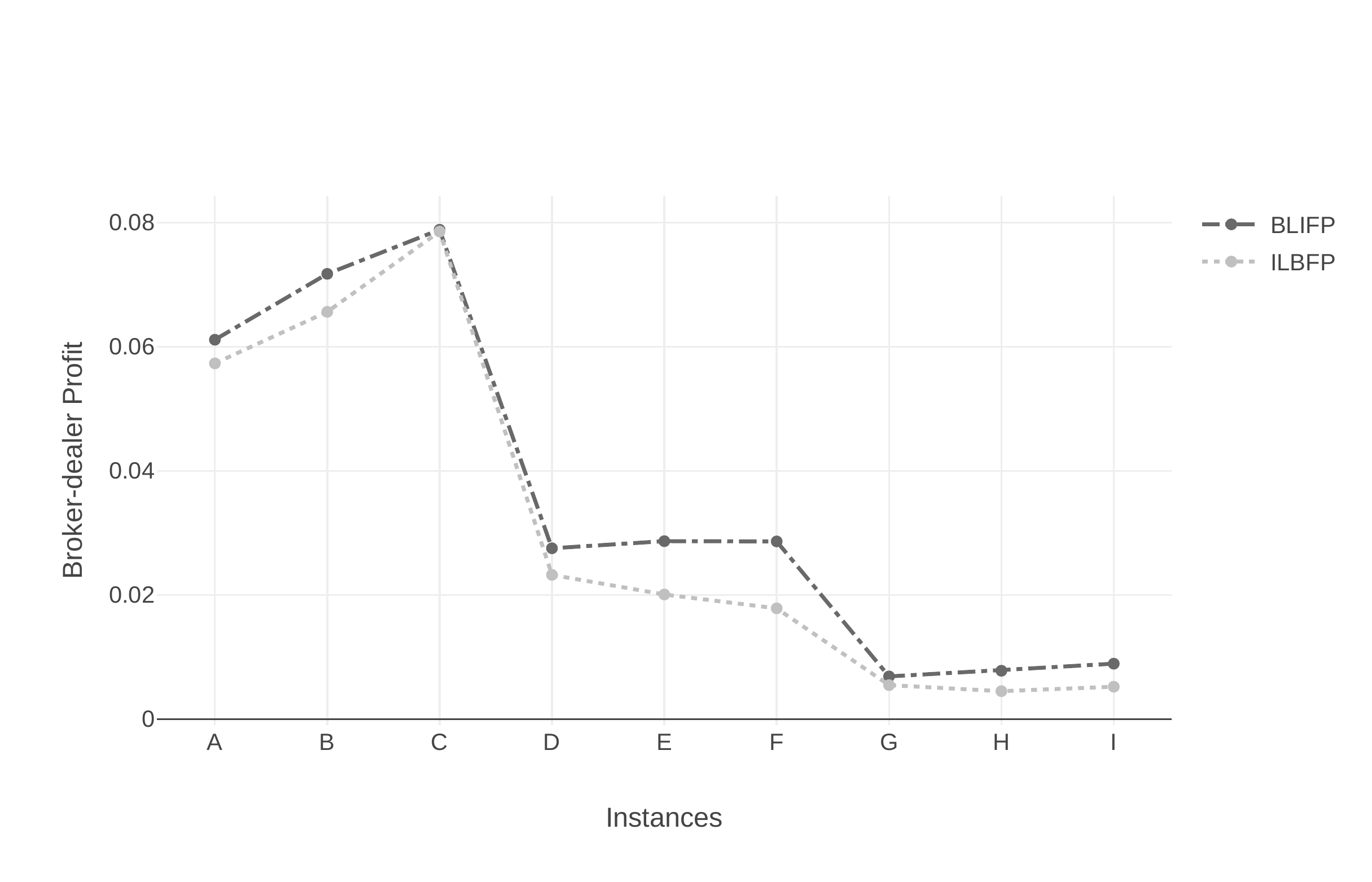}}

\end{multicols}
\caption{ {\footnotesize Values of the broker-dealer profit for {\textbf{BLIFP}} and {\textbf{ILBFP}} for $\alpha=0.5$ and $\mu_0=0$ (left) and for $\alpha=0.5$ and $\mu_0=0.05$ (right)}}
\end{figure}

\begin{figure}[H]
\centering

\begin{multicols}{2}

\fbox{\includegraphics[scale=0.35]{B05_010.pdf}}

\fbox{\includegraphics[scale=0.35]{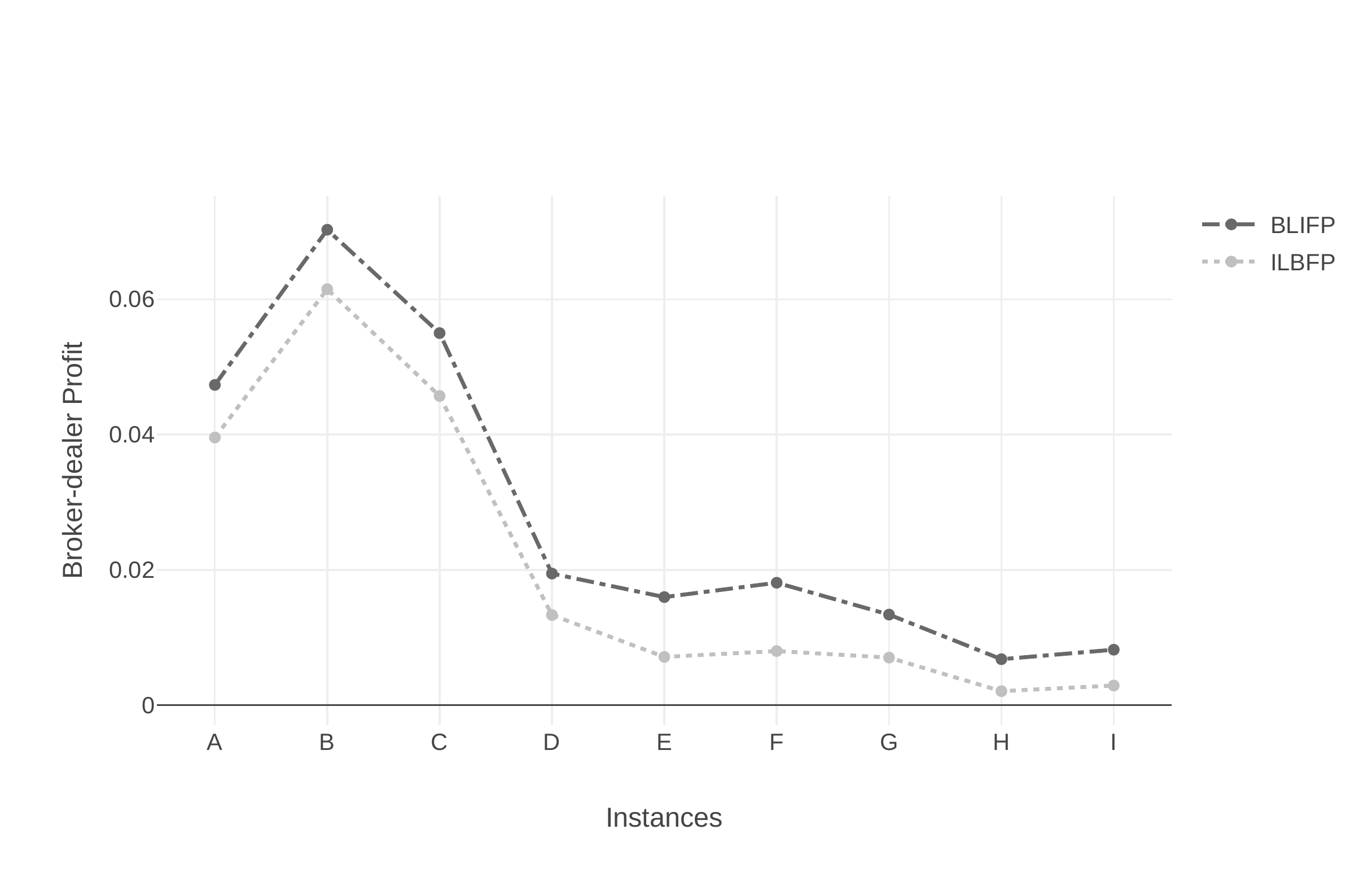}}

\end{multicols}
\caption{ {\footnotesize Values of the broker-dealer profit for {\textbf{BLIFP}} and {\textbf{ILBFP}} for $\alpha=0.5$ and $\mu_0=0.1$ (left) and for $\alpha=0.9$ and $\mu_0=0$ (right)}}
\end{figure}

\begin{figure}[H]
\centering

\begin{multicols}{2}

\fbox{\includegraphics[scale=0.35]{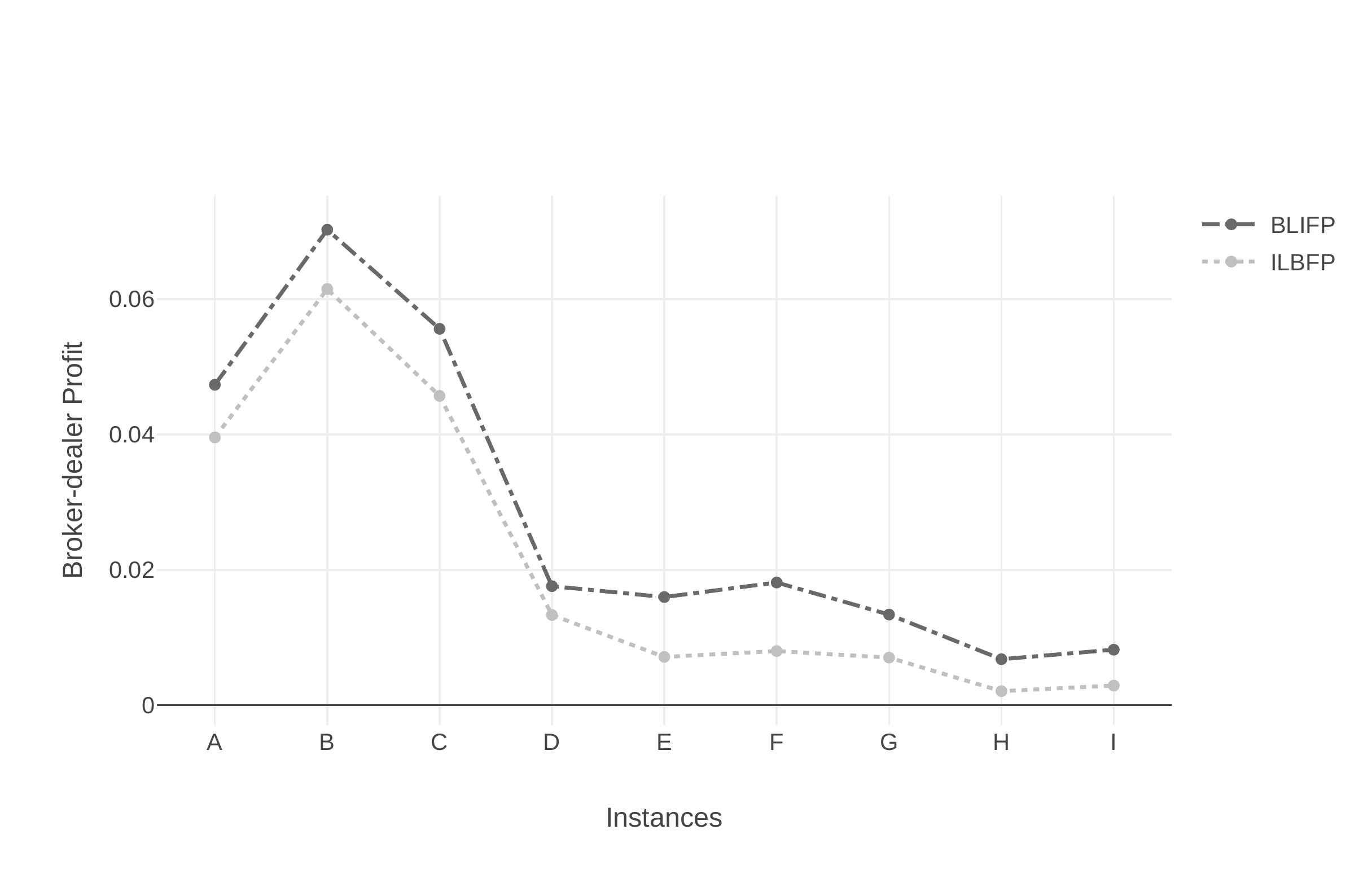}}

\fbox{\includegraphics[scale=0.35]{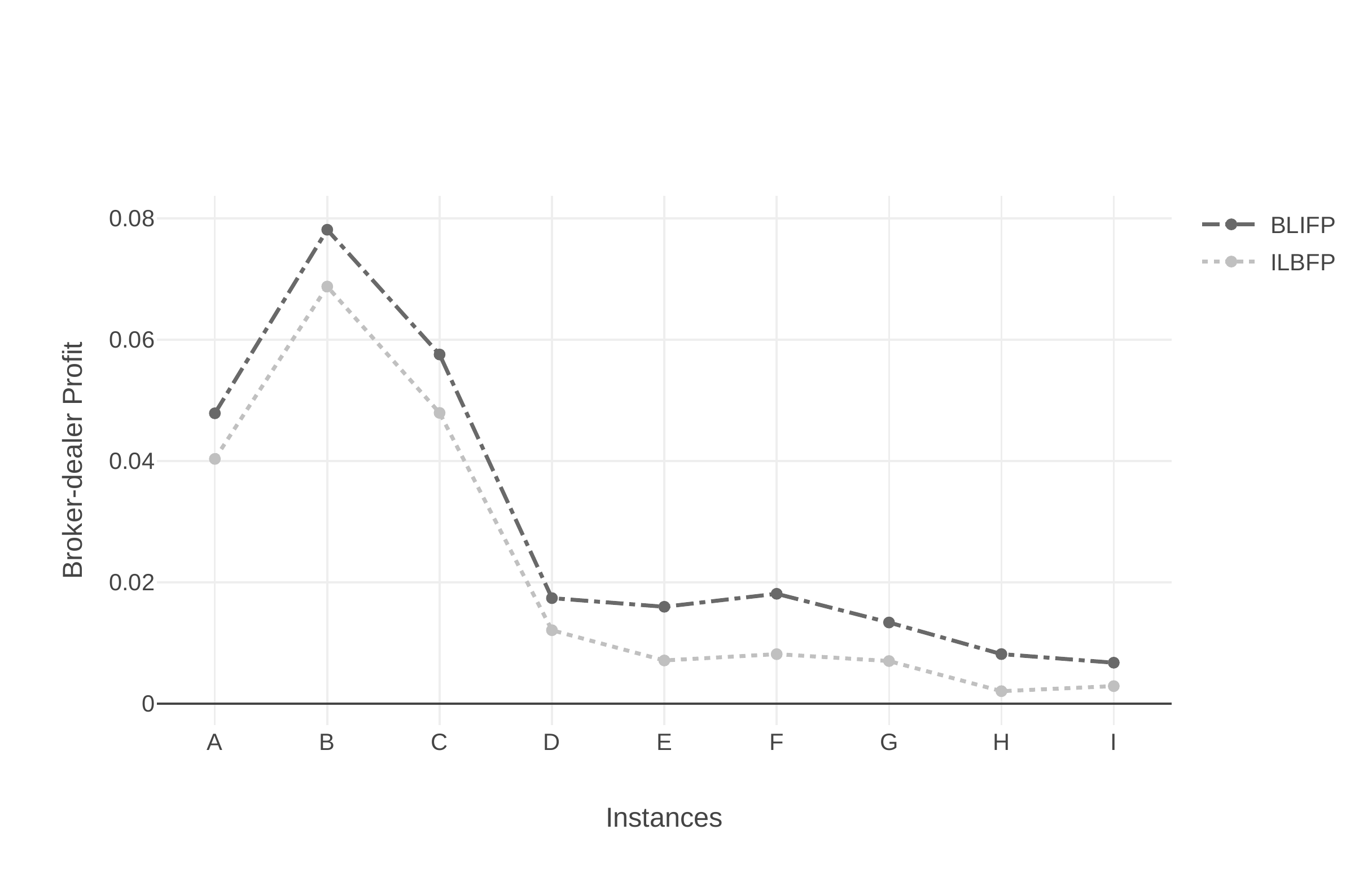}}

\end{multicols}
\caption{ {\footnotesize Values of the broker-dealer profit for {\textbf{BLIFP}} and {\textbf{ILBFP}} for $\alpha=0.9$ and $\mu_0=0.05$ (left) and for $\alpha=0.$ and $\mu_0=0.1$ (right)}}
\end{figure}

\subsubsection*{Broker-dealer profit + CVaR}
\begin{figure}[H]
\centering

\begin{multicols}{2}

\fbox{\includegraphics[scale=0.35]{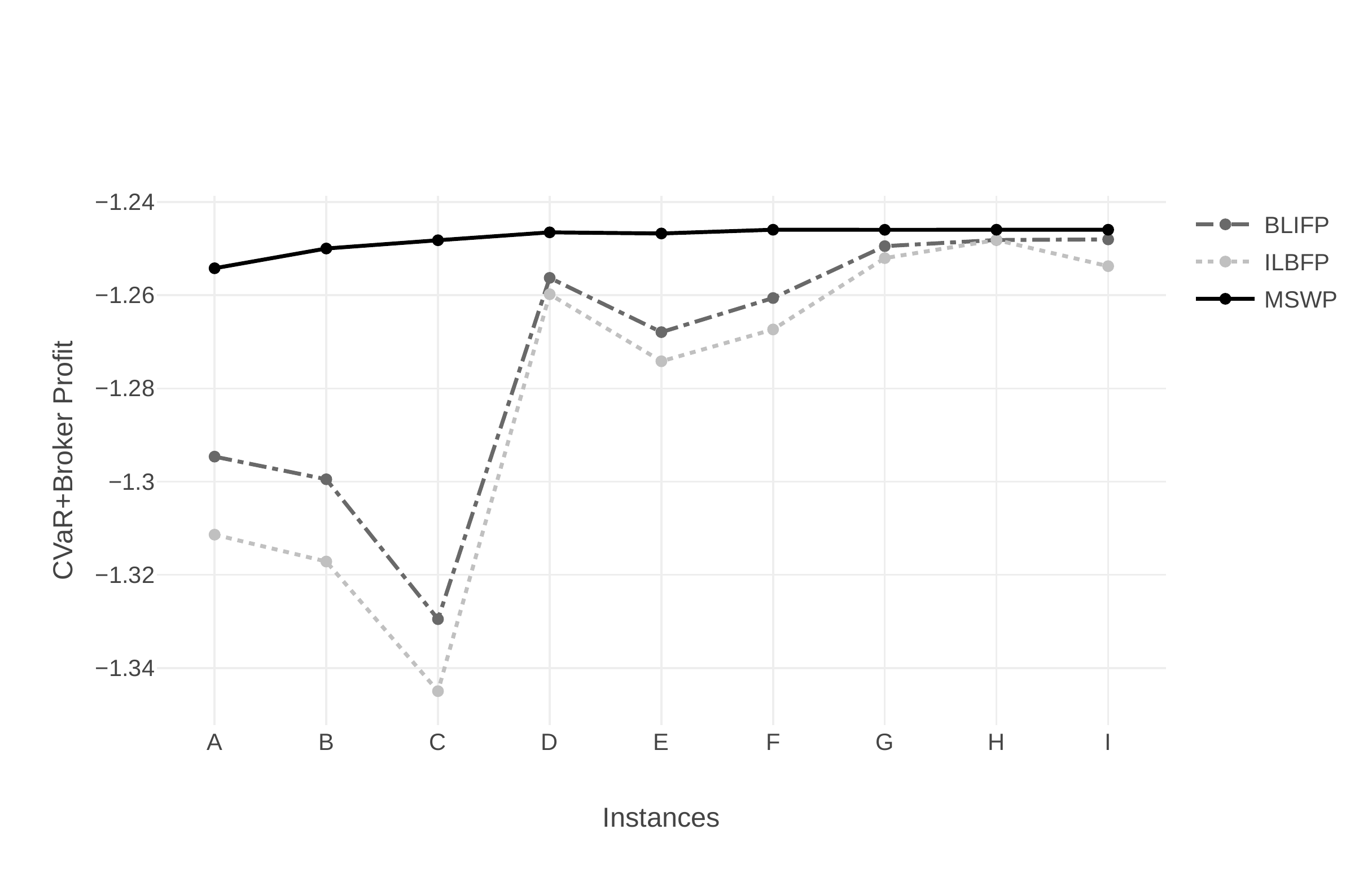}}

\fbox{\includegraphics[scale=0.35]{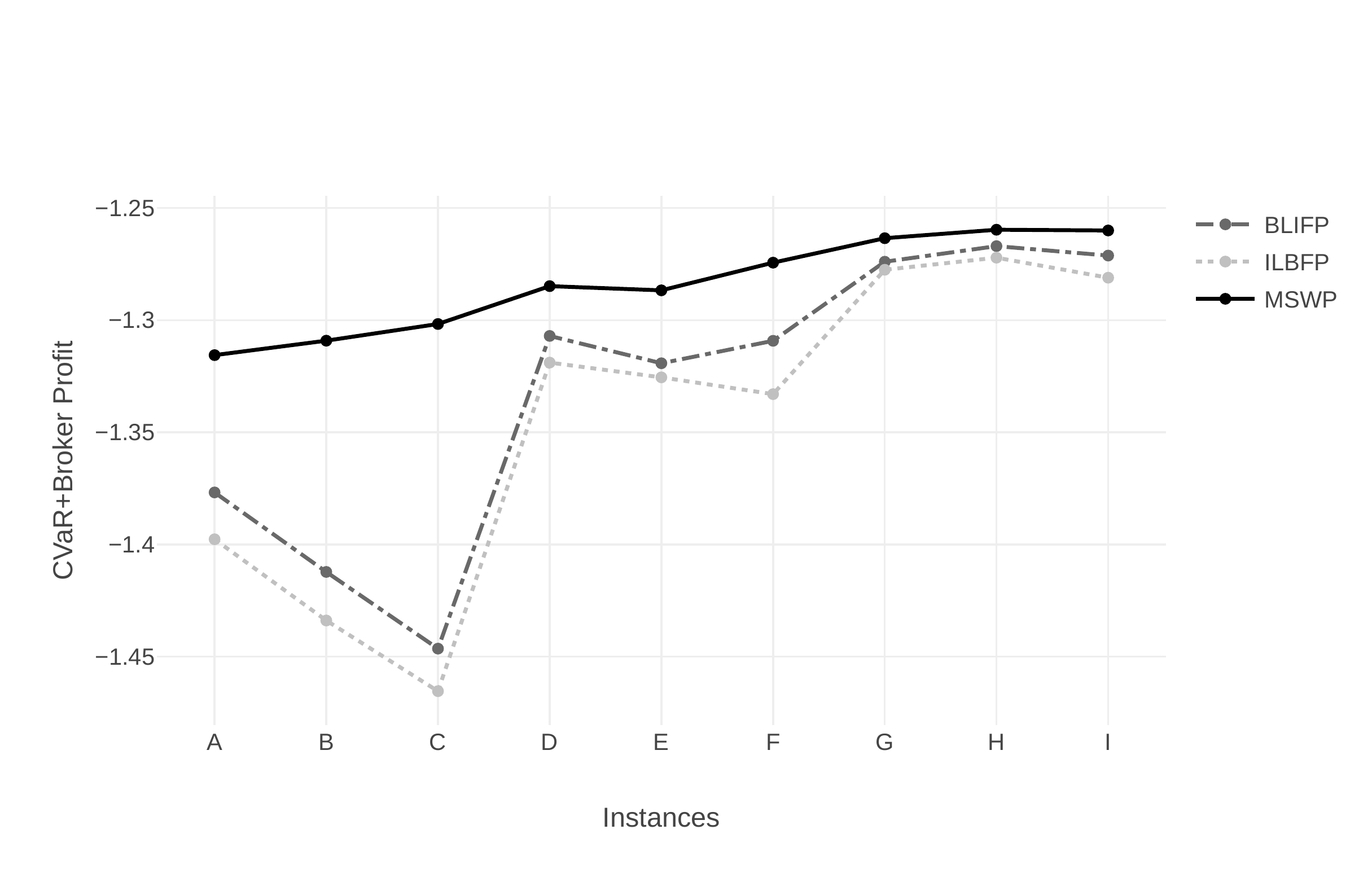}}

\end{multicols}
\caption{ {\footnotesize Values of the broker-dealer profit + CVaR for {\textbf{BLIFP}},  {\textbf{ILBFP} and \textbf{MSWP}} for $\alpha=0.05$ and $\mu_0=0$ (left) and for $\alpha=0.05$ and $\mu_0=0.05$ (right)}}\label{Grap:broker-dealer profit + CVaR_compare}
\end{figure}

\begin{figure}[H]
\centering

\begin{multicols}{2}

\fbox{\includegraphics[scale=0.35]{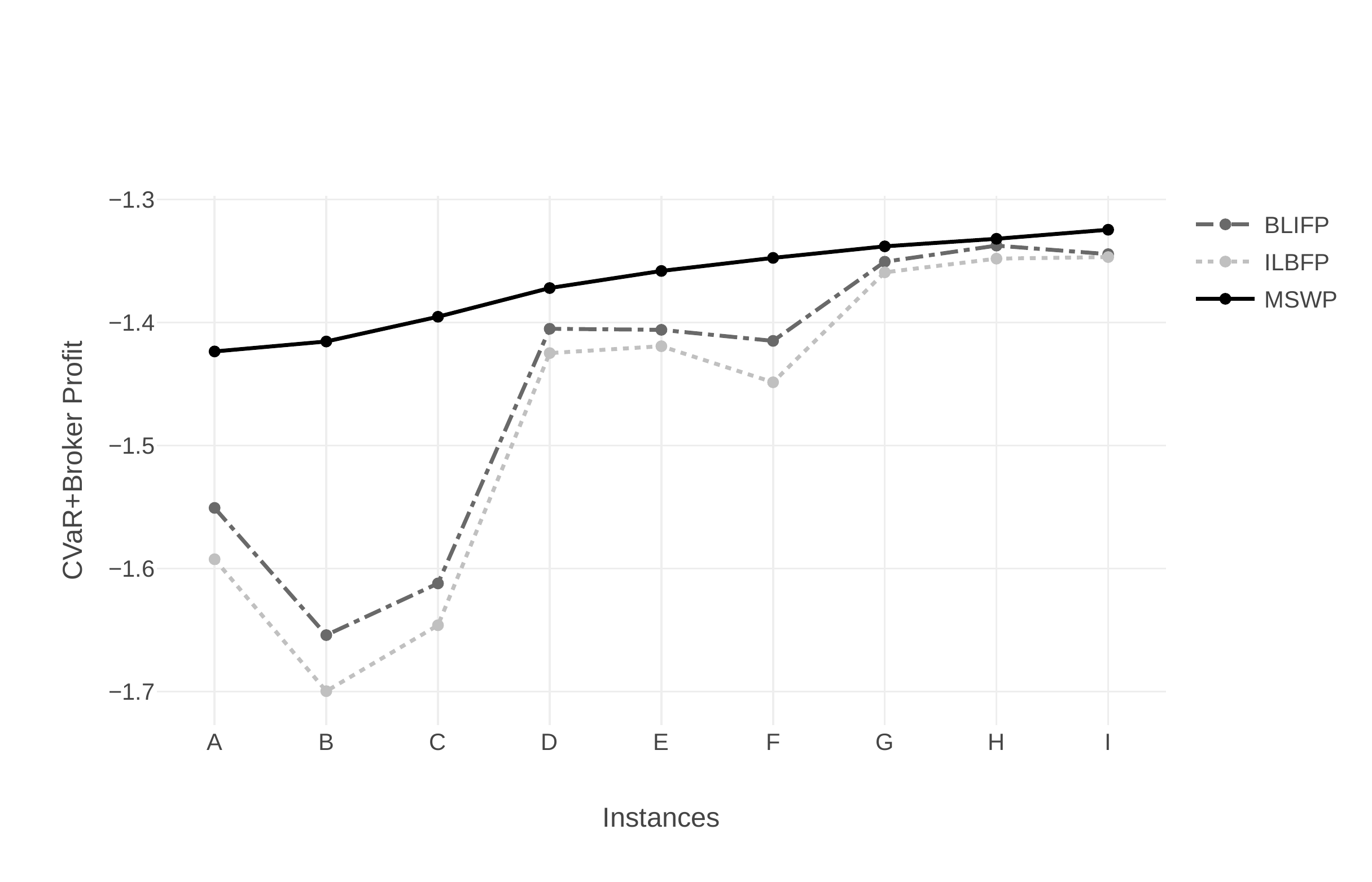}}

\fbox{\includegraphics[scale=0.35]{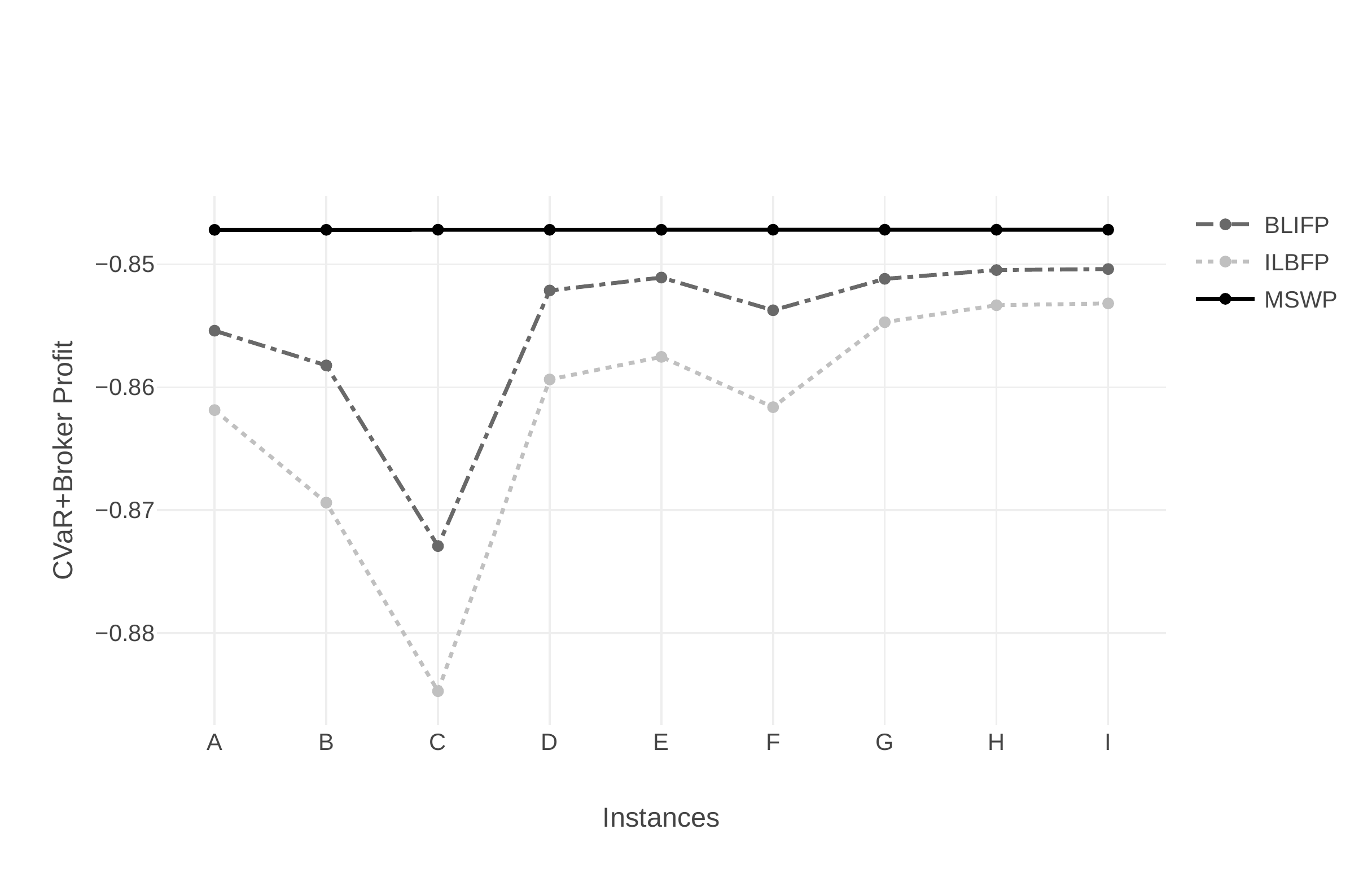}}

\end{multicols}
\caption{ {\footnotesize Values of the broker-dealer profit + CVaR for {\textbf{BLIFP}},  {\textbf{ILBFP} and \textbf{MSWP}} for $\alpha=0.05$ and $\mu_0=0.1$ (left) and for $\alpha=0.1$ and $\mu_0=0$ (right)}}
\end{figure}

\begin{figure}[H]
\centering

\begin{multicols}{2}

\fbox{\includegraphics[scale=0.35]{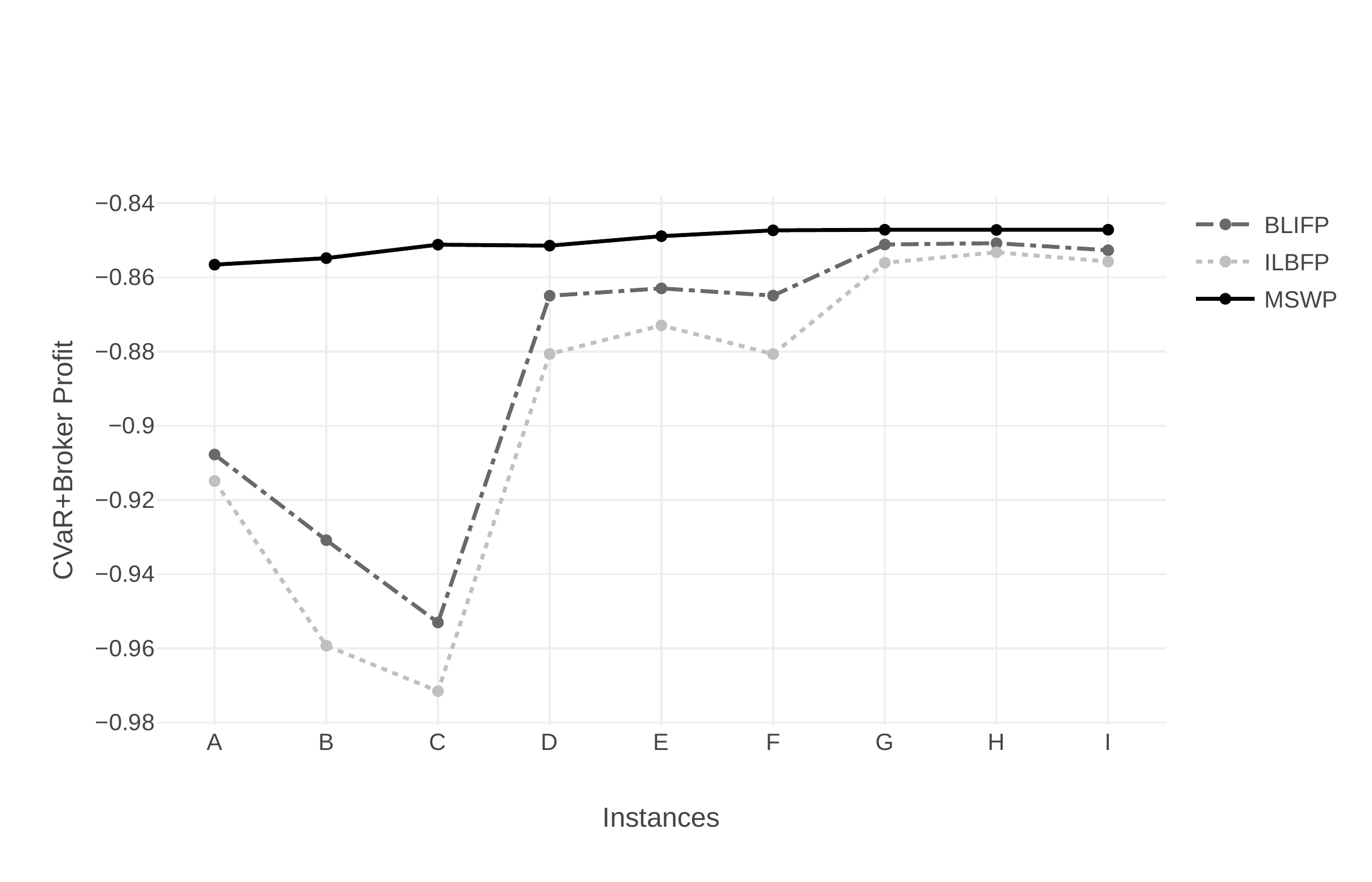}}

\fbox{\includegraphics[scale=0.35]{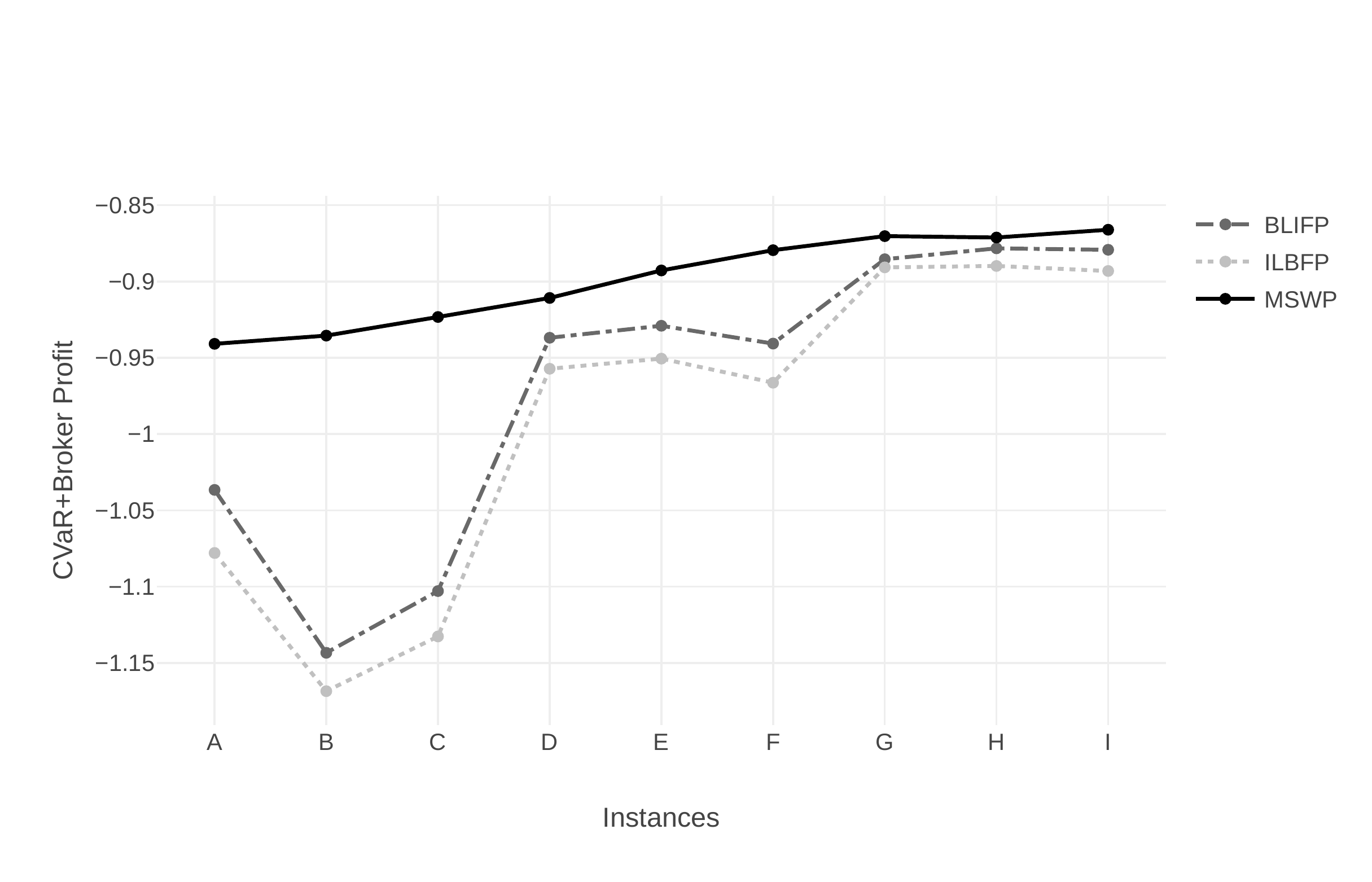}}

\end{multicols}
\caption{ {\footnotesize Values of the broker-dealer profit + CVaR for {\textbf{BLIFP}},  {\textbf{ILBFP} and \textbf{MSWP}} for $\alpha=0.1$ and $\mu_0=0.05$ (left) and for $\alpha=0.1$ and $\mu_0=0.1$ (right)}}
\end{figure}

\begin{figure}[H]
\centering

\begin{multicols}{2}

\fbox{\includegraphics[scale=0.35]{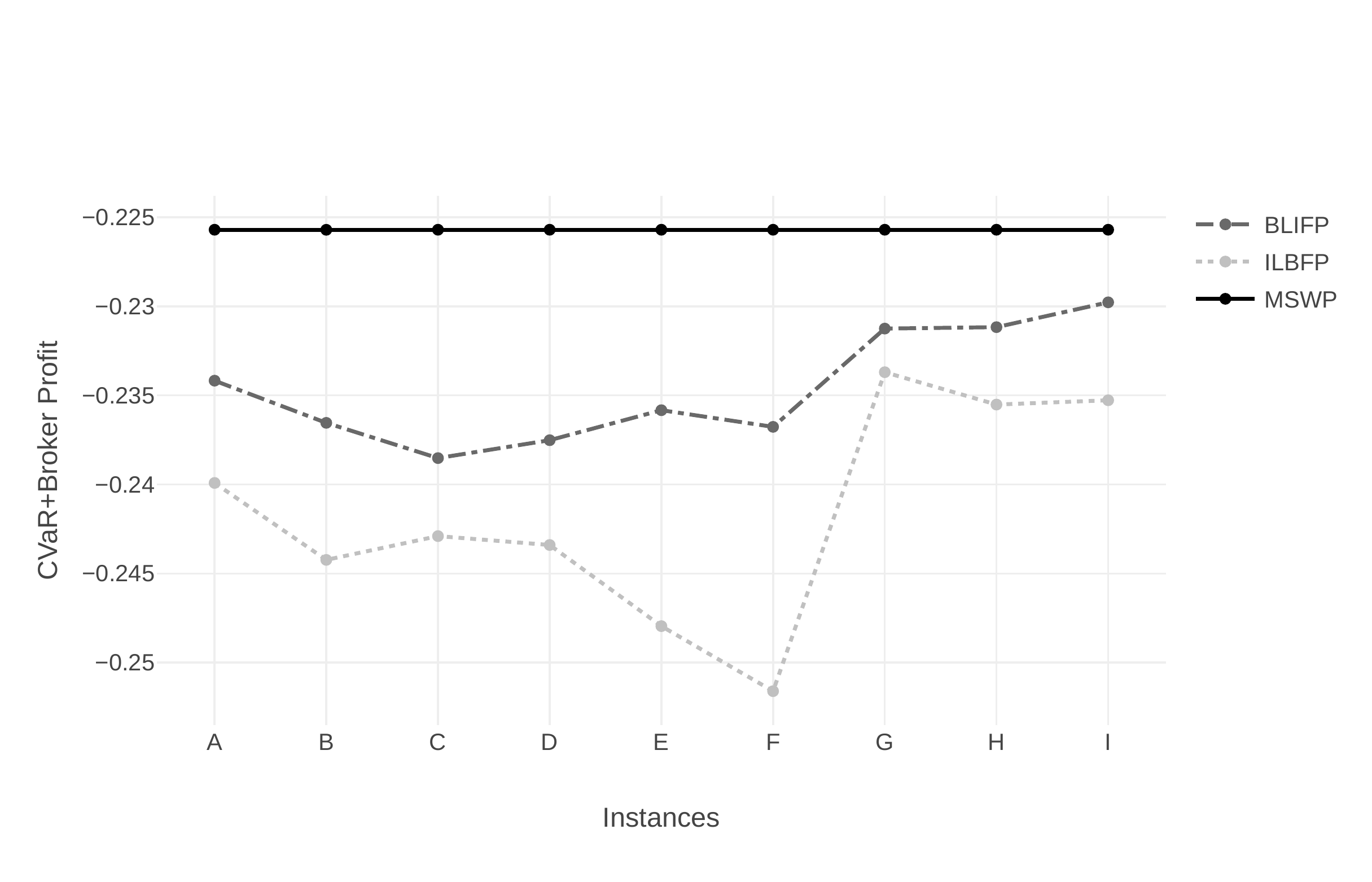}}

\fbox{\includegraphics[scale=0.35]{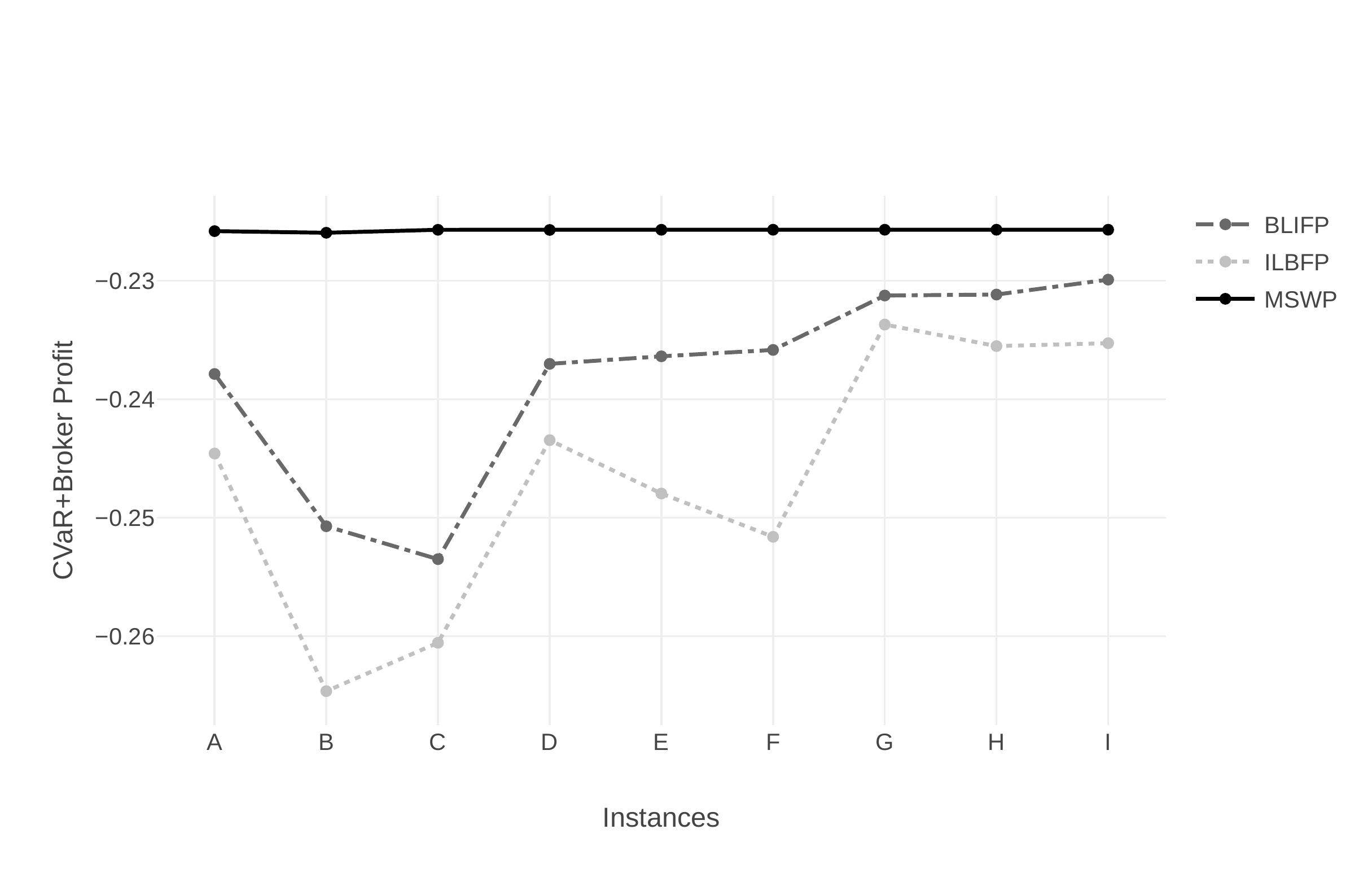}}

\end{multicols}
\caption{ {\footnotesize Values of the broker-dealer profit + CVaR for {\textbf{BLIFP}},  {\textbf{ILBFP} and \textbf{MSWP}} for $\alpha=0.5$ and $\mu_0=0$ (left) and for $\alpha=0.5$ and $\mu_0=0.05$ (right)}}
\end{figure}

\begin{figure}[H]
\centering

\begin{multicols}{2}

\fbox{\includegraphics[scale=0.35]{S05_010.pdf}}

\fbox{\includegraphics[scale=0.35]{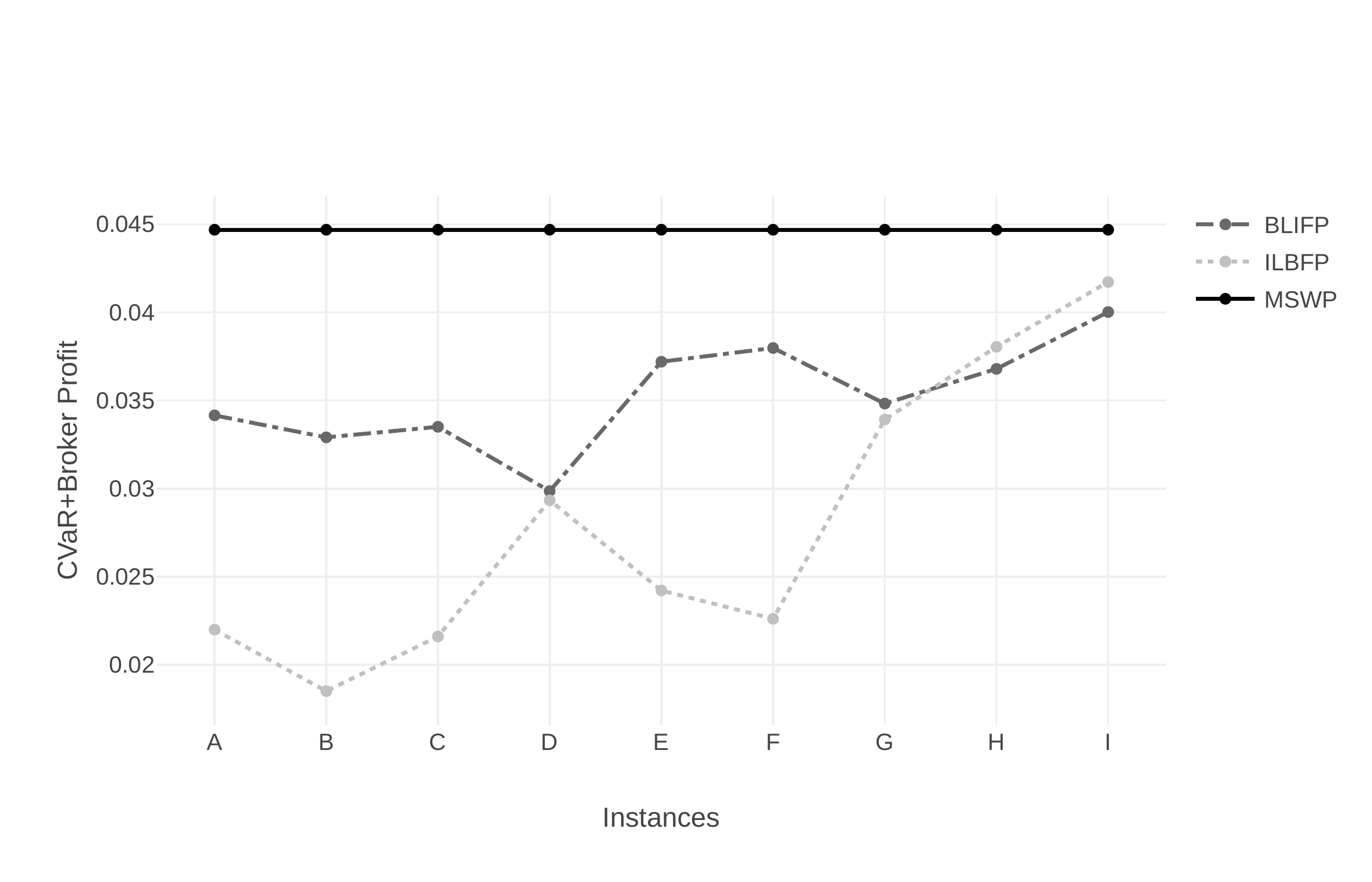}}

\end{multicols}
\caption{ {\footnotesize Values of the broker-dealer profit + CVaR for {\textbf{BLIFP}},  {\textbf{ILBFP} and \textbf{MSWP}} for $\alpha=0.5$ and $\mu_0=0.1$ (left) and for $\alpha=0.9$ and $\mu_0=0$ (right)}}
\end{figure}

\begin{figure}[H]
\centering

\begin{multicols}{2}

\fbox{\includegraphics[scale=0.35]{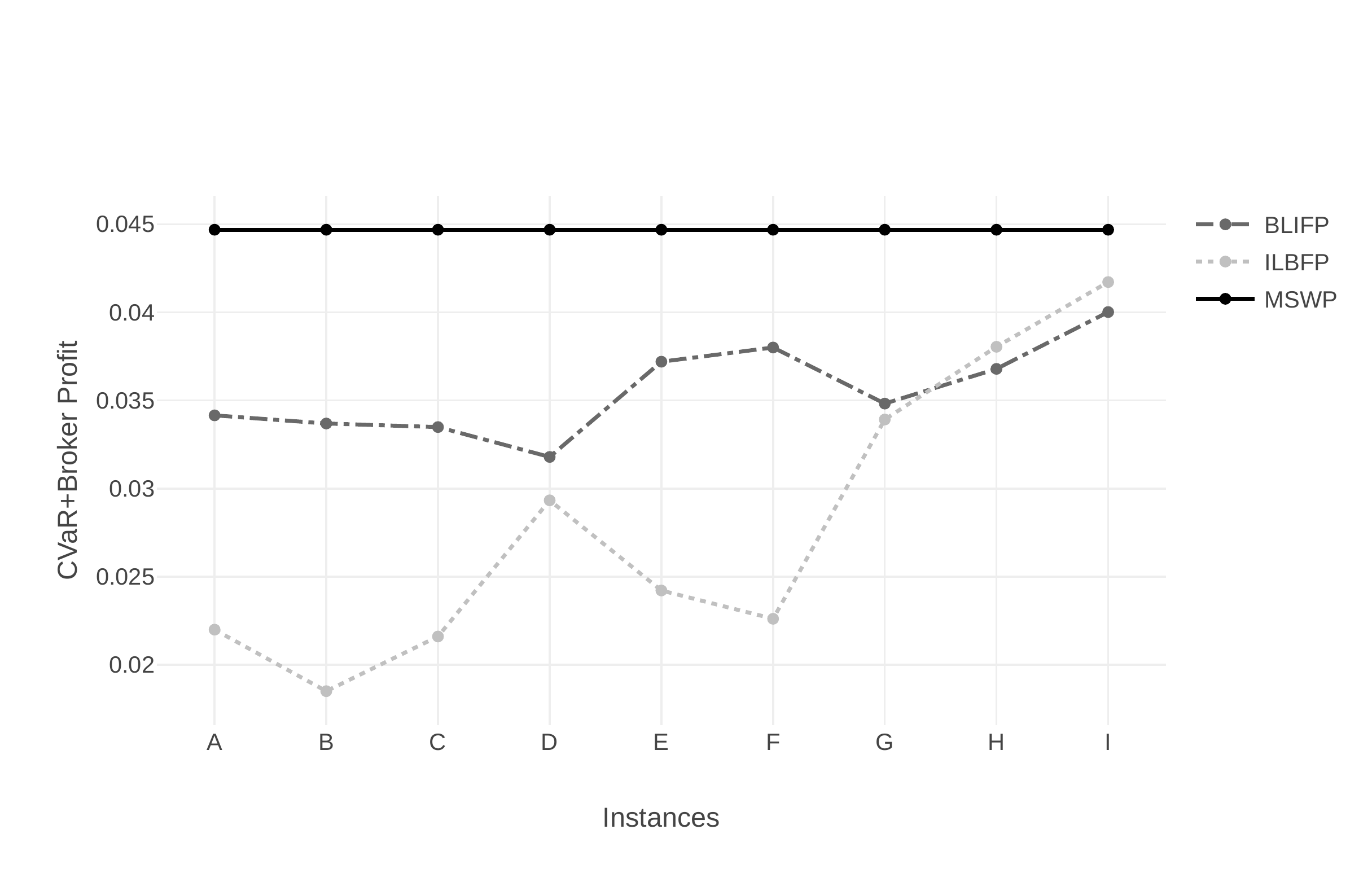}}

\fbox{\includegraphics[scale=0.35]{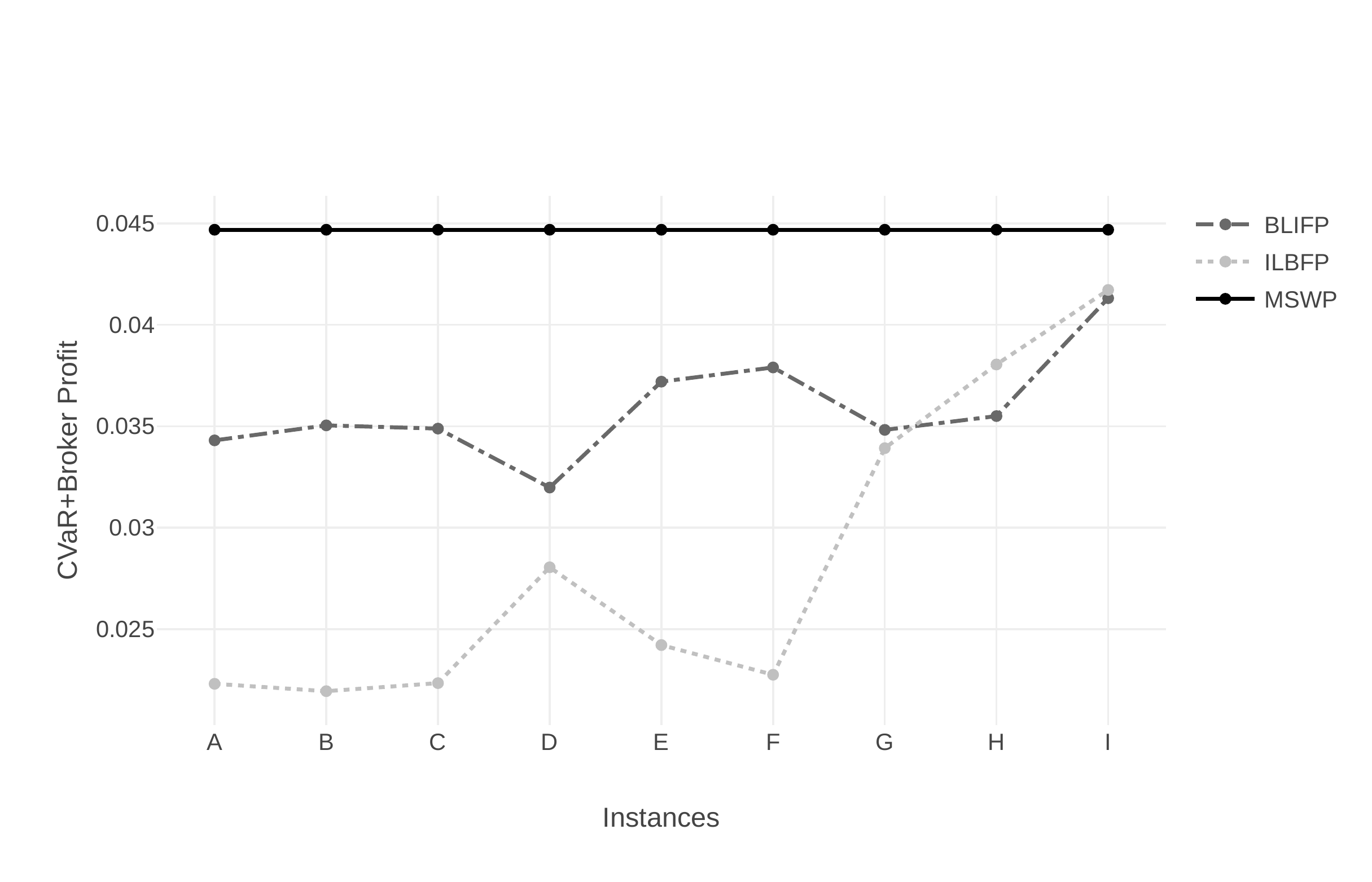}}

\end{multicols}
\caption{ {\footnotesize Values of the broker-dealer profit + CVaR for {\textbf{BLIFP}},  {\textbf{ILBFP} and \textbf{MSWP}} for $\alpha=0.9$ and $\mu_0=0.05$ (left) and for $\alpha=0.$ and $\mu_0=0.1$ (right)}}
\end{figure}

\subsubsection*{Expected Return}
\begin{figure}[H]
\centering

\begin{multicols}{2}

\fbox{\includegraphics[scale=0.35]{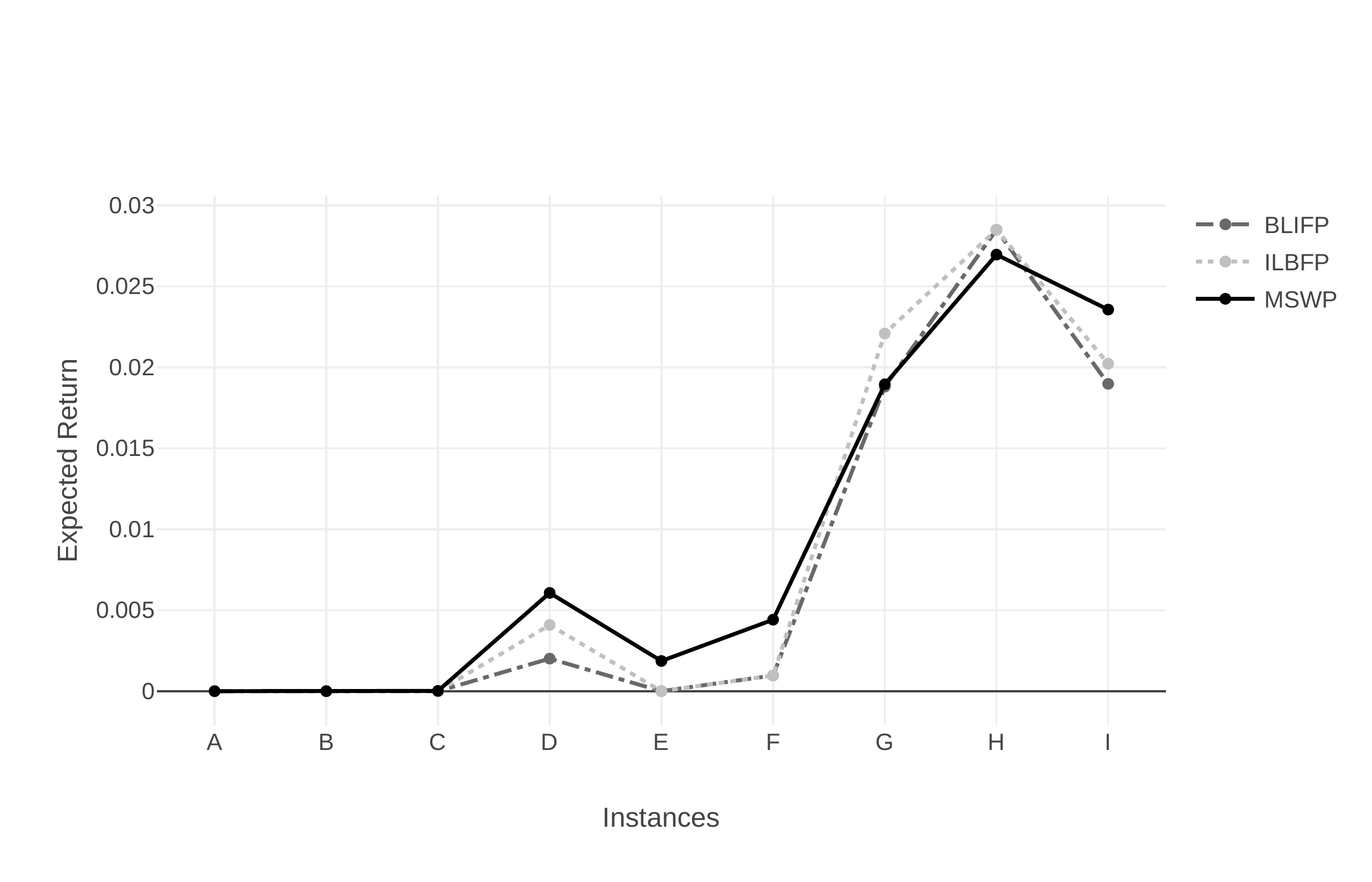}}

\fbox{\includegraphics[scale=0.35]{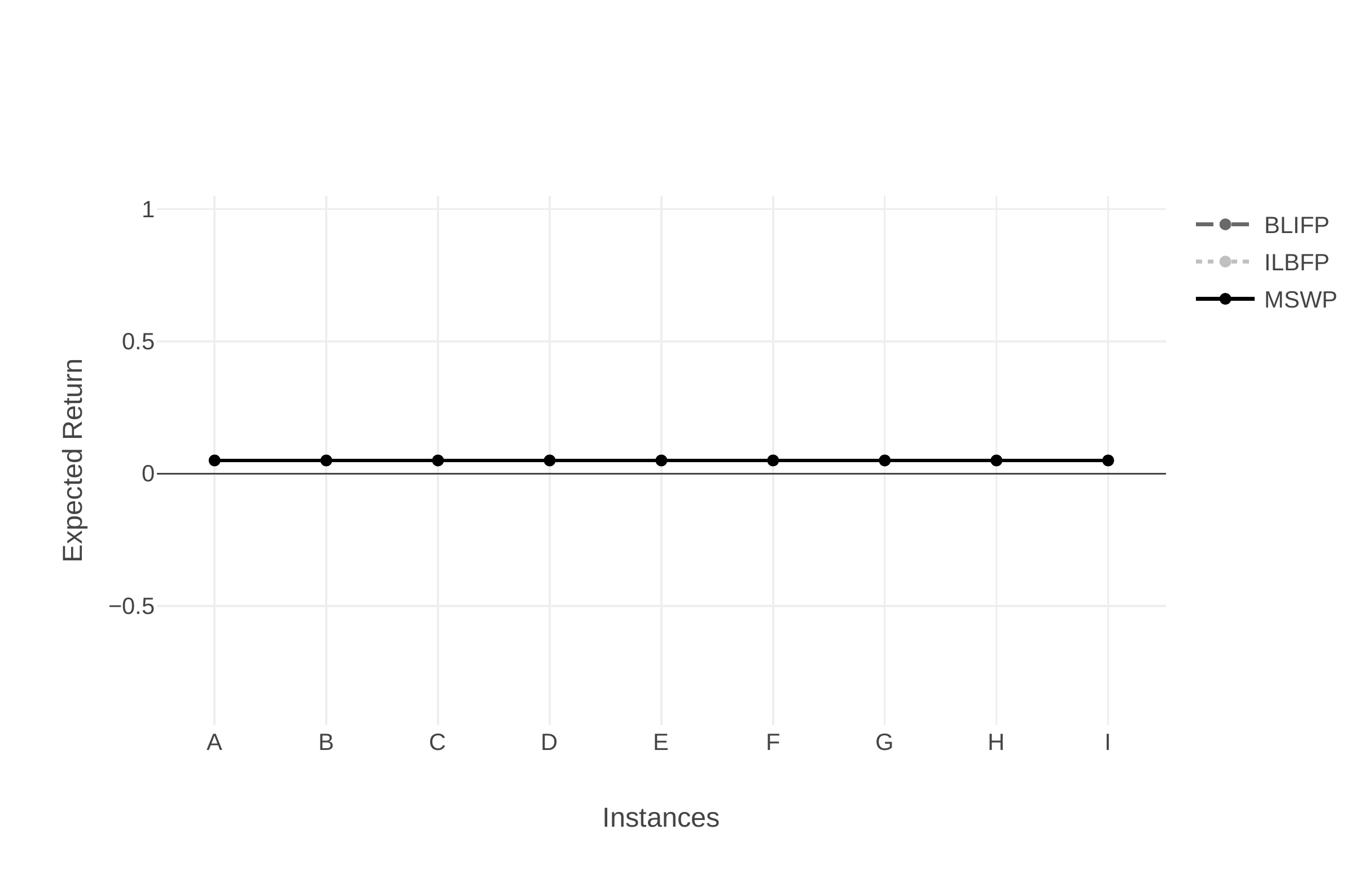}}

\end{multicols}
\caption{ {\footnotesize Expected return for {\textbf{BLIFP}},  {\textbf{ILBFP} and \textbf{MSWP}} for $\alpha=0.05$ and $\mu_0=0$ (left) and for $\alpha=0.05$ and $\mu_0=0.05$ (right)}}\label{Grap:broker-dealer profit + CVaR_compare}
\end{figure}

\begin{figure}[H]
\centering

\begin{multicols}{2}

\fbox{\includegraphics[scale=0.35]{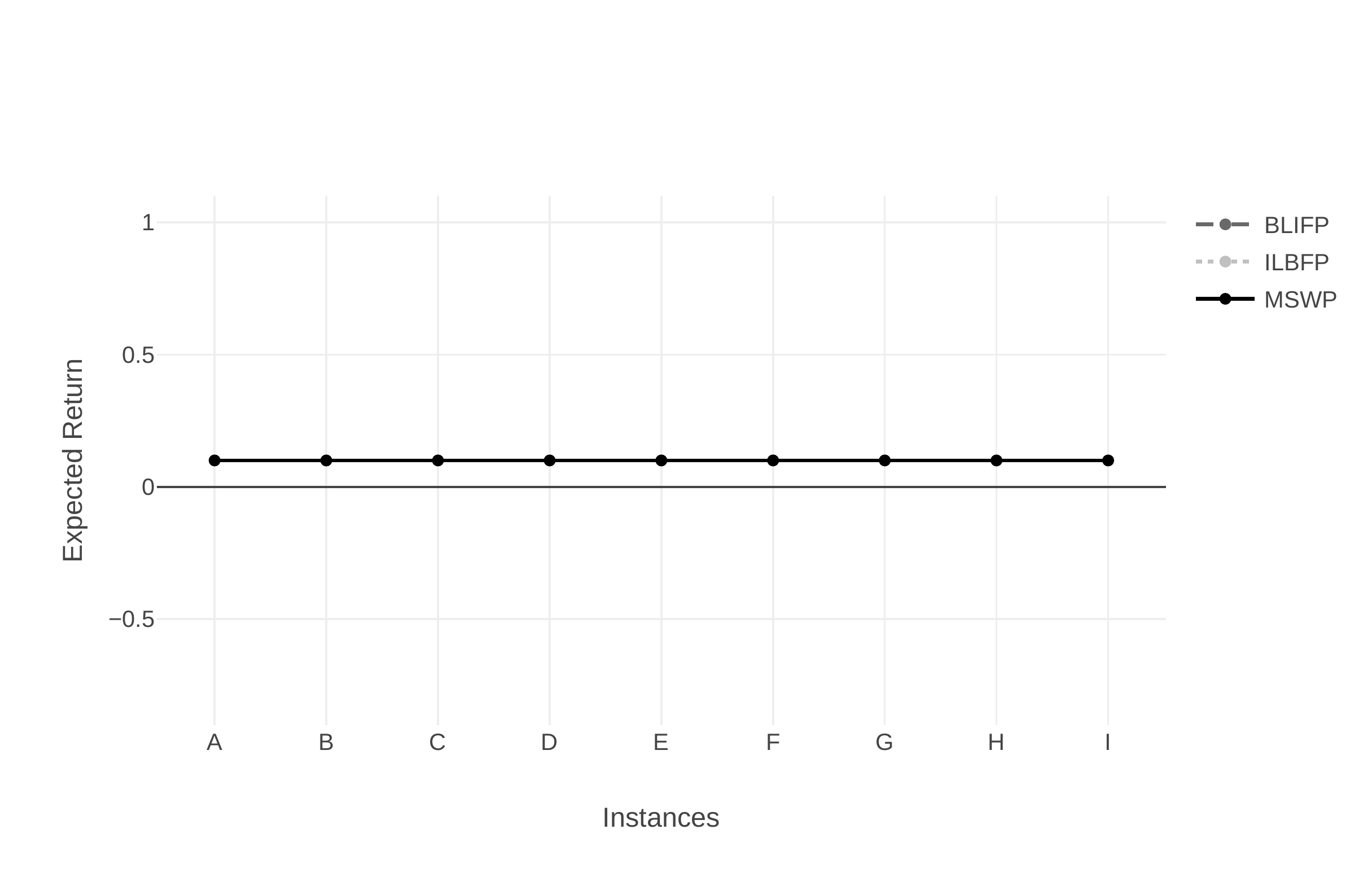}}

\fbox{\includegraphics[scale=0.35]{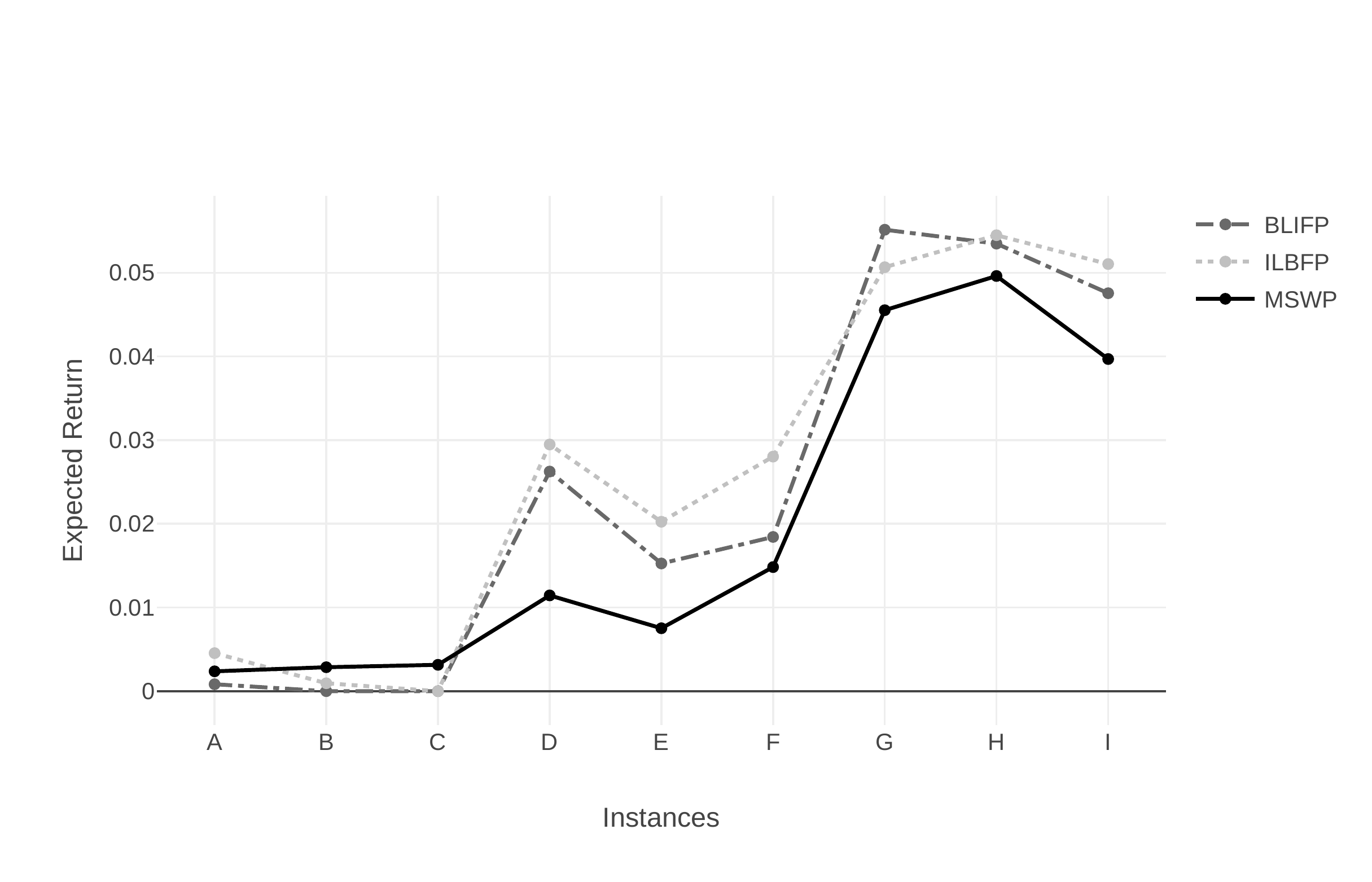}}

\end{multicols}
\caption{ {\footnotesize Expected return for {\textbf{BLIFP}},  {\textbf{ILBFP} and \textbf{MSWP}} for $\alpha=0.05$ and $\mu_0=0.1$ (left) and for $\alpha=0.1$ and $\mu_0=0$ (right)}}
\end{figure}

\begin{figure}[H]
\centering

\begin{multicols}{2}

\fbox{\includegraphics[scale=0.35]{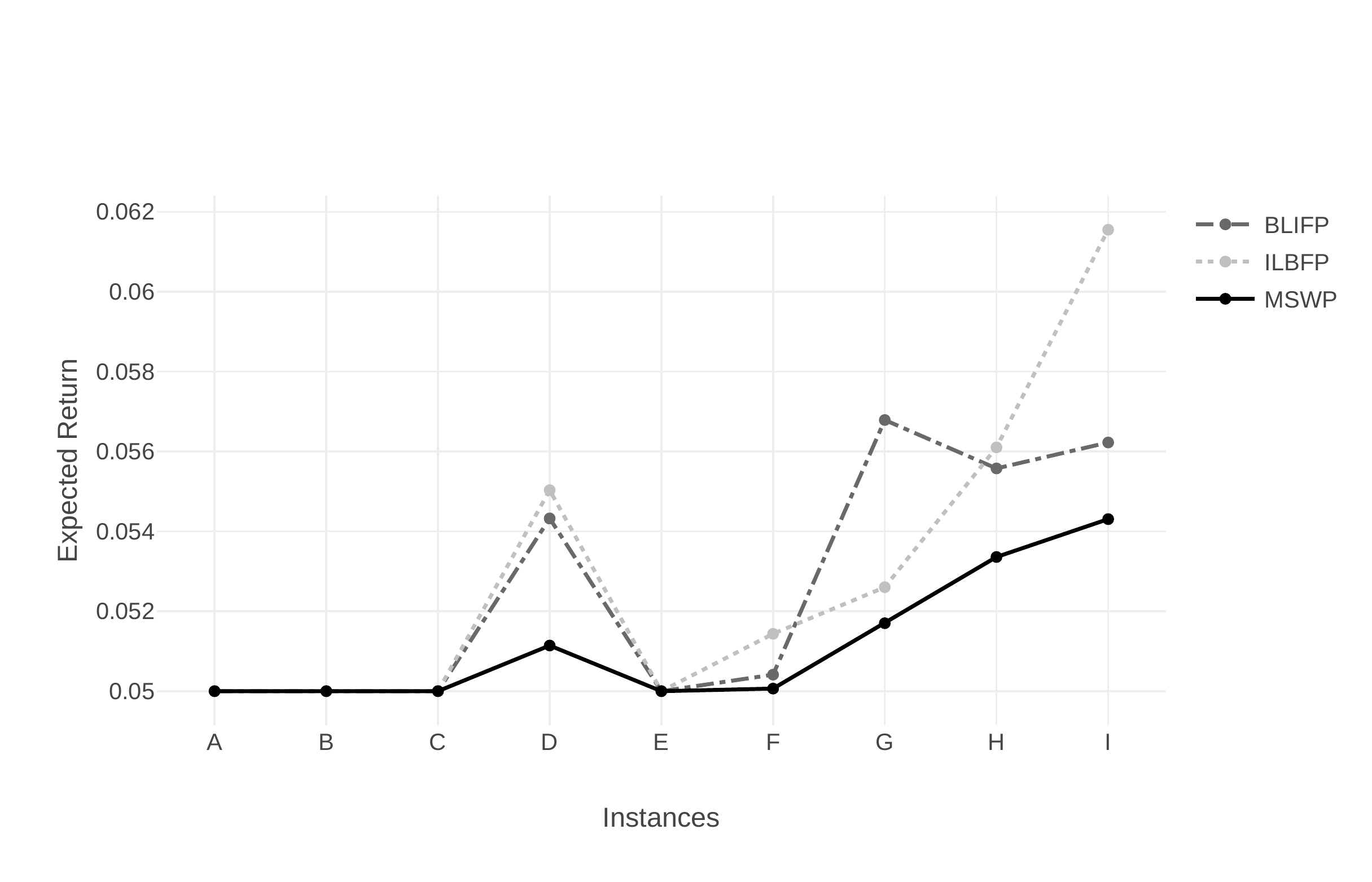}}

\fbox{\includegraphics[scale=0.35]{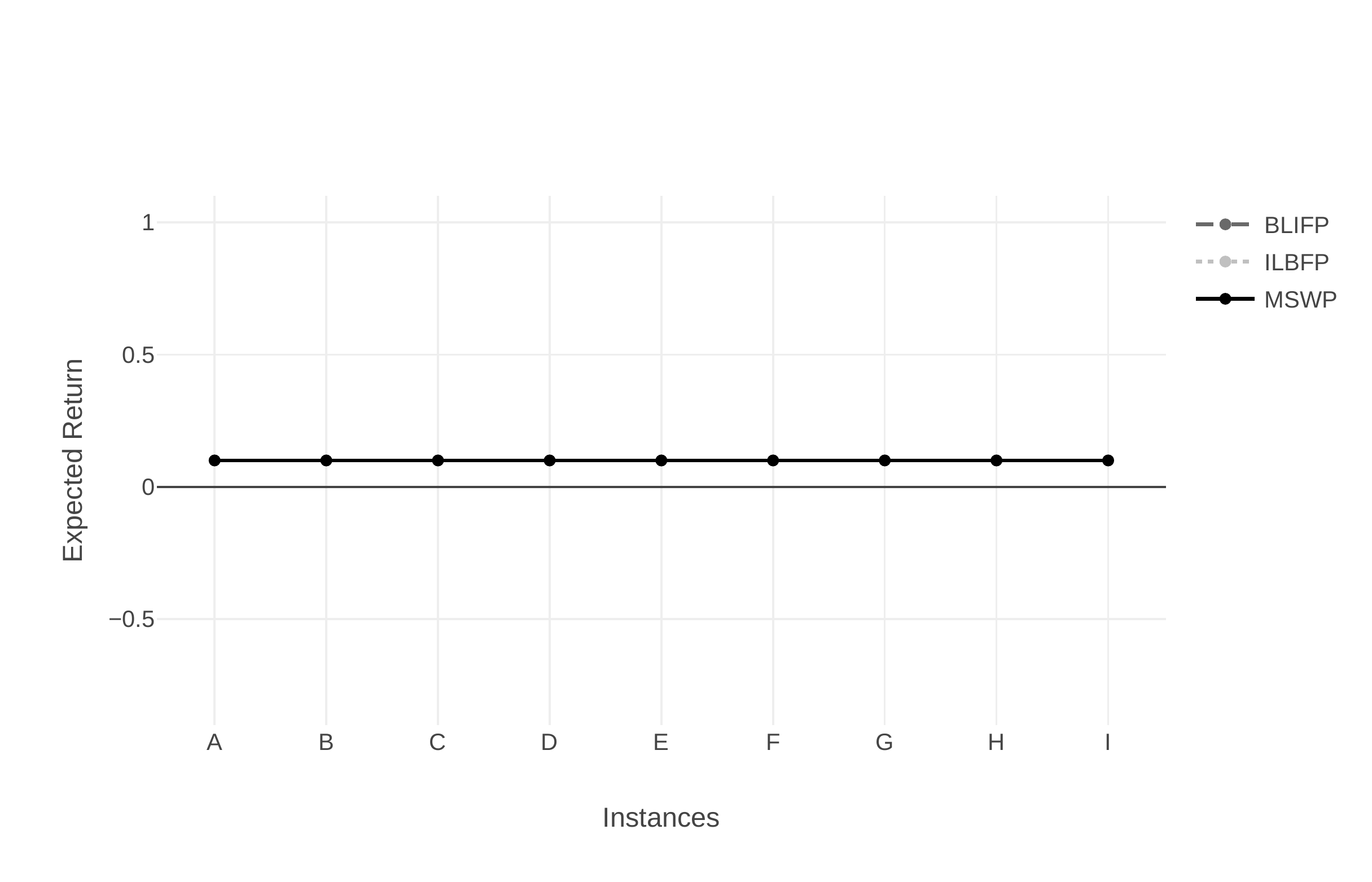}}

\end{multicols}
\caption{ {\footnotesize Expected return for {\textbf{BLIFP}},  {\textbf{ILBFP} and \textbf{MSWP}} for $\alpha=0.1$ and $\mu_0=0.05$ (left) and for $\alpha=0.1$ and $\mu_0=0.1$ (right)}}
\end{figure}

\begin{figure}[H]
\centering

\begin{multicols}{2}

\fbox{\includegraphics[scale=0.35]{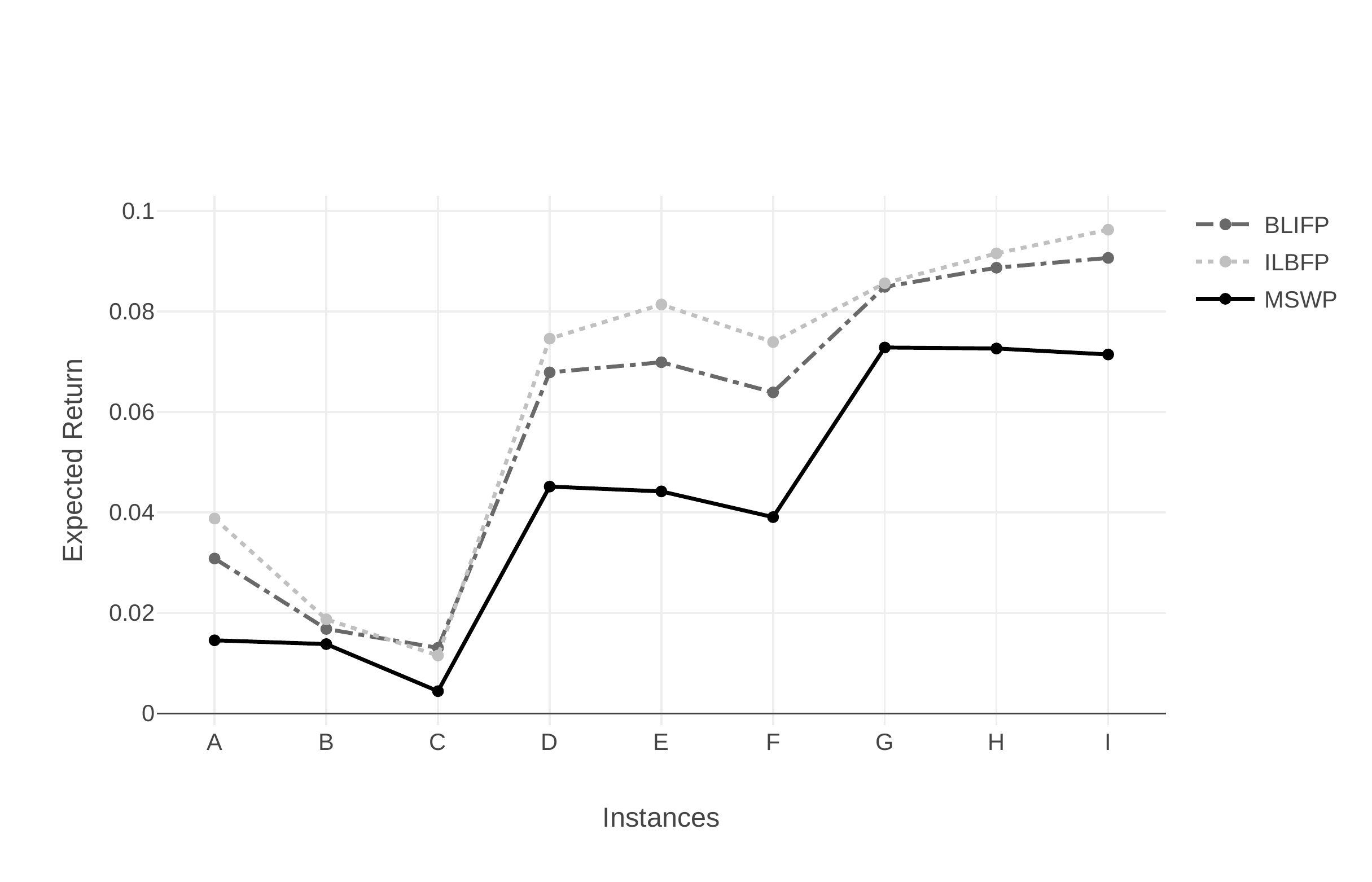}}

\fbox{\includegraphics[scale=0.35]{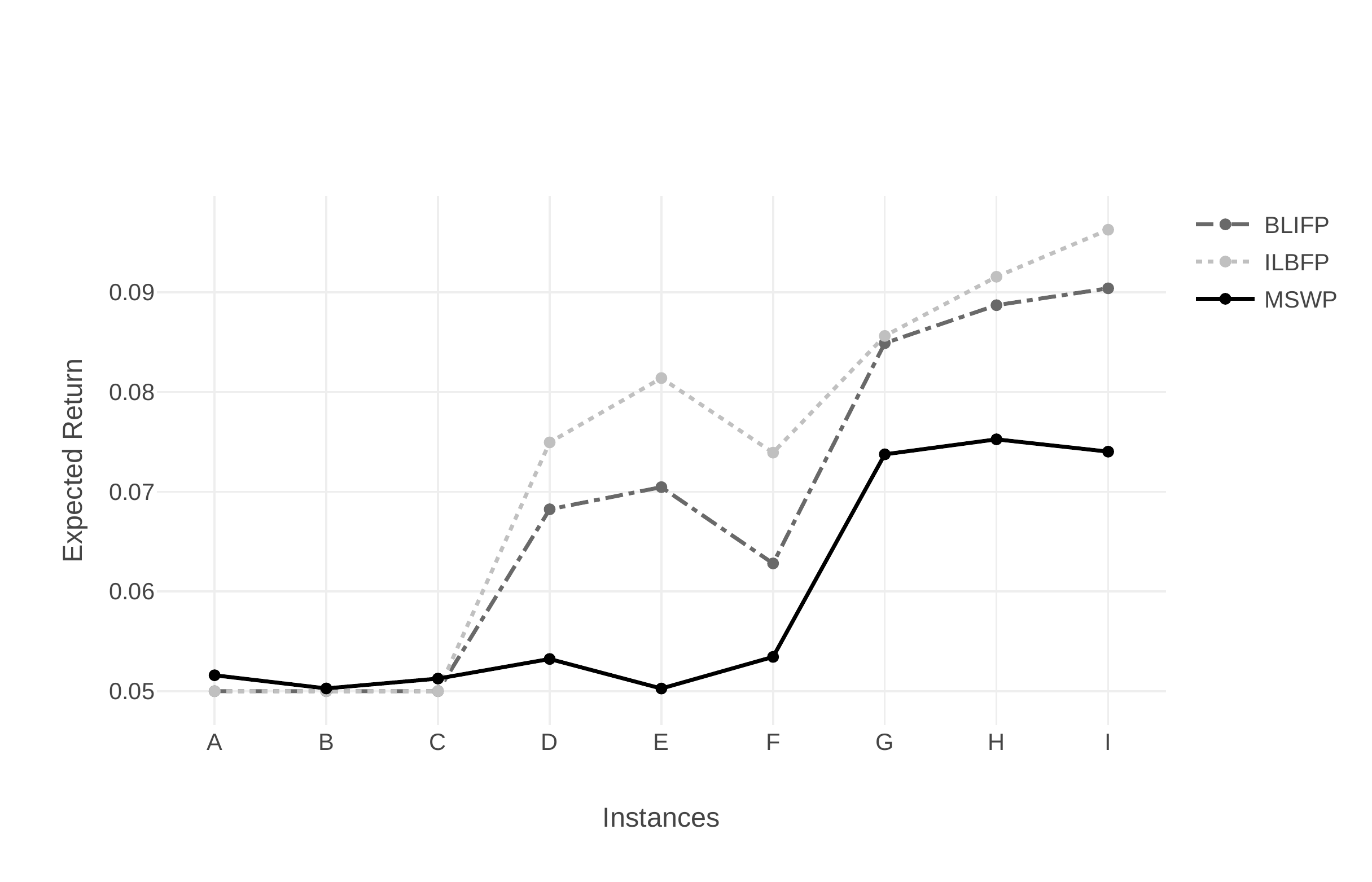}}

\end{multicols}
\caption{ {\footnotesize Expected return for {\textbf{BLIFP}},  {\textbf{ILBFP} and \textbf{MSWP}} for $\alpha=0.5$ and $\mu_0=0$ (left) and for $\alpha=0.5$ and $\mu_0=0.05$ (right)}}
\end{figure}

\begin{figure}[H]
\centering

\begin{multicols}{2}

\fbox{\includegraphics[scale=0.35]{E05_010.pdf}}

\fbox{\includegraphics[scale=0.35]{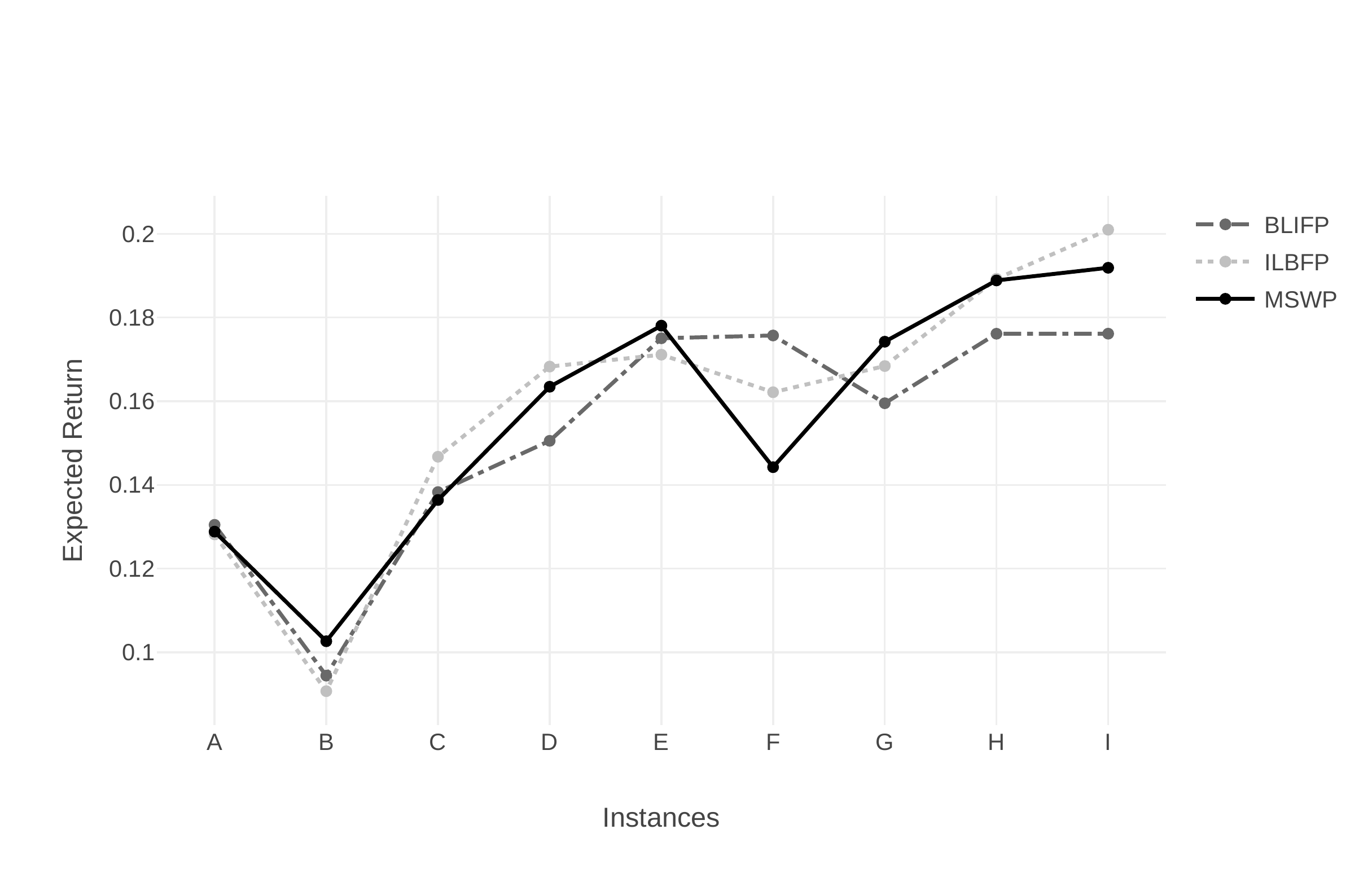}}

\end{multicols}
\caption{ {\footnotesize Expected return for {\textbf{BLIFP}},  {\textbf{ILBFP} and \textbf{MSWP}} for $\alpha=0.5$ and $\mu_0=0.1$ (left) and for $\alpha=0.9$ and $\mu_0=0$ (right)}}
\end{figure}

\begin{figure}[H]
\centering

\begin{multicols}{2}

\fbox{\includegraphics[scale=0.35]{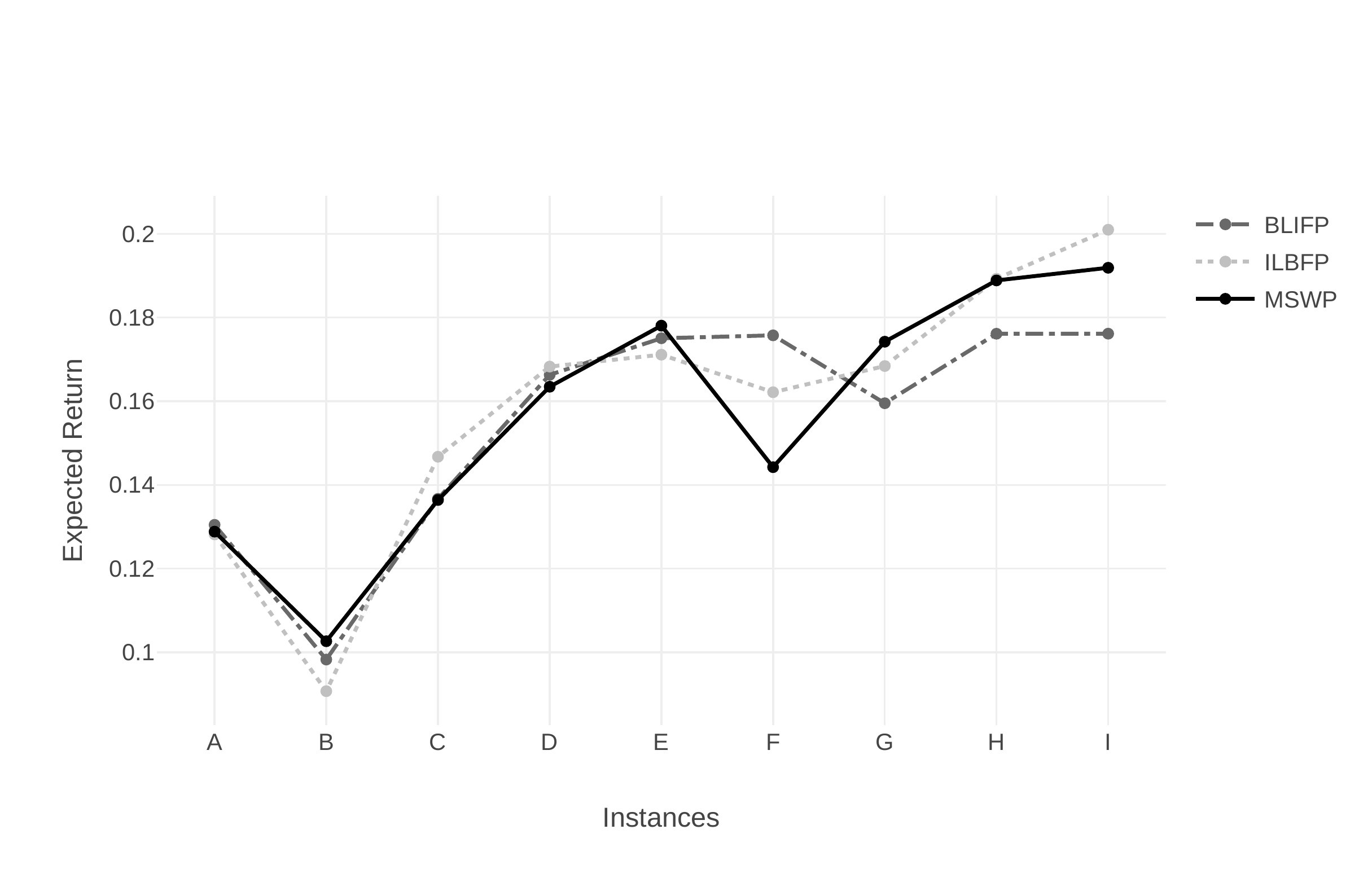}}

\fbox{\includegraphics[scale=0.35]{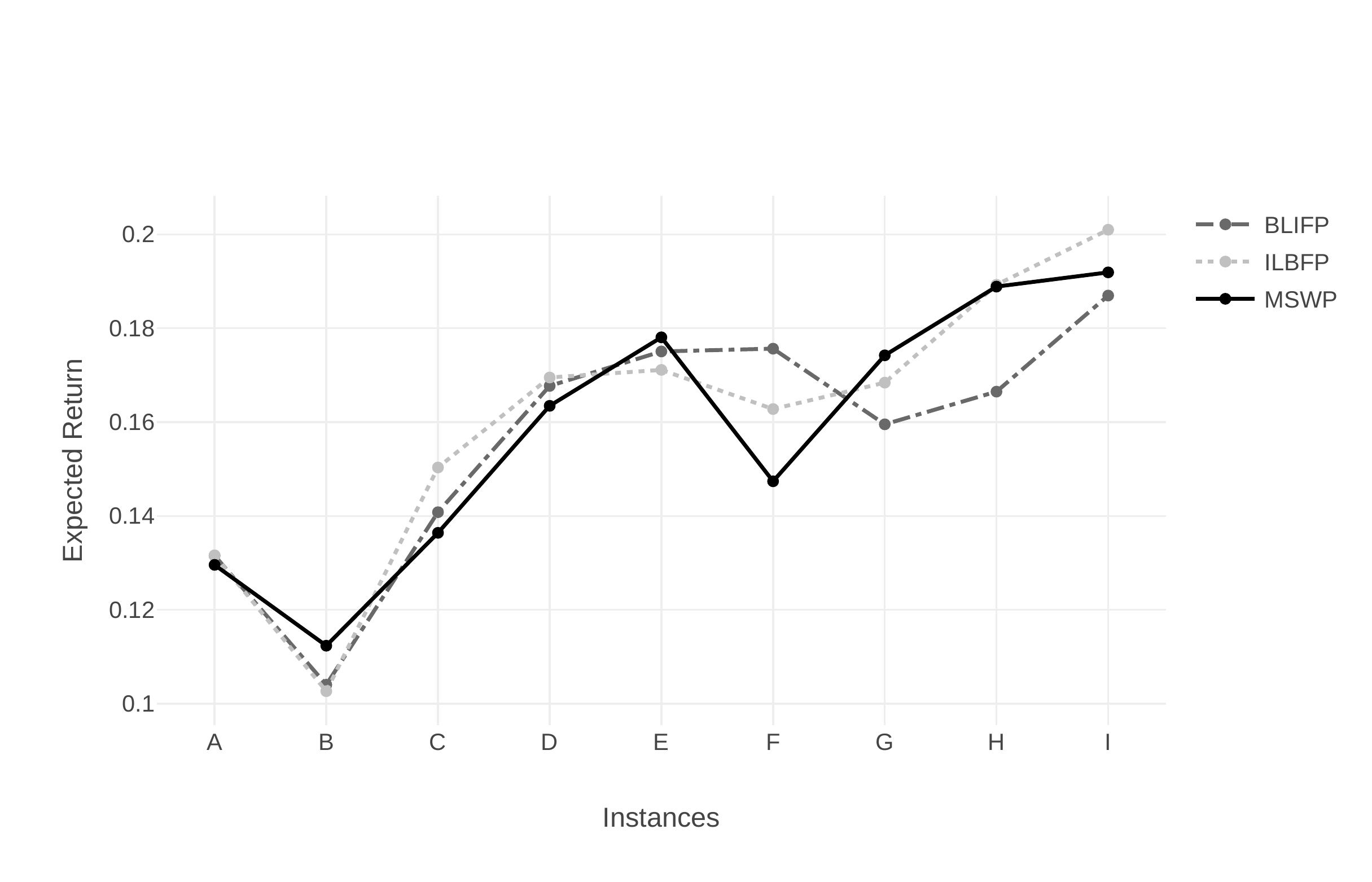}}

\end{multicols}
\caption{ {\footnotesize Expected return for {\textbf{BLIFP}},  {\textbf{ILBFP} and \textbf{MSWP}} for $\alpha=0.9$ and $\mu_0=0.05$ (left) and for $\alpha=0.$ and $\mu_0=0.1$ (right)}}
\end{figure}


\begin{thebibliography}{}

\bibitem{ang12} E. Angelelli, R. Mansini, and M.G. Speranza. \newblock A new heuristic framework for portfolio selection. \emph{Computational Optimization and Applications}, 51, 345-361. 2012.


\bibitem{bar13} J.F. Bard. \newblock Practical bilevel optimization: algorithms and applications. \newblock Springer Science \& Business Media. 2013.

\bibitem{bau}
T. Baumann and N. Trautmann. Portfolio-optimization models for small
investors. \emph{Mathematical Methods of Operations Research}, 77, 345-356. 2013.

\bibitem{bau10}
R. Baule. Optimal portfolio selection for the small investor considering risk
and transaction costs. \emph{OR Spectrum}, 32, 61--76. 2010.



\bibitem{ben03}
S. Benati. The optimal portfolio problem with coherent risk measure constraints. \emph{European Journal of Operational Research}, 150(3): 572--584. 2003.

\bibitem{ben14} S. Benati. Using medians in portfolio optimization. \emph{Journal of the  Operational  Research Society}, 66(5), 1--12, 720--731. 2014.

{
\bibitem{ben62} J. F. Benders. Partitioning procedures for solving mixed-variables programming problems. \emph{Numerische mathematik}, 4(1), 238-252. 1962.}

\bibitem{cas11}
F. Castro, J. Gago, I. Hartillo, J. Puerto, and J.M. Ucha. An algebraic approach to Integer Portfolio problems. \emph{European Journal of Operational Research} 210,  647--659. 2011.

\bibitem{col05} B. Colson, P. Marcotte, and G. Savard.  Bilevel programming: A survey. \emph{4or}, 3(2), 87-107. 2005.


\bibitem{dav90} M. H. Davis and A. R. Norman. Portfolio selection with transaction costs. \emph{Mathematics of operations research}, 15(4), 676-713. 1990.


\bibitem{kel00}
H. Kellerer, R. Mansini, and M.G. Speranza. Selecting portfolios with fixed
cost and minimum transaction lots. \emph{Annals of Operations Research}, 99, 287--304. 2000.

\bibitem{kre11} N. Krejic, M. Kumaresan, and A. Roznjik. \newblock VaR optimal portfolio with transaction costs. \emph{Applied Mathematics and Computation}, 218, 4626-4637. 2011.

\bibitem{kol14} P. N. Kolm, R. Tütüncü, and F.J. Fabozzi. 60 Years of portfolio optimization: Practical challenges and current trends. \emph{European Journal of Operational Research}, 234(2), 356-371. 2014.

\bibitem{kon05}
H. Konno, K. Akishino, and R. Yamamoto. Optimization of a long-short
portfolio under nonconvex transaction cost. \emph{Computational Optimization and Applications}, 32, 115--132. 2005.

\bibitem{kor98} R. Korn. Portfolio optimisation with strictly positive transaction costs and impulse control. \emph{Finance and Stochastics}, 2(2), 85-114. 1998.

\bibitem{lab16} M. Labbé and A. Violin. Bilevel programming and price setting problems. \emph{Annals of Operations Research}, 240(1), 141-169. 2016.

\bibitem{le09}
H. A. Le Thi, M. Moeini,  and T. Pham Dinh. DC programming approach for
portfolio optimization under step increasing transaction costs. \emph{Optimization: A Journal of Mathematical Programming and Operations Research}, 58, 267--289. 2009.


\bibitem{Liu02} H. Liu, M. Loewenstein. Optimal Portfolio Selection with Transaction Costs and Finite Horizons. \emph{ The Review of Financial Studies},  15(3), 805--835. 2002. 



\bibitem{lob07} M. S. Lobo, M. Fazel, and S. Boyd. Portfolio optimization with linear and fixed transaction costs. \emph{Annals of Operations Research}, 152(1), 341-365. 2007.

\bibitem{Lynch11} A.W. Lynch and S. Tan. Explaining the Magnitude of Liquidity Premia: The Roles of Return Predictability, Wealth Shocks, and State-Dependent Transaction Costs. \emph{The Journal of Finance}, 66: 1329-1368. 2011.

{\bibitem{mag76} M. J. Magill and G.M. Constantinides. Portfolio selection with transactions costs. \emph{Journal of economic theory}, 13(2), 245-263. 1976}

\bibitem{man03}
R.~Mansini, W.~Ogryczak, and M.G.~Speranza.
\newblock On LP Solvable Models for Portfolio Selection.
\newblock {\em Informatica}, Vol. 14, No. 1, 37--62. 2003.


\bibitem{man14}
R.~Mansini, W.~Ogryczak, and M.G.~Speranza. \newblock  Twenty years of linear programming based on portfolio optimization. \newblock {\em European Journal of Operational Research}. Vol. 234, Issue 2, 518-535. 2014.

\bibitem{man05}
R.~Mansini and M.G.~Speranza.
\newblock An exact approach for portfolio selection with transaction costs and dounds.
\newblock {\em IIE Transactions}, 37,  919-929. 2005.


\bibitem{man15}
R.~Mansini, W.~Ogryczak, and M.G.~Speranza.\newblock Portfolio Optimization and Transaction Costs, in Quantitative Financial Risk Management: Theory and Practice (eds C. Zopounidis and E. Galariotis), John Wiley, and Sons, Inc, Hoboken, NJ, USA. 2015.

{
\bibitem{man15_2} R.~Mansini, W.~Ogryczak, and M.G.~Speranza.  Linear and Mixed Integer Programming for Portfolio Optimization. Springer, Berlin. 2015.
}

\bibitem{mar52} H. M. Markowitz.  \newblock{Portfolio selection}. \emph{Journal of Finance}, 7, 77--91. 1952.

{\bibitem{mcc76} G. P. McCormick. Computability of global solutions to factorable nonconvex programs: Part I—Convex underestimating problems. \emph{Mathematical programming}, 10(1), 147-175. 1976.}

{\bibitem{mor95} A. J. Morton and Pliska, S. R. Optimal portfolio management with fixed transaction costs. \emph{Mathematical Finance}, 5(4), 337-356. 1995.}

\bibitem{oli18} A. V. Olivares-Nadal and V. DeMiguel. \newblock{A Robust Perspective on Transaction Costs in Portfolio Optimization}. \newblock {\emph{Operations Research}, 66(3):733-739. 2018.
}
W.~Ogryczak, and T.~Sliwinski.
\newblock On efficient optimisation of the CVaR and related LP computable risk measures for portfolio selection.
\newblock {\em Mathematical and Statistical Methods for Actuarial Sciences and Finance}. Springer, Milano. 2010.

\bibitem{prt17} J.~Puerto, A.M.~Rodr{\'\i}guez-Ch{\'\i}a, and A.~Tamir, \newblock Revisiting $k$-sum  Optimization, \newblock{\em Mathematical Programming} 165, 579--604. 2017.

\bibitem{roc00} R.T. Rockafellar and S. Uryasev.
\newblock Optimization of conditional value-at-risk. \newblock \emph{Journal of Risk}, 21-41. 2000.

\bibitem{sin17} A. Sinha, M. Pekka, and D. Kalyanmoy. \newblock{ A review on bilevel optimization: From classical to evolutionary approaches and applications.} \newblock IEEE Transactions on Evolutionary Computation 22.2: 276-295. 2017.

\bibitem{val14} C.A.Valle, N.Meade, and J.E.Beasley.
\newblock{Absolute return portfolios}. \newblock{Omega}, 45, 20-41. 2014.

\bibitem{woo13}
M. Woodside-Oriakhi, C. Lucas, and J.E. Beasley. Portfolio rebalancing with an investment horizon and transaction costs. \emph{Omega}, 41(2), 406--420. 2013.

\end{thebibliography}
\end{document}